\documentclass[11pt]{amsart}   	
\usepackage{geometry}             
\usepackage[dvipsnames]{xcolor}   		             		
\usepackage{graphicx}				
\usepackage{mathrsfs}
\usepackage{amssymb,amsthm,amsmath,mathtools}
\usepackage{tikz}
\usepackage{tikz-cd}
\usepackage{adjustbox}
\usepackage{accents,upgreek,enumerate}
\usepackage[headings]{fullpage}
\usepackage{bm}
\usepackage[all]{xy}
\usepackage{caption}
\usepackage{floatrow}       
\usepackage{csquotes}
\usepackage{enumitem}
\usepackage{comment}
\usepackage{verbatim}
\usepackage{hyperref}

\newtheorem{theorem}{Theorem}[section]

\newtheorem{lemma}[theorem]{Lemma}
\newtheorem{proposition}[theorem]{Proposition}

\newtheorem{quasi-theorem}[theorem]{Quasi-Theorem}

\theoremstyle{definition}
\newtheorem{definition}[theorem]{Definition}

\newtheorem{remark}[theorem]{Remark}

\newcommand{\direction}{\vec{\mathbf{v}}} 

\newcommand{\quotient}[2]{\raisebox{0.5ex}{$#1$}/\raisebox{-0.5ex}{$#2$}}

\newcommand{\bigmid}{\mathrel{\big|}}

\newcommand{\A}{{\mathbb{A}}}                     
\newcommand{\EE}{\mathbf{E}}       
\newcommand{\FF} {{\mathbb F}}	
\newcommand{\NN} {{\mathbb N}}		
\newcommand{\PP}{\mathbb{P}}         
\newcommand{\QQ} {{\mathbb Q}}	

\newcommand{\RR} {{\mathbb R}}	
\newcommand{\CC} {{\mathbb C}}
\newcommand{\ZZ} {{\mathbb Z}}	
\renewcommand{\AA}{\mathbb{A}}	

\newcommand{\HM}[1]{{\color{blue} \sf H:~[#1]}}
\newcommand{\AZ}[1]{{\color{red} \sf A:~[#1]}}

\newcommand{\Puiseux}[2][t]{#2\{\!\{#1\}\!\}}

\DeclareMathOperator{\Hom}{Hom}
\DeclareMathOperator{\Ker}{Ker}
\DeclareMathOperator{\image}{Im}
\DeclareMathOperator{\End}{End}
\DeclareMathOperator{\Coker}{Coker}
\DeclareMathOperator{\Ext}{Ext}
\DeclareMathOperator{\Aut}{Aut}

\DeclareMathOperator{\ini}{in}
\DeclareMathOperator{\chara}{char} 
\DeclareMathOperator{\codim}{codim} 
\DeclareMathOperator{\conv}{conv}
\DeclareMathOperator{\pr}{pr} 
\DeclareMathOperator{\stab}{stab}
\DeclareMathOperator{\ext}{\cE \it{xt}} 
\DeclareMathOperator{\diag}{diag} 
\DeclareMathOperator{\Area}{Area}
\DeclareMathOperator{\Span}{span}
\DeclareMathOperator{\id}{id} 
\DeclareMathOperator{\val}{val}
\DeclareMathOperator{\ev}{ev}
\DeclareMathOperator{\floor}{floor}
\DeclareMathOperator{\sign}{sign} 
\DeclareMathOperator{\res}{res}
\DeclareMathOperator{\Br}{Br}
\DeclareMathOperator{\Stab}{Stab}
\DeclareMathOperator{\rank}{rank}
\DeclareMathOperator{\spec}{Spec}
\DeclareMathOperator{\Symm}{Symm}
\DeclareMathOperator{\inv}{{inv}}

\newcommand{\M}{\overline{{M}}}
\newcommand{\smargin}[1]{\marginpar{\tiny{#1}}}

\def\VZ{\mathcal{V\!Z}}
\def\VQ{\mathcal{V\!Q}}

\newcommand{\ba}{{\mathbf{a}}}
\newcommand{\bb}{{\mathbf{b}}}
\newcommand{\bc}{{\mathbf{c}}}
\newcommand{\bi}{{\mathbf{i}}}
\newcommand{\bj}{{\mathbf{j}}}
\newcommand{\kk}{\mathbf{k}}
\newcommand{\bL}{\mathbf{L}}
\newcommand{\bT}{\mathbf{T}}
\newcommand{\bk}{\mathbf{k}}
\newcommand{\bn}{\mathbf{n}}
\newcommand{\bp}{\mathbf{p}}
\newcommand{\bq}{\mathbf{q}}
\newcommand{\bs}{\mathbf{s}}
\newcommand{\bt}{\mathbf{t}}
\newcommand{\bw}{\mathbf{w}}
\newcommand{\bx}{\mathbf{x}}

\newcommand{\cal}{\mathcal}

\def\cA{{\cal A}}
\def\cB{{\cal B}}
\def\cC{{\cal C}}
\def\cD{{\cal D}}
\def\cE{{\cal E}}
\def\cF{{\cal F}}
\def\cH{{\cal H}}
\def\cK{{\cal K}}
\def\cL{{\cal L}}
\def\cM{{\cal M}}
\def\cN{{\cal N}}
\def\cO{{\cal O}}
\def\cP{{\cal P}}
\def\cQ{{\cal Q}}
\def\cR{{\cal R}}
\def\cT{{\cal T}}
\def\cU{{\cal U}}
\def\cV{{\cal V}}
\def\cW{{\cal W}}
\def\cX{{\cal X}}
\def\cY{{\cal Y}}
\def\cZ{{\cal Z}}

\def\fB{\mathfrak{B}}
\def\fC{\mathfrak{C}}
\def\fD{\mathfrak{D}}
\def\fE{\mathfrak{E}}
\def\fF{\mathfrak{F}}
\def\fM{\mathfrak{M}}
\def\fS{\mathfrak{S}}
\def\fT{\mathfrak{T}}
\def\fV{\mathfrak{V}}
\def\fX{\mathfrak{X}}
\def\fY{\mathfrak{Y}}

\def\ff{\mathfrak{f}}
\def\fm{\mathfrak{m}}
\def\fp{\mathfrak{p}}
\def\ft{\mathfrak{t}}
\def\fu{\mathfrac{u}}
\def\fv{\mathfrak{v}}
\def\frev{\mathfrak{rev}}

\newcommand{\tPhi }{\tilde{\Phi} }
\newcommand{\tGa}{\tilde{\Gamma}}
\newcommand{\tbeta}{\tilde{\beta}}
\newcommand{\trho }{\tilde{\rho} }
\newcommand{\tpi  }{\tilde{\pi}  }

\newcommand{\tC}{\tilde{C}}
\newcommand{\tD}{\tilde{D}}
\newcommand{\tE}{\tilde{E}}
\newcommand{\tF}{\tilde{F}}
\newcommand{\tK}{\tilde{K}}
\newcommand{\tL}{\tilde{L}}
\newcommand{\tT}{\tilde{T}}
\newcommand{\tY}{\tilde{Y}}

\newcommand{\plC}{\scalebox{0.8}[1.3]{$\sqsubset$}}

\newcommand{\tf}{\tilde{f}}
\newcommand{\tih}{\tilde{h}}
\newcommand{\tz}{\tilde{z}}
\newcommand{\tu}{\tilde{u}}

\def\tilcW{{\tilde\cW}}
\def\tilcM{\tilde\cM}
\def\tilcC{\tilde\cC}
\def\tilcZ{\tilde\cZ}

\newcommand{\Mbar}{\overline{\cM}\vphantom{\cM}}
\newcommand{\Hbar}{\overline{\cH}}
\newcommand{\MX}{\Mbar^\bu_\chi(\Gamma,\vec{d},\vmu)}
\newcommand{\GX}{G^\bu_\chi(\Gamma,\vec{d},\vmu)}
\newcommand{\Mi}{\Mbar^\bu_{\chi^i}(\Po,\nu^i,\mu^i)}
\newcommand{\Mv}{\Mbar^\bu_{\chi^v}(\Po,\nu^v,\mu^v)}
\newcommand{\GYfm}{G^\bu_{\chi,\vmu}(\Gamma)}
\newcommand{\Gdmu}{G^\bu_\chi(\Gamma,\vd,\vmu)}
\newcommand{\Mdmu}{\cM^\bu_\chi(\hatYrel,\vd,\vmu)}
\newcommand{\tMd}{\tilde{\cM}^\bu_\chi(\hatYrel,\vd,\vmu)}
\newcommand{\Lsm}{\mathcal{LSM}}
\newcommand{\Tsm}{{TSM}}

\def\trop{\mathrm{tr}}
\def\mult{\mathrm{mult}}
\def\an{\mathrm{an}}
\def\di{\mathrm{div}}
\def\sigbar{\overline \Delta}
\def\Sbar{\overline{\mathfrak{S}}}
\newcommand{\Spec}{\operatorname{Spec}}
\newcommand{\tr}{\operatorname{tr}}

\def\vertexsize {1.2pt}   
\newcommand{\vertex}[2][1]{\fill (#2) circle [radius = #1 * \vertexsize];}
\newcommand{\vertexblue}[2][1]{\fill [blue] (#2) circle [radius = 1* \vertexsize];}

\title{Trigonal and embedded tropical curves of low genus}
\author {Hannah Markwig}
\address {Universit\"at T\"ubingen, Fachbereich Mathematik, Auf der Morgenstelle 10, 72076 T\"ubingen, Germany }
\email {hannah@math.uni-tuebingen.de}
\author {Angelina Zheng}
\address {Universit\"at T\"ubingen, Fachbereich Mathematik, Auf der Morgenstelle 10, 72076 T\"ubingen, Germany }
\email {zheng@math.uni-tuebingen.de}

\subjclass[2020]{14T15, 14T20}
\keywords{Tropical curves, harmonic morphisms of graphs, trigonality, plane tropical curves and dual Newton subdivisions}

\begin{document}

\begin{abstract}
   In algebraic geometry, trigonal curves can always be embedded into Hirzebruch surfaces. In tropical geometry, the notion of trigonality does not have a unique translation. We focus on the characterization in terms of the existence of a degree $3$ morphism to a line, and discuss relations to possible embeddings into $\mathbb{R}^2$ reflecting an embedding into a Hirzebruch surface. Our results can be divided into three parts: for tropical curves of low genus $3$ and $4$, we discuss the relation between a trigonal morphism and an embedding dual to the polygon of a Hirzebruch surface, building on works on embeddings of hyperelliptic tropical curves and curves of low genus \cite{Mor21, BJMS15}. We compare obstructions for embeddings with obstructions for the existence of a degree $3$ morphism to a line. Finally, we showcase examples where a non-smooth embedding can be unfolded to reflect certain features of a degree $3$ morphism to a line.
\end{abstract}

\maketitle

\section{Introduction}
In algebraic geometry, \emph{trigonal curves} form a well-understood class.
They are defined as smooth projective non-hyperelliptic curves admitting a $g^1_3$, or equivalently, a degree 3 morphism to $\mathbb{P}^1$. 
Classical results due to Noether, Enriques, Babbage, and Petri \cite{ACGH}, as well as subsequent work by Miranda \cite{Mir} and Casnati--Ekedahl \cite{CE}, provide a detailed description of the canonical model of a trigonal curve. In particular, the canonical image of such a curve lies on a \emph{rational normal scroll}, that is, the image of a rational ruled surface over $\mathbb{P}^1$ embedded in projective space. This surface is called a \emph{Hirzebruch surface} $\mathbb{F}_n := \mathbb{P}(\mathcal{O}_{\mathbb{P}^1} \oplus \mathcal{O}_{\mathbb{P}^1}(n))$, where the integer $n$, known as the \emph{Maroni invariant}, satisfies
\begin{equation}\label{eq:Maroni}
0 \leq n \leq \left\lfloor \frac{g+2}{3} \right\rfloor \quad \text{and} \quad g \equiv n \mod 2.
\end{equation}

More precisely, if $C \subset \mathbb{F}_n$ is a trigonal curve of genus $g$ and Maroni invariant $n$, then its divisor class is
\[
[C] \sim 3E_n + \frac{g + 3n + 2}{2} F_n,
\]
where $E_n$ denotes the class of the unique curve with self-intersection $-n$ ( when $n > 0$), and $F_n$ is the class of a line of the ruling. The projection $\mathbb{F}_n\rightarrow \mathbb{P}^1$ restricted to $C$ is a degree $3$ morphism.
When viewing $\mathbb{F}_n$ as a toric surface defined by a polygon, and the class of $C$ as the class of a hyperplane section, then the polygon to consider is as depicted in Figure \ref{fig-polygon}.

\begin{figure}[h]
    \centering
\tikzset{every picture/.style={line width=0.75pt}} 
\tikzset{every picture/.style={line width=0.75pt}} 

\begin{tikzpicture}[x=0.75pt,y=0.75pt,yscale=-1,xscale=1]

\draw    (200,20) -- (200,130) ;
\draw    (230,130) -- (200,130) ;
\draw    (200,20) -- (230,80) ;
\draw    (230,80) -- (230,130) ;

\draw (103,8) node [anchor=north west][inner sep=0.75pt]  [font=\scriptsize] [align=left] {$\displaystyle \left( 0,\frac{g+3n+2}{2}\right)$};
\draw (151,122) node [anchor=north west][inner sep=0.75pt]  [font=\scriptsize] [align=left] {$\displaystyle ( 0,0)$};
\draw (243,122) node [anchor=north west][inner sep=0.75pt]  [font=\scriptsize] [align=left] {$\displaystyle ( 3,0)$};

\draw (243,70) node [anchor=north west][inner sep=0.75pt]  [font=\scriptsize] [align=left] {$\displaystyle \left( 3,\frac{g-3n+2}{2}\right)$};

\end{tikzpicture}
    \caption{The polygon of a curve of class $3E_n + \frac{g + 3n + 2}{2} F_n$ in $\mathbb{F}_n$.}
    \label{fig-polygon}
\end{figure}
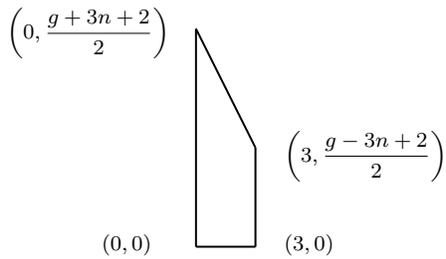

In particular, from \eqref{eq:Maroni}, a trigonal curve of genus $3$ has always Maroni invariant $1.$ Indeed, a genus $3$ trigonal curves $C$ is a plane quartic and each $g_3^1$ is obtained by fixing a point $p\in C$ and considering all lines passing through $p$. The blow-up of the plane at a point $p$ is isomorphic to the first Hirzebruch surface: $\mathbb F_1\cong Bl_{\left[0,0,1\right]}\mathbb P(1,1,1)$; $\mathbb P^2=\mathbb P(1,1,1)$.

For genus $4$ curves, again from \eqref{eq:Maroni}, the possible Maroni invariants are $0,2$. Indeed, a trigonal curve of genus $4$ is canonically embedded in $\mathbb P^3$ as the complete intersection of a non-singular cubic and a quadric, which may or may not be singular.
If the quadric is non-singular, one obtains naturally the embedding in the degree $0$ Hirzebruch surface since $\mathbb F_0\cong \mathbb P^1\times \mathbb P^1.$ If the quadric is singular, instead, the blow up of the vertex of the cone, which can be identified with the weighted projective plane $\mathbb P(1,1,2)$, is isomorphic to the second Hirzebruch surface via isomorphism $\mathbb F_n\cong Bl_{\left[0,0,1\right]}\mathbb P(1,1,n)$; $n\geq 1$.

Vice versa, in general a curve of class $3E_n + k F_n$ in $\mathbb{F}_n$ is trigonal, because the projection morphism induces a degree $3$ cover of $\mathbb{P}^1$. 

In contrast to the algebraic setting, the notion of gonality for abstract tropical curves is not uniquely defined, and several inequivalent definitions have been introduced. Most notably, there are (at least) two distinct notions of gonality for graphs: the so-called \emph{gonality}, defined via harmonic morphisms, and the \emph{divisorial gonality}, defined via linear systems on graphs.

A metric graph is said to be \emph{$d$-gonal} if it admits a non-degenerate harmonic morphism of degree $d$ to a tree (possibly after a tropical modification). On the other hand, it is \emph{divisorially $d$-gonal} if it carries a divisor of degree $d$ and rank at least one. These two definitions are known to coincide when 
$d=2$ \cite{Chan}, corresponding to the tropical analogue of hyperelliptic curves. Their relationship for 
$d=3$, i.e. in the trigonal case, has been investigated in detail in \cite{MZ1, MZ2}.

The existence of a non-degenerate harmonic morphism of degree $d$ from a tropical curve (or from a tropical modification of it) to a tree always implies the existence of a divisor of degree $d$ and rank at least one. In particular, every $d$-gonal tropical curve is necessarily divisorially $d$-gonal. However, the converse does not hold in general.

For the case $d=3$, the two notions coincide under certain additional assumptions, specifically when the underlying graph is 3-edge-connected. More generally, if one drops the edge-connectivity assumption the two notions are again equivalent if the graph has no separating vertices.

Furthermore, in the 3-edge-connected case, if a tropical curve is trigonal and equipped with a non-degenerate harmonic morphism of degree $3$ to a tree, then this morphism defines a \emph{tropical admissible cover} of degree $3$ in the sense of \cite{CMR}. The same relation continues to hold for graphs with no assumption on the edge-connectivity, if one restricts to graphs with no separating vertices or multiple edges.
Naturally, tropical admissible covers are harmonic morphisms of the same degree, while in general the converse is not true.

Of course, we can use \emph{tropicalization} to transport concepts from algebraic geometry to tropical geometry, but this is intricate and can lead to different results in tropical geometry when considering notions that are equivalent in the world of algebraic geometry. This is, roughly, why we observe different notions of gonality for abstract tropical curves.

Given a curve defined over a field with a non-Archimedean valuation, the theory of semistable reduction produces a tropicalization which can be viewed as a skeleton of the Berkovich analytification of the curve, or, equivalently, as a metrization of the dual graph of the special fiber of the semistable reduction, where the length of an edge reflects ``how fast'' the dual node forms in the family, i.e.\ what the valuation of $h$ in a local equation $xy=h$ of the forming node is. Given a cover of curves, one can tropicalize both source and target and obtain a harmonic morphism of tropical curves, which in addition satisfies the Riemann-Hurwitz condition at every vertex (i.e.\ a tropical cover), and such that the local Hurwitz numbers which can be associated to each vertex are nonzero. In fact, this is also the realizability condition for a harmonic morphism of tropical curves: if it satisfies the Riemann-Hurwitz condition (i.e.\ if it is a tropical cover) and the local Hurwitz numbers are all nonzero, then it is the tropicalization of a cover of curves \cite{ABBR15, ABBR15b}. We therefore call such a tropical cover a \emph{realizable cover}.
As observed in \cite[Section 2.2]{Cap13} degree $3$ tropical covers are all realizable. Then, all tropical covers that we will consider are also realizable.

Assume we have a curve $C$ which is embedded in a toric surface, e.g.\ $\mathbb{F}_n$. Restricting to the torus points of the curve and applying (minus) the valuation map coordinatewise, we obtain an \emph{embedded tropicalization}. There is always a map from the abstract tropicalization of $C$ to the tropicalization of the torus, which equals $\mathbb{R}^2$. But in general this map can be bad in the sense that it may contract important features of the abstract tropical curve such as cycles. Only if the embedded tropicalization is \emph{smooth}, we know that its underlying metric graph structure actually reflects the abstract tropicalization \cite{BPR11a}.

\begin{figure}[h]
    \centering
\tikzset{every picture/.style={line width=0.75pt}} 

\begin{tikzpicture}[x=0.75pt,y=0.75pt,yscale=-1,xscale=1]

\draw    (184.67,80) -- (311.67,68.81) ;
\draw [shift={(313.67,68.63)}, rotate = 174.96] [color={rgb, 255:red, 0; green, 0; blue, 0 }  ][line width=0.75]    (10.93,-3.29) .. controls (6.95,-1.4) and (3.31,-0.3) .. (0,0) .. controls (3.31,0.3) and (6.95,1.4) .. (10.93,3.29)   ;
\draw    (170,100) -- (170,168) ;
\draw [shift={(170,170)}, rotate = 270] [color={rgb, 255:red, 0; green, 0; blue, 0 }  ][line width=0.75]    (10.93,-3.29) .. controls (6.95,-1.4) and (3.31,-0.3) .. (0,0) .. controls (3.31,0.3) and (6.95,1.4) .. (10.93,3.29)   ;
\draw    (322.13,116.73) -- (321.36,172.67) ;
\draw [shift={(321.33,174.67)}, rotate = 270.79] [color={rgb, 255:red, 0; green, 0; blue, 0 }  ][line width=0.75]    (10.93,-3.29) .. controls (6.95,-1.4) and (3.31,-0.3) .. (0,0) .. controls (3.31,0.3) and (6.95,1.4) .. (10.93,3.29)   ;
\draw    (177.6,190) -- (270.83,205.73) ;
\draw [shift={(272.8,206.07)}, rotate = 189.58] [color={rgb, 255:red, 0; green, 0; blue, 0 }  ][line width=0.75]    (10.93,-3.29) .. controls (6.95,-1.4) and (3.31,-0.3) .. (0,0) .. controls (3.31,0.3) and (6.95,1.4) .. (10.93,3.29)   ;
\draw    (180,70) -- (241.66,18.55) ;
\draw [shift={(243.2,17.27)}, rotate = 140.16] [color={rgb, 255:red, 0; green, 0; blue, 0 }  ][line width=0.75]    (10.93,-3.29) .. controls (6.95,-1.4) and (3.31,-0.3) .. (0,0) .. controls (3.31,0.3) and (6.95,1.4) .. (10.93,3.29)   ;
\draw    (273.2,26.87) -- (314.86,61.19) ;
\draw [shift={(316.4,62.47)}, rotate = 219.49] [color={rgb, 255:red, 0; green, 0; blue, 0 }  ][line width=0.75]    (10.93,-3.29) .. controls (6.95,-1.4) and (3.31,-0.3) .. (0,0) .. controls (3.31,0.3) and (6.95,1.4) .. (10.93,3.29)   ;
\draw    (256.4,69.2) -- (256.4,95.67) -- (256.4,114.87) ;
\draw [shift={(256.4,116.87)}, rotate = 270] [color={rgb, 255:red, 0; green, 0; blue, 0 }  ][line width=0.75]    (10.93,-3.29) .. controls (6.95,-1.4) and (3.31,-0.3) .. (0,0) .. controls (3.31,0.3) and (6.95,1.4) .. (10.93,3.29)   ;
\draw    (180,180) -- (242.8,132.87) ;
\draw [shift={(244.4,131.67)}, rotate = 143.11] [color={rgb, 255:red, 0; green, 0; blue, 0 }  ][line width=0.75]    (10.93,-3.29) .. controls (6.95,-1.4) and (3.31,-0.3) .. (0,0) .. controls (3.31,0.3) and (6.95,1.4) .. (10.93,3.29)   ;
\draw    (272,137.67) -- (308.66,178.51) ;
\draw [shift={(310,180)}, rotate = 228.09] [color={rgb, 255:red, 0; green, 0; blue, 0 }  ][line width=0.75]    (10.93,-3.29) .. controls (6.95,-1.4) and (3.31,-0.3) .. (0,0) .. controls (3.31,0.3) and (6.95,1.4) .. (10.93,3.29)   ;
\draw    (296,207.6) -- (307.6,198.87) ;
\draw [shift={(309.2,197.67)}, rotate = 143.04] [color={rgb, 255:red, 0; green, 0; blue, 0 }  ][line width=0.75]    (10.93,-3.29) .. controls (6.95,-1.4) and (3.31,-0.3) .. (0,0) .. controls (3.31,0.3) and (6.95,1.4) .. (10.93,3.29)   ;
\draw    (268,58.07) -- (309.66,92.39) ;
\draw [shift={(311.2,93.67)}, rotate = 219.49] [color={rgb, 255:red, 0; green, 0; blue, 0 }  ][line width=0.75]    (10.93,-3.29) .. controls (6.95,-1.4) and (3.31,-0.3) .. (0,0) .. controls (3.31,0.3) and (6.95,1.4) .. (10.93,3.29)   ;
\draw    (340,68.63) .. controls (369.52,70.13) and (375.61,214.18) .. (303.76,213) ;
\draw [shift={(302.67,212.97)}, rotate = 2.09] [color={rgb, 255:red, 0; green, 0; blue, 0 }  ][line width=0.75]    (10.93,-3.29) .. controls (6.95,-1.4) and (3.31,-0.3) .. (0,0) .. controls (3.31,0.3) and (6.95,1.4) .. (10.93,3.29)   ;

\draw (161,72) node [anchor=north west][inner sep=0.75pt]   [align=left] {$\displaystyle C$};
\draw (317.87,56.53) node [anchor=north west][inner sep=0.75pt]   [align=left] {$\displaystyle \mathbb{P}^{1}$};
\draw (172,112) node [anchor=north west][inner sep=0.75pt]   [align=left] {trop};
\draw (329.6,129) node [anchor=north west][inner sep=0.75pt]   [align=left] {trop};
\draw (161,181) node [anchor=north west][inner sep=0.75pt]   [align=left] {$\displaystyle \Gamma $};
\draw (311,182) node [anchor=north west][inner sep=0.75pt]   [align=left] {$\displaystyle \mathbb{R}$};
\draw (248.4,5) node [anchor=north west][inner sep=0.75pt]   [align=left] {$\displaystyle \mathbb{F}_{n}$};
\draw (260.8,83.6) node [anchor=north west][inner sep=0.75pt]   [align=left] {trop};
\draw (248.4,121) node [anchor=north west][inner sep=0.75pt]   [align=left] {$\displaystyle \mathbb{R}^{2}$};
\draw (251.2,45.2) node [anchor=north west][inner sep=0.75pt]   [align=left] {$\displaystyle T^{2}$};
\draw (233.8,132.8) node [anchor=north west][inner sep=0.75pt]   [align=left] {\textsuperscript{}};
\draw (246.25,28.77) node [anchor=north west][inner sep=0.75pt]  [rotate=-358.96] [align=left] {$\displaystyle \cup $};
\draw (288.6,22.47) node [anchor=north west][inner sep=0.75pt]   [align=left] {pr};
\draw (287,137.07) node [anchor=north west][inner sep=0.75pt]   [align=left] {pr};
\draw (278.2,197.8) node [anchor=north west][inner sep=0.75pt]   [align=left] {$\displaystyle \Gamma '$};
\draw (221.93,57) node [anchor=north west][inner sep=0.75pt]   [align=left] {$\displaystyle f$};
\draw (212.2,175.8) node [anchor=north west][inner sep=0.75pt]   [align=left] {$\displaystyle f^{\rm trop}$};
\draw (318.2,90.87) node [anchor=north west][inner sep=0.75pt]   [align=left] {$\displaystyle T$};
\draw (317.51,74.71) node [anchor=north west][inner sep=0.75pt]  [rotate=-358.96] [align=left] {$\displaystyle \cup $};
\draw (281,49.4) node [anchor=north west][inner sep=0.75pt]   [align=left] {pr};
\draw (372.6,103.33) node [anchor=north west][inner sep=0.75pt]   [align=left] {trop};

\end{tikzpicture}

    \caption{Tropicalization of a curve with a morphism to $\mathbb{P}^1$ which factors through the projection of a Hirzebruch surface to $\mathbb{P}^1$.}
    \label{fig:sketchtropicalization}
\end{figure}
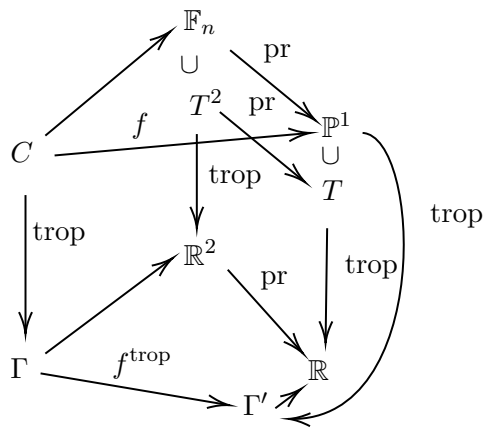

Figure \ref{fig:sketchtropicalization} shows the relation of tropicalization for a curve with a morphism to $\mathbb{P}^1$ which factors through the projection of the Hirzebruch surface to $\mathbb{P}^1$. Since we restrict our tropicalization for $\mathbb{F}_n$ to the torus points, the projection yields a tropical cover from the tropicalization of $C$ to $\mathbb{R}$. If we directly tropicalize the cover $C\rightarrow \mathbb{P}^1$ abstractly, we would not necessarily expect $\mathbb{R}$ for the target but rather any metric tree $\Gamma'$, which can be contracted to $\mathbb{R}$.
Since in this paper, we focus in situations in which the tropical version of this diagram is most meaningful, we accordingly focus on realizable covers to $\mathbb{R}$. 

As mentioned above, we expect the nicest relation in the tropical version of the diagram when the embedded tropicalization of $C$ is smooth, because then the map sending $\Gamma$ to this plane tropical curve does not contract essential features. We build on \cite{BJMS15}, which studied which abstract tropical curves of low genus can be smoothly embedded as plane tropical curves.

Our study can be viewed as an extension of \cite{Mor21}, which studies embeddings of tropical hyperelliptic curves into the plane.

Our study starts with the following observation, which we prove in Section \ref{sec-proj}:

\begin{proposition}[See \ref{lm:realizability_cover} ]\label{prop-proj}
Assume $\Gamma$ can be embedded smoothly into $\mathbb{R}^2$, i.e.\ such that (a modification of) $\Gamma$ is dual to a unimodular triangulation of the polygon in Figure \ref{fig-polygon}. Then the projection restricted to $\Gamma$ is a realizable
well-contracted tropical cover of degree $3$.
    \end{proposition}
Here, the term well-contracted (see Definition \ref{def-wellcontraced}) roughly means that no important features, in particular no cycles, can be contracted by the cover.

Vice versa, given a realizable tropical cover $\Gamma\rightarrow \mathbb{R}$ of degree $3$, then $\Gamma$ can be mapped to $\mathbb{R}^2$ such that its image is dual to the polygon in Figure \ref{fig-polygon}. This holds true by relying on the algebraic geometry side: $\Gamma\rightarrow \mathbb{R}$ can be lifted to a trigonal morphism $C\rightarrow \mathbb{P}^1$. Since $C$ is trigonal, it can be embedded into $\mathbb{F}_n$. Then we can use embedded tropicalization.
However, the map from $\Gamma\rightarrow \mathbb{R}^2$ which we obtain like this can be very bad in the sense that it can contract important features, since the embedded tropicalization cannot be expected to be smooth.

Our results are concerned with the converse of Proposition \ref{prop-proj} and are centered around the question to what extent contractions as mentioned above can be controlled.
Our results can be divided into three parts.

\begin{enumerate}
    \item We focus on a study of tropical curves of low genus $3,4$ and show that for tropical curves of maximal combinatorial type, the possible contractions are not bad, in the sense that we remain within the same combinatorial type. In this study of embeddings of curves of genus $3$ and $4$ into $\mathbb{R}^2$, we build on \cite{BJMS15}, which initiated the study of planar embeddings of abstract tropical curves. We also study a tropical version of the Maroni invariant for genus $3$ and $4$. The results in this context are presented in Section \ref{sec-maroni}.
\item We study obstructions for smooth embeddings of $\Gamma$ into $\mathbb{R}^2$ dual to the polygon in Figure \ref{fig-polygon} and find that such obstructions arise for combinatorial types which do not allow a realizable harmonic degree $3$ morphism to $\mathbb{R}$, see Section \ref{sec-obstructions}.
\item Let $\Gamma\rightarrow\mathbb{R}$ be a realizable harmonic degree $3$ morphism.
If $\Gamma$ is not smoothly embedded into $\mathbb{R}^2$, one can usually use tropical modification of the surrounding plane to unfold contracted features, also in a way that respects the morphism. It is hard to describe this process in general, but we present two case studies exemplarily in Section \ref{sec-unfolding}.
    
\end{enumerate}

\subsection{Acknowledgements}
We would like to thank the Humboldt foundation for support.
We thank Erwan Brugall\'e, Andr\'es Jaramillo Puentes, Michael Joswig, Johannes Rau and Ilya Tyomkin for useful discussions.
The first author acknowledges support by the Deutsche Forschungsgemeinschaft (DFG, German Research Foundation), Project-ID 286237555, TRR 195. 
The second author is a member of the INDAM group GNSAGA.

\section{Preliminaries}
\subsection{Plane tropical curves and dual Newton polygons}\label{ssc:DualNewtonPolygons}

We give a short overview about plane tropical curves and their dual Newton subdivisions, for more details, see e.g.\ \cite{BS14, BIMS14, MS15, RST03, CMR23}.

A plane algebraic curve (or, more generally, a curve embedded in a toric surface) over a field $K$ is given as the zero-set of a polynomial in $K[x, y]$. In tropical geometry we can use the analogous definition, by just adding the word tropical everywhere:
tropical plane curves are tropical vanishing loci of tropical polynomials over the tropical semifield.
Alternatively, one can also start with an (algebraically closed) field $K$ which is equipped with a non-Archimedean valuation $\val : K^\times \to \RR$ and an algebraic curve $C$ defined over $K$. 
A natural example is the field of Puiseux series over $\mathbb{C}$.

The \emph{tropicalization} of $C$ is the closure of the image of $C \cap (K^\times)^2$ under the \emph{tropicalization map}
\[ \trop : (K^\times)^2 \longrightarrow \RR^2, \qquad (x,y) \longmapsto \big( -\val(x), -\val(y) \big) \]
and this is a tropical plane curve. It is a piecewise linear graph with edges of rational slope in $\mathbb{R}^2$. In fact, the two definitions mentioned here are equivalent by Kapranov's theorem.

Every tropical plane curve is dual to a Newton subdivision of the Newton polygon of its defining equation. The subdivision is given by projecting upper faces of the extended Newton polygon, which takes the valuations of the coefficients into account. Every vertex of the tropical curve is dual to a polygon of the subdivision, every edge dual to an orthogonal edge and every part of $\mathbb{R}^2$ minus the tropical curve is dual to a vertex of the subdivision. An example is shown in Figure \ref{fig:ex-planecurve}.

\begin{figure}[h]
    \centering

\tikzset{every picture/.style={line width=0.75pt}} 

\begin{tikzpicture}[x=0.75pt,y=0.75pt,yscale=-1,xscale=1]

\draw    (200,20) -- (200,140) ;
\draw    (260,140) -- (200,140) ;
\draw    (200,20) -- (260,140) ;
\draw    (220,60) -- (220,140) ;
\draw    (240,100) -- (240,140) ;
\draw    (240,120) -- (260,140) ;
\draw    (220,100) -- (240,120) ;
\draw    (200,120) -- (220,140) ;
\draw    (200,120) -- (220,120) ;
\draw    (200,120) -- (220,100) ;
\draw    (220,120) -- (240,120) ;
\draw    (220,140) -- (240,120) ;
\draw    (220,100) -- (240,100) ;
\draw    (220,80) -- (240,100) ;
\draw    (200,100) -- (220,100) ;
\draw    (200,100) -- (220,80) ;
\draw    (200,100) -- (220,60) ;
\draw    (200,80) -- (220,60) ;
\draw    (200,60) -- (220,60) ;
\draw    (200,40) -- (220,60) ;
\draw    (310,140) -- (290,140) ;
\draw    (310,160) -- (310,140) ;
\draw    (310,140) -- (320,130) ;
\draw    (320,130) -- (320,120) ;
\draw    (330,130) -- (320,130) ;
\draw    (320,120) -- (310,110) ;
\draw    (330,120) -- (320,120) ;
\draw    (330,130) -- (330,120) ;
\draw    (310,110) -- (290,110) ;
\draw    (310,110) -- (310,100) ;
\draw    (310,100) -- (300,90) ;
\draw    (340,100) -- (310,100) ;
\draw    (340,110) -- (340,100) ;
\draw    (330,120) -- (340,110) ;
\draw    (340,140) -- (330,130) ;
\draw    (340,160) -- (340,140) ;
\draw    (350,140) -- (340,140) ;
\draw    (350,160) -- (350,140) ;
\draw    (350,140) -- (380,110) ;
\draw    (380,110) -- (340,110) ;
\draw    (400,100) -- (380,110) ;
\draw    (340,100) -- (350,90) ;
\draw    (350,90) -- (300,90) ;
\draw    (370,80) -- (350,90) ;
\draw    (280,80) -- (300,90) ;
\draw    (280,80) -- (260,80) ;
\draw    (280,80) -- (270,70) ;
\draw    (270,70) -- (250,70) ;
\draw    (270,70) -- (270,60) ;
\draw    (270,60) -- (250,60) ;
\draw    (270,60) -- (280,50) ;
\draw    (280,50) -- (260,50) ;
\draw    (300,40) -- (280,50) ;
\draw    (480,140) -- (460,140) ;
\draw    (480,160) -- (480,140) ;
\draw    (480,140) -- (490,130) ;
\draw [color={rgb, 255:red, 139; green, 87; blue, 42 }  ,draw opacity=1 ]   (490,130) -- (490,120) ;
\draw [color={rgb, 255:red, 139; green, 87; blue, 42 }  ,draw opacity=1 ]   (500,130) -- (490,130) ;
\draw [color={rgb, 255:red, 184; green, 233; blue, 134 }  ,draw opacity=1 ]   (490,120) -- (480,110) ;
\draw [color={rgb, 255:red, 139; green, 87; blue, 42 }  ,draw opacity=1 ]   (500,120) -- (490,120) ;
\draw [color={rgb, 255:red, 80; green, 227; blue, 194 }  ,draw opacity=1 ]   (500,130) -- (500,120) ;
\draw    (480,110) -- (460,110) ;
\draw [color={rgb, 255:red, 184; green, 233; blue, 134 }  ,draw opacity=1 ]   (480,110) -- (480,100) ;
\draw [color={rgb, 255:red, 74; green, 144; blue, 226 }  ,draw opacity=1 ]   (480,100) -- (470,90) ;
\draw [color={rgb, 255:red, 208; green, 2; blue, 27 }  ,draw opacity=1 ]   (510,100) -- (480,100) ;
\draw [color={rgb, 255:red, 184; green, 233; blue, 134 }  ,draw opacity=1 ]   (510,110) -- (510,100) ;
\draw [color={rgb, 255:red, 144; green, 19; blue, 254 }  ,draw opacity=1 ]   (500,120) -- (510,110) ;
\draw [color={rgb, 255:red, 144; green, 19; blue, 254 }  ,draw opacity=1 ]   (510,140) -- (500,130) ;
\draw    (510,160) -- (510,140) ;
\draw [color={rgb, 255:red, 144; green, 19; blue, 254 }  ,draw opacity=1 ]   (520,140) -- (510,140) ;
\draw    (520,160) -- (520,140) ;
\draw [color={rgb, 255:red, 144; green, 19; blue, 254 }  ,draw opacity=1 ]   (520,140) -- (550,110) ;
\draw [color={rgb, 255:red, 144; green, 19; blue, 254 }  ,draw opacity=1 ]   (550,110) -- (510,110) ;
\draw    (570,100) -- (550,110) ;
\draw [color={rgb, 255:red, 74; green, 144; blue, 226 }  ,draw opacity=1 ]   (510,100) -- (520,90) ;
\draw [color={rgb, 255:red, 74; green, 144; blue, 226 }  ,draw opacity=1 ]   (520,90) -- (470,90) ;
\draw    (540,80) -- (520,90) ;
\draw    (450,80) -- (470,90) ;
\draw    (450,80) -- (430,80) ;
\draw    (450,80) -- (440,70) ;
\draw    (440,70) -- (420,70) ;
\draw    (440,70) -- (440,60) ;
\draw    (440,60) -- (420,60) ;
\draw    (440,60) -- (450,50) ;
\draw    (450,50) -- (430,50) ;
\draw    (470,40) -- (450,50) ;
\draw [color={rgb, 255:red, 144; green, 19; blue, 254 }  ,draw opacity=1 ]   (630,120) .. controls (670,90) and (697.25,146.72) .. (630,130) ;
\draw [color={rgb, 255:red, 80; green, 227; blue, 194 }  ,draw opacity=1 ]   (630,130) -- (630,120) ;
\draw [color={rgb, 255:red, 139; green, 87; blue, 42 }  ,draw opacity=1 ]   (630,130) .. controls (611.25,131.72) and (614.25,116.22) .. (630,120) ;
\draw [color={rgb, 255:red, 184; green, 233; blue, 134 }  ,draw opacity=1 ]   (610,100) .. controls (610.75,108.72) and (611.75,115.72) .. (620,120) ;
\draw [color={rgb, 255:red, 184; green, 233; blue, 134 }  ,draw opacity=1 ]   (640.5,114) .. controls (642.75,110.09) and (639.25,108.09) .. (636.5,104.59) ;
\draw [color={rgb, 255:red, 208; green, 2; blue, 27 }  ,draw opacity=1 ]   (610,100) .. controls (620,100.09) and (633.5,102.09) .. (636.5,104.59) ;
\draw [color={rgb, 255:red, 74; green, 144; blue, 226 }  ,draw opacity=1 ]   (610,100) .. controls (596.75,74.59) and (658.5,82.84) .. (636.5,104.59) ;

\end{tikzpicture}

     \caption{A Newton subdivision and a dual tropical plane curve. On the right, an abstract tropical curve that can be embedded into the plane, as the skeleton of the plane tropical curve, as depicted in color on the second to right picture.}
    \label{fig:ex-planecurve}
\end{figure}
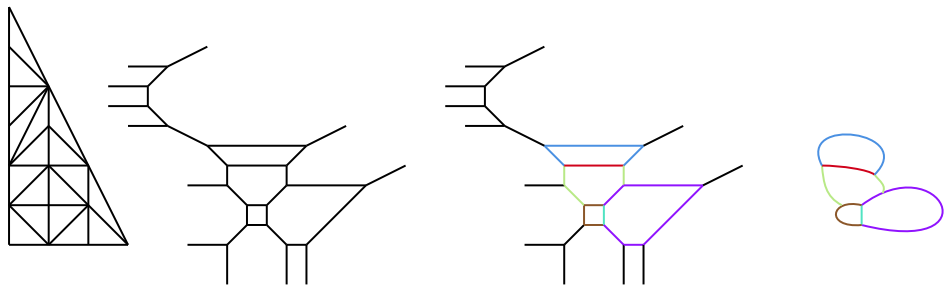

Every edge of a tropical plane curve has a \emph{direction vector}, which is defined to be the primitive direction vector of the line of which it is a segment times a weight which equals the lattice length of the dual edge in the dual Newton subdivision. Using the direction vector, we associate a \emph{length} to an edge of a tropical plane curve, which equals the Euclidean length divided by the norm of its direction vector. 

We say that a tropical plane curve is \emph{smooth} if its dual Newton subdivision contains only triangles of minimal area.

The \emph{degree} of a tropical plane curve is its dual Newton polygon. We say that a tropical plane curve is of degree $d$, if its dual Newton polygon is the triangle with vertices $(0,0)$, $(d,0)$ and $(0,d)$.

We say that a tropical plane curve is in $\mathbb{F}_n$, of degree $(3,\frac{g+3n+2}{2})$, if it is dual to the Newton polygon depicted in Figure \ref{fig-polygon}.

\subsection{Abstract tropical curves}
An \emph{abstract tropical curve} is a connected metric graph $\Gamma$ with unbounded rays or ``ends'' and a genus function $g:\Gamma\rightarrow \NN$ which is nonzero only at finitely many points. Locally around a point $p$, $\Gamma$ is homeomorphic to a star with $r$ half-rays. 
The number $r$ is called the \emph{valence} of the point $p$ and denoted by $\val(p)$.
We require that there are only finitely many points with $\val(p)\neq 2$ and that the set of all points of nonzero genus or valence larger than $2$ is nonempty.
A finite set of points containing  (but not necessarily equal to the set of) all points of nonzero genus or valence larger than $2$ may be chosen; its elements are called  \emph{vertices}. 

Edges which are not ends are called \emph{bounded edges}. The graph is equipped with a length function $l$ that assigns to each bounded edge a positive real number. 
An isomorphism of a tropical curve is a homeomorphism respecting the metric and the genus function. The \emph{genus} of a tropical curve is the first Betti number $b_1(\Gamma)$ plus the genera of all vertices. 

The \emph{combinatorial type} $G$ of a tropical curve $\Gamma$ is obtained by dropping the information on the metric. We will also denote by $V(G)$, $E(G)$ the sets of its vertices and bounded edges, respectively.
We say that $(G,l)$ is a \emph{model} for $\Gamma$ and write $\Gamma=(G,l)$. Notice that the model of a tropical curve (or equivalently of its underlying metric graph) is not unique; one can for instance always refine $G$ by adding an arbitrary number of vertices and define the length function in the new edges accordingly. 
Often, we use the model with the smallest vertex set and denote it by $\Gamma$, so that the graph structure is given with the tropical curve in the following. When another graph structure is needed, we will point to it.

Following this convention, let $\Gamma$ be an abstract tropical curve with the graph structure given by the smallest vertex set. We can shrink all ends, and, iteratively, all edges which become leaf edges. The thus resulting abstract tropical curve $\Gamma'$ is called the \emph{canonical model} of $\Gamma$. 

Two metric graphs are \emph{tropically equivalent} if they have the same canonical model. 
A \emph{tropical modification} of a metric graph $\Gamma$ is a metric graph $\tilde\Gamma$ tropically equivalent to $\Gamma,$ obtained by gluing trees at some points. 

Fix a genus $g$ and a combinatorial type of a canonical model of a tropical curve of genus $g$ without ends. All abstract tropical curves of this combinatorial type are parametrized by an orthant $\mathbb{R}_{>0}^b$, where $b$ is the number of bounded edges, respectively, by the quotient cone of this orthant by the automorphism group of the underlying graph. Moving to the boundary of a cone means shrinking an edge of the tropical curve to $0$, thus obtaining a new combinatorial type with less edges. We can identify the boundary with the cone corresponding to this new combinatorial type. In this way, cones for different combinatorial types can be glued to form an \emph{abstract cone complex} called the \emph{moduli space of abstract tropical curves of genus $g$} which parametrizes all abstract tropical curves of genus $g$ and is denoted by $\mathcal{M}_{g}^{\rm trop}$ \cite{BMV11, CMV12, Cha10, Cap13, Mi07}.

Abstract tropical curves whose vertices are all of genus $0$ and $3$-valent have the maximal number of bounded edges and therefore lead to cones of maximal dimension in the moduli space of abstract tropical curves. 
For that reason, we call them abstract tropical curves of \emph{maximal combinatorial type}. 
Since the subset of abstract tropical curves which are not of maximal combinatorial type in the moduli space is of lower dimension, a restriction to abstract tropical curves of maximal combinatorial types can be viewed as a way to consider general abstract tropical curves.
We will take this point of view in the following.

Consider a smooth plane tropical curve. By the smoothness condition, every vertex is $3$-valent.  Forgetting the embedding in the plane, and taking the canonical model, with the lengths we defined for the edges of a plane tropical curve, is an abstract tropical curve which we call the image under the \emph{forgetful map}. It is of genus $g$, if the dual Newton polygon had $g$ interior points. By fixing a dual Newton polygon $\Delta$ and considering all smooth plane tropical curves dual to a subdivision of $\Delta$, the forgetful map goes from plane tropical curves to the moduli space of abstract tropical curves of genus $g$ \cite{BJMS15}. By smoothness, the image of this map is always of maximal combinatorial type, that is, generic in the moduli space of abstract tropical curves.

\begin{definition}
    We say that an abstract tropical curve $\Gamma$ is \emph{realizable as a plane tropical curve} (or \emph{can be embedded as a plane tropical curve}) if there exists a smooth plane tropical curve whose image under the forgetful map is $\Gamma$.
\end{definition}

The theory of semistable reduction offers a way to tropicalize an abstract curve $C$ defined over a field with non-Archimedean valuation by taking the dual graph of the abstract curve of the special fiber (with a metric taking into account ''how fast'' nodes form, i.e.\ what the valuation of a Puiseux series $p$ in a local equation $xy=p$ of the node is). This produces an abstract tropical curve $\Gamma$ \cite{Tyo09, BPR11a}. This process can also be described by taking the minimal skeleton of the Berkovich analytic space of $C$. 

Assume the abstract curve $C$ is embedded into a toric surface, and assume that the embedded tropicalization is smooth. Then the image under the forgetful map of the thus resulting plane tropical curve is equal to the abstract tropicalization \cite[Section 5]{BPR11a}.

\subsection{Harmonic morphisms/tropical covers}

The analogues of ramified covers of Riemann surfaces in the tropical world are based on harmonic morphisms:
\begin{definition}
    A continuous map $\varphi:\Gamma\to\Gamma'$ is a \emph{morphisms of metric graphs} if there exist models $\Gamma=(G,l)$, $\Gamma'=(G',l')$ such that $\varphi(V(G))\subset V(G')$, $\varphi^{-1}(E(G'))\subset E(G)$ and for any $e\in E(G)$
    \[
    \mu_{\varphi}(e)=\begin{cases}
    \frac{l'(\varphi(e))}{l(e)} \in\mathbb Z_{>0}& \text{if } \varphi(e)\in E(G');\\
    0&\text{if } \varphi(e)\in V(G').\\
    \end{cases}
    \]
    Furthermore, if for any $v\in V(G)$ there exists $e\in E_v(G)$ such that $\mu_{\varphi}(e)>0$ and it contracts no loops we will say that the morphism is \emph{non-degenerate}. If  $\mu_{\varphi}(e)>0$ for any $e\in E(G)$, then the morphism is \emph{finite}. 
    
    The morphism of metric graphs is \emph{harmonic} if for any $v\in V(G)$ the integer 
    \[m_{\varphi}(v):=\sum_{\substack{e\in E_v(G);\\ \varphi(e)=e'}}\mu_{\varphi}(e)
    \]
    is independent of $e'\in E_{\varphi(v)}(G').$
    This number is called the \emph{multiplicity} of $\varphi$ at $v$. 
    When the morphism is harmonic, then its \emph{degree} is $d(\varphi):=\sum_{e;\varphi(e)=e'}\mu_{\varphi}(e)$ and it is independent of $e'\in E(G')$.
\end{definition}

\begin{definition}
A \emph{tropical cover} $\varphi:\Gamma_1\to \Gamma_2$ of degree $d$ is a harmonic morphism of degree $d,$ which satisfies the \emph{Riemann--Hurwitz inequality} at any $v\in \{\varphi^{-1}(t)|\, t\in \Gamma_2; \operatorname{val}(t)\geq 2\}$, i.e.
\begin{equation}\label{eq:RH}
    (2m_{\varphi}(v)-2)-\sum_{e\in E_v(\Gamma_1)}(\mu_{\varphi}(e)-1)\geq 0,
\end{equation}

where $E_v(G_1)$ denotes the set of outgoing edges from $v$ in $\Gamma_1=(G_1,l_1)$.
\end{definition}

\begin{remark}\label{rem-RH}
    In \cite{CMR}, a tropical cover is defined as a harmonic morphism which satisfies the \emph{Riemann--Hurwitz equality}. However, in the next section we will consider morphisms allowing tropical modifications.
    In this case, we can always attach leaves, counted with multiplicities, whenever we have only a strict inequality and reach the equality. Then with abuse of notation, we will refer to \eqref{eq:RH} as an equality.
\end{remark}

Given a tropical cover as above, to any $v\in V(G_1)$ we can associate a number, called \emph{local Hurwitz number}, defined as in \cite[Section 3.2.4]{CMR}. 
These numbers depend on \eqref{eq:RH} and a tropical cover is called \emph{realizable} if its local Hurwitz numbers are all nonzero.

Whether the Riemann-Hurwitz equality gives rise to nonzero Hurwitz numbers is an open problem referred to as the \emph{Hurwitz existence problem}. This has been studied in \cite[Section 2.2]{Cap13} and in particular it has been observed that for $d\leq 3$ the Hurwitz existence problem is solved, i.e. any tropical cover of degree $d\leq 3$ has nonzero local Hurwitz numbers and therefore it is always realizable. Since our focus is on degree $3$ covers we will just implicitly consider realizable tropical covers while writing simply tropical covers.

Tropical covers of metric trees have in particular been extensively studied, as they represent the tropical analogue of admissible covers. 
We will call a finite tropical cover of a metric tree an \emph{admissible tropical cover}. 
We now restrict our attention to a particular class of these covers, defined as follows.

\begin{definition}\label{def-wellcontraced}
A \emph{well-contracted cover} is a non-degenerate tropical cover $\varphi:\Gamma_1\to \Gamma_2$ with $\Gamma_2$ a path.
\end{definition}

\begin{remark}
    Let us observe that a well-contracted cover can always be recovered from the following data: an admissible cover $\varphi:\Gamma_1\to \Gamma_2$ where $\Gamma_2$ is a metric tree that becomes a path, up to contraction of leaf-edges, such that the contraction of the corresponding edges in the pre-image of any leaf-edge doesn't decrease the genus of $\Gamma_2.$ 
    
    For instance, let us consider an admissible cover $\varphi:\Gamma_1\to \Gamma_2$ where $\Gamma_2$ is metric tree with three edges and exactly one vertex of valence $3$ and the contraction of the pre-images of a leaf-edge does not decrease the genus of $\Gamma_1$.
    In particular, if an edge in the pre-image is a leaf-edge, we just remove it and consider the tropically equivalent graph with that leaf-edge contracted.
    For any other edge, instead, we notice that this cannot be a loop (otherwise its contraction would then decrease the genus by 1).
    Then its endpoints cannot coincide and we just define a morphism that contracts this edge. This morphism continues to be harmonic, since harmonicity only depends on the incident edges to a non-leaf vertex on the target. By removing a leaf-edge, this doesn't affect harmonicity. For an example see Figure \ref{fg:well-contracted}.
\end{remark}

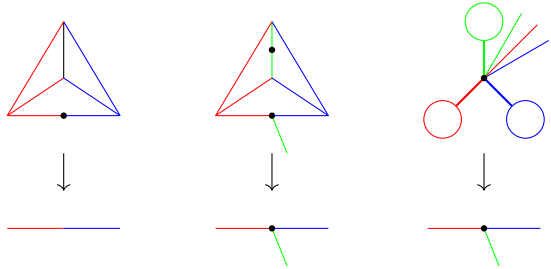
\begin{figure}[h]
        \centering
        \begin{tikzcd}
            \begin{tikzpicture}
                \draw[](0,0)--(0,0.75);
                \draw[red](-0.75,-0.5)--(0,0);
                \draw[blue](0.75,-0.5)--(0,0);
                \draw[blue](0.75,-0.5)--(0,0.75);
                \draw[red](-0.75,-0.5)--(0,0.75);
                \draw[red](-0.75,-0.5)--(0,-0.5);
                \draw[blue](0.75,-0.5)--(0,-0.5);
           
                \vertex{0,-0.5}
                \draw[->](0,-1)--(0,-1.5);
                \draw[red](-0.75,-2)--(0,-2);
                \draw[blue](0.75,-2)--(0,-2);
                
            \end{tikzpicture}&
            \begin{tikzpicture}
                \draw[green](0,0)--(0,0.75);\vertex{0,0.375}
                \draw[red](-0.75,-0.5)--(0,0);
                \draw[blue](0.75,-0.5)--(0,0);
                \draw[blue](0.75,-0.5)--(0,0.75);
                \draw[red](-0.75,-0.5)--(0,0.75);
                \draw[red](-0.75,-0.5)--(0,-0.5);
                \draw[blue](0.75,-0.5)--(0,-0.5);
                \draw[green](0,-0.5)--(0.2,-1);
                \vertex{0,-0.5}
                \draw[->](0,-1)--(0,-1.5);
                \draw[red](-0.75,-2)--(0,-2);
                \draw[blue](0.75,-2)--(0,-2);
                \draw[green](0,-2)--(0.2,-2.5);\vertex{0,-2}
            \end{tikzpicture}&
            \begin{tikzpicture}
                \draw[blue,thick](0,0)--(0.37,-0.37);
                \draw[red,thick](0,0)--(-0.37,-0.37);
                \draw[green,thick](0,0)--(0,0.5);
                \draw[green](0,0.75) circle (0.25);
                \draw[red](-0.55,-0.55) circle (0.25);
                \draw[blue](0.55,-0.55) circle (0.25);
                \draw[green](0,0)--(0.5,0.86);
                \draw[blue](0,0)--(0.86,0.5);
                \draw[red](0,0)--(0.71,0.71);
                \vertex{0,0}
                \draw[->](0,-1)--(0,-1.5);
                \draw[red](-0.75,-2)--(0,-2);
                \draw[blue](0.75,-2)--(0,-2);
                \draw[green](0,-2)--(0.2,-2.5);\vertex{0,-2}
            \end{tikzpicture}
            \end{tikzcd}\caption{On the left a well-contracted cover and in the middle the admissible cover from which we can recover it. On the right instead is depicted an admissible cover which does not define a well-contracted cover: for any edge in the target tree, the contraction of the edges in its preimage  would decrease the genus of the graph.}\label{fg:well-contracted}
\end{figure}

\subsection{Definitions of gonality for tropical curves and their relation}\label{ssc:def_gonality_and_relations}

Here we recall the various definitions of gonality for tropical curves. Similarly to the algebraic case, one can define the gonality via the existence of a divisor of given degree and rank, or via the existence of a harmonic morphism to a metric tree. However, unlike the algebraic case, such definitions do not coincide for tropical curves.
One of the reasons behind this discrepancy is that the Baker--Norine rank of a divisor on a tropical curve behaves very differently from its algebraic counterpart.

A divisor $D_C$ on an algebraic curve $C$ naturally defines
a divisor $D$ on the metric graph obtained by the tropicalization of $C$. Baker and Norine introduced the concept of rank of a divisor on a graph, mirroring the notion of rank for divisors on algebraic curves, and proved the remarkable fact that it obeys an exact analogue of the Riemann-Roch Theorem, \cite{BN}. However, Baker’s Specialization Lemma \cite[Lemma 2.8]{B08} shows that the rank of $D$ can only increase with respect to that of $D_C$. 

In order to define gonality via the existence of a divisor, let us recall first some notations and definitions from divisor theory for abstract tropical curves.

\begin{definition}
Let $\Gamma$ be a metric graph. A \emph{divisor} of $\Gamma$ is a formal sum
$D:=\sum_{x\in \Gamma}D(x)x,$
where $D(x)\in\mathbb{Z}$ and $D(x)\neq0$ for finitely many $x\in\Gamma.$ 

The \emph{degree} of $D$ is $\operatorname{deg}(D)=\sum_{x\in \Gamma}D(x).$ and we say that $D$ is \emph{effective}, and write $D\geq 0,$ if $D(x)\geq 0,$ for any $x\in\Gamma.$ 

A \emph{rational function} on $\Gamma$ is a continuous and piecewise linear function $f:\Gamma\rightarrow\mathbb{R},$ with integer slopes along its domains of linearity.

A divisor $D$ is \emph{principal} if $D=\operatorname{div}(f),$ for some rational function $f,$ where $\operatorname{div}(f)$ is the divisor with $\operatorname{div}(f)(x)$ equal to the sum of all slopes of $f$ along edges emanating from $x.$ 

Two divisors $D_1, D_2$ are \emph{linearly equivalent} if their difference is principal and we write $D_1\sim D_2$.

The (Baker--Norine) \emph{rank} of $D$ is set to be  $\operatorname{r}(D)=-1$ if $D$ is not equivalent to an effective divisor; otherwise $$\operatorname{r}(D)=\operatorname{max}\{k\in\mathbb{Z}_{\geq 0}:\forall E\geq 0\text{ with }\operatorname{deg}(E)=k, \exists\,E'\geq 0 \text { such that } D-E\sim E'\}.$$

\end{definition}

\begin{definition}
    A metric graph $\Gamma$ is \emph{divisorially $d$-gonal} if there exists a divisor $D$ of degree $d$ and (Baker--Norine) rank at least $1.$
\end{definition}

The definition of gonality via the existence of a morphism is instead as follows.
\begin{definition}
A metric graph $\Gamma$ is \emph{$d$-gonal} if there exists a non-degenerate harmonic morphism $\varphi:\Gamma'\to T$ of degree $d,$ where $\Gamma'$ is a tropical modification of $\Gamma$ and $T$ a metric tree. 
\end{definition}

\begin{remark}
    Let us observe that in the literature, $d$-gonality can also be defined via a finite morphism. Notice however, that a non-degenerate morphism can always be replaced with a finite one, up to tropical modification, as proved in \cite[Proposition 2.11]{MZ1}.
\end{remark}

The definition of $d$-gonality can also be related to the existence of an admissible cover, as it is harmonic by definition.
In particular a tropical curve that admits a tropical cover (after tropical modifications) of degree $d$ of a metric tree, is $d$-gonal. 
Moreover, given a non-degenerate harmonic morphism to a tree of degree $d$, one can always consider the pullback of a point. This will yield a divisor of degree $d$ and rank at least $1,$ \cite[Lemma 2.25]{MZ1}. 
Then in general the following holds for any $d\geq 2$.
\begin{align*}
\Gamma\text{ admits a degree $d$ well-contracted cover}&\Rightarrow\Gamma\text{ admits a degree $d$ admissible cover}\\
\Rightarrow \Gamma\,\text{$d$-gonal }&\Rightarrow \text{$\Gamma$ divisorially $d$-gonal.}
\end{align*}

When $d=3$, let us recall the known relations between the above definitions, proved in \cite{MZ1}, \cite{MZ2}.

\begin{theorem}
    Let $\Gamma$ be a $3$-edge connected metric graph. Then, up to tropical modifications,  
    $$\Gamma\text{ admits a degree $3$ admissible cover}\Leftrightarrow \text{$\Gamma$ trigonal }\Leftrightarrow \text{$\Gamma$ divisorially trigonal.}$$
\end{theorem}

\begin{theorem}
    Let $\Gamma$ be a metric graph with no separating vertices or multiple edges. Then, up to tropical modifications,  
    $$\Gamma\text{ admits a degree $3$ admissible cover}\Leftrightarrow \text{$\Gamma$ trigonal }\Leftrightarrow \text{$\Gamma$ divisorially trigonal.}$$
\end{theorem}

From \cite[Theorem 2]{DV}, the tree gonality of a metric graph, i.e. the minimum
degree of a tropical morphism from any tropical modification of the metric graph to a
metric tree, is at most $\lceil g/2\rceil+1$ and such a bound is sharp. Hence all genus $3,4$ tropical curves admit a degree $3$ admissible cover, therefore they are all trigonal and divisorially trigonal. The relation with well-contracted covers will instead be inspected in Theorem \ref{th:realizability_genus_3_4}. We will see that for tropical curve of maximal combinatorial type, having a degree $3$ well-contracted cover is strongly related to being smoothly embedded in $\mathbb R^2,$ as a dual unimodular triangulation of a polygon in Figure \ref{fig-polygon}.

\section{The projection of a tropical curve in a Hirzebruch surface}\label{sec-proj}

Here, we prove the observation from the introduction relating tropical plane curves in Hirzebruch surfaces to trigonal abstract tropical curves.

\begin{proposition}\label{lm:realizability_cover}
Assume $\Gamma$ can be embedded into $\mathbb{R}^2$, such that (a modification of) $\Gamma$ is dual to a unimodular triangulation of the polygon in Figure \ref{fig-polygon}, i.e.\ it is a smooth tropical plane curve in a Hirzebruch surface $\mathbb{F}_n$ of degree $(3,b)$ for some $b$. Then the projection restricted to $\Gamma$ is a realizable
well-contracted tropical cover of degree $3$.
\end{proposition}
\begin{proof}
    The projection induces a map to  $\mathbb{R}$, where the source graph $\Gamma'$ is a modification of $\Gamma$ by definition. 
    Harmonicity follows from the balancing condition. 
    
    By the smoothness hypothesis, every vertex of $\Gamma$ is $3$-valent and maps to a $2$-valent vertex of the subdivided $\mathbb{R}$. By the balancing condition, either two edges map in the same direction and their weights add up to the weight of the third, which maps in the opposite direction, or one edge is contracted and the other two map to opposite directions, with the same weight. 
    
    The multiplicity equals the sum of the two weights of the edges mapping to the same direction resp.\ the weight of the two edges mapping in opposite directions in the second case. In any case, the equation on the left of the Riemann-Hurwitz inequality yields $1$.
    
    For a $2$-valent vertex, which we obtain from subdividing to have the structure of a graph morphism, the equation on the left of the Riemann-Hurwitz inequality yields $0$. In any case, the Riemann-Hurwitz condition is satisfied. 

We thus obtain a tropical cover. Since it is not possible that all edges adjacent to a vertex are contracted, the cover is non-degenerate.
    
    By smoothness, the embedded tropical curve in $\FF_n$ has three vertical ends of weight $1$. Taking preimages of a sufficiently negative point, we obtain three points in these ends, so altogether, the tropical cover $\Gamma'\rightarrow \mathbb{R}$ is of degree $3$. 

Because of the smooth embedding into $\mathbb{R}^2$, the projection cannot contract loops.

Thus, we obtain a well-contracted cover of degree $3$. Notice that by the discussion below Remark \ref{rem-RH}, since the cover is of degree $3$, it is automatically realizable.
   
\end{proof}

One can also prove Proposition \ref{lm:realizability_cover} by means of algebraic geometry: the smooth plane tropical curve is the tropicalization of an algebraic curve embedded into $\mathbb{F}_n$. The projection there induces a degree $3$ morphism to $\mathbb{P}^1$ whose tropicalization can be contracted to yield our realizable harmonic morphism, see Figure \ref{fig:sketchtropicalization}.

\section{Abstract tropical curves of genus $3$ and $4$ of maximal combinatorial type and their Maroni invariant}\label{sec-maroni}

Let us now consider abstract tropical curve of maximal combinatorial type of genus $3$ and $4.$
In \cite{BJMS15}, the authors have studied which abstract tropical curves arise from tropical plane curves. In other words, they have determined edge length conditions that an abstract tropical curve has to satisfy in order to be embeddable in the tropical plane as a curve of degree $d$, i.e.\ with ends of directions $(-1,0)$, $(0,-1)$ and $(1,1)$ each $d$ times.
Here we want to repeat the same study, replacing degree $d$ with being dual to the polytope of Hirzebruch surfaces, whose degree varies as the algebraic Maroni invariant.

\subsection{Genus 3}\label{ssc: genus3}
As an algebraic genus $3$ trigonal curve has Maroni invariant $1,$  we will consider here $\mathbb F_1.$
Indeed, from Figure \ref{fig-polygon}, the integers $g,n$ must have the same parity, otherwise $\frac{g+3n+2}{2}$ and $\frac{g-3n+2}{2}$ wouldn't be integers. Moreover, a genus $g$ abstract tropical curve cannot be realizable in $\mathbb F_n$ with $n>\frac{g+2}{3}.$ Then a genus $3$ curve can only be realizable in $\mathbb F_1,$ as in the algebraic case.

From Figure \ref{fig-polygon} one can see that the polygon defining $\mathbb F_1$ is strictly contained in that of a plane curve of degree $4$ and of genus $3,$ and differs just by a small triangle, as in Figure \ref{fg:polygons_g3}.

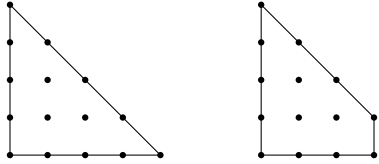
\begin{figure}[h]
    \centering
    \begin{tikzcd}
        \begin{tikzpicture}
            \draw(0,0)--(0,2);
            \draw(0,0)--(2,0);
            \draw(2,0)--(0,2);
            \foreach \j in {0,1,2,0.5,1.5}{
            \vertex{0,\j}
        }
        \foreach \j in {0,0.5,1,1.5}{
            \vertex{0.5,\j}
        }
        \foreach \j in {0,0.5,1}{
            \vertex{1,\j}
        }
        \vertex{1.5,0}\vertex{1.5,0.5}
        \vertex{2,0}
        \end{tikzpicture}&
        \begin{tikzpicture}
            \draw(0,0)--(0,2);
            \draw(0,0)--(1.5,0);
            \draw(1.5,0.5)--(0,2);
            \draw(1.5,0)--(1.5,0.5);
            \foreach \j in {0,1,2,0.5,1.5}{
            \vertex{0,\j}
        }
        \foreach \j in {0,0.5,1,1.5}{
            \vertex{0.5,\j}
        }
        \foreach \j in {0,0.5,1}{
            \vertex{1,\j}
        }
        \vertex{1.5,0}\vertex{1.5,0.5}
        \end{tikzpicture}
    \end{tikzcd}
    \caption{The polygons defining a genus $3$ curve embedded with degree $4$ in the plane (on the left) and  in the first Hirzebruch surface (on the right).}
    \label{fg:polygons_g3}
\end{figure}

Accordingly, for the genus $3$ case, we can deduce most of the edge length conditions that an abstract tropical curve has to satisfy in order to be embedded in the first Hirzebruch surface directly from \cite[Section 5]{BJMS15}.

Let us use the same notation used in \cite[Section 5]{BJMS15} for the maximal combinatorial types, depicted in Figure \ref{fg:genus3_maxtype}. Clearly, if a tropical curve is not realizable in the plane it cannot be realizable in a Hirzebruch surface, therefore we will consider just the maximal types of tropical curves realizable as plane curves.

\begin{figure}[h]
\centering
\begin{tikzcd}
    \begin{tikzpicture}
        \coordinate (1) at (0,0);
        \coordinate (2) at (1,0);
        \coordinate (3) at (0.5,1);
        \coordinate (4) at (0.5,0.4);
        \draw (1)--(2);
        \draw (1)--(3);
        \draw (1)--(4);
        \draw (3)--(4);
        \draw (3)--(2);
        \draw (4)--(2);
        \foreach \i in {1,2,3,4}{
            \vertex{\i}
        }
        \draw (0.5,-0.5) node[] {(000)};
    \end{tikzpicture}&
        \begin{tikzpicture}
        \coordinate (1) at (0,0);
        \coordinate (3) at (1,0);
        \coordinate (4) at (0,1);
        \coordinate (5) at (1,1);
        \draw (1)--(3);
        \draw (4)--(5);
        \draw(1)[] to [out=120, in=240] (4);
        \draw(1)[] to [out=60, in=300] (4);
        \draw(3)[] to [out=120, in=240] (5);
        \draw(3)[] to [out=60, in=300] (5);
        \foreach \i in {1,3,4,5}{
            \vertex{\i}
        }
        \draw (0.5,-0.5) node[] {(020)};
    \end{tikzpicture}&
    \begin{tikzpicture}
        \coordinate (1) at (0,0);
        \coordinate (2) at (0.5,0.5);
       
        \coordinate (4) at (0,1);
        \coordinate (6) at (1,0.5);
        \draw (1)--(2);
        \draw (2)--(6);

        \draw (4)--(2);
       
        \draw(1)[] to [out=120, in=240] (4);
        \draw(1)[] to [out=60, in=300] (4);
        \draw (1.3,0.5) circle (0.3);
        \foreach \i in {1,2,4,6}{
            \vertex{\i}
        }
        \draw (0.75,-0.5) node[] {(111)};
    \end{tikzpicture}&
    \begin{tikzpicture}
        \coordinate (2) at (0,0.5);
        \coordinate (4) at (0.5,0.5);
        \coordinate (5) at (1,0.5);
        \coordinate (6) at (1.5,0.5);
        \draw (2)--(4);
        \draw (5)--(6);
        \draw (-0.25,0.5) circle (0.25);
        \draw (1.75,0.5) circle (0.25);
        \draw(4)[] to [out=30, in=150] (5);
        \draw(4)[] to [out=330, in=210] (5);
        \foreach \i in {2,4,5,6}{
            \vertex{\i}
        }
        \draw (0.75,-0.5) node[] {(212)};
    \end{tikzpicture}
\end{tikzcd}\caption{Maximal combinatorial types of genus $3$ curves realizable as plane curves.}\label{fg:genus3_maxtype}
\end{figure}
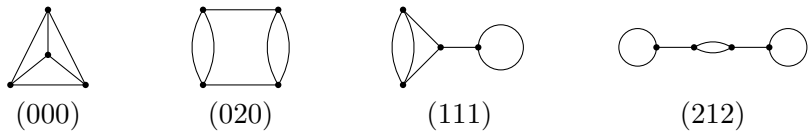

From the proof of \cite[Theorem 5.1] {BJMS15}, one can check that the triangulation realizing the graphs (020), (111), (212) as tropical plane curves induces a triangulation in the polygon associated to $\mathbb F_1$ by changing only the unbounded edges, therefore the edge length conditions will be precisely the same.

Instead, the triangulation in \cite[Figure 5] {BJMS15} defining the edge length conditions in \cite[Theorem 5.1] {BJMS15} for a graph of combinatorial type (000), induces a slightly different triangulation, realizing a graph of type (000) in $\mathbb F_1.$

\begin{figure}[h]
    \centering
    \begin{tikzcd}
        \begin{tikzpicture}
            \draw(0,0)--(0,2);
            \draw(0,0)--(1.5,0);
            \draw(1.5,0.5)--(0,2);
            \draw(1.5,0)--(1.5,0.5);
            \draw(0,2)--(0.5,1);
            \draw(0,1)--(1,1);
            \draw(0,0.5)--(1.5,0.5);
            \draw(0.5,0)--(0.5,1.5);
            \draw(1,0)--(1,1);
            \draw(0,1.5)--(1.5,0);
            \draw(0,1)--(1,0);
            \draw(0,0)--(0.5,0.5);
            
            \foreach \j in {0,1,2,0.5,1.5}{
            \vertex{0,\j}
        }
        \foreach \j in {0,0.5,1,1.5}{
            \vertex{0.5,\j}
        }
        \foreach \j in {0,0.5,1}{
            \vertex{1,\j}
        }
        \vertex{1.5,0}\vertex{1.5,0.5}
        \end{tikzpicture}&
        \begin{tikzpicture}[scale=0.3]
            \draw[blue](0,0+3)--(2,2+3);
            \draw(1,4.5) node{y};
            \draw(0.5,2) node{z};
            \draw(3,1.5) node{w};
            \draw(2,2+3)--(2,3+3);
            \draw(1,3+3)--(2,3+3);
            \draw(1,3+3)--(-1,2+3);
            \draw(-2,1+3)--(-1,2+3);
            \draw(-2,1+3)--(-2,0+3);
            \draw[green](2,2+3)--(3,2+3);
        \draw[green](3,0+3)--(3,2+3);
            \draw[green](3,3)--(1,-2+3);
        \draw[green](0,1)--(1,-2+3);
            \draw[red](0,-2+3)--(0,3);
            \draw[](0,-2+3)--(-1,0);
            \draw(-2,0)--(-1,-3+3);
            \draw(-2,0)--(-3,-2+3);
            \draw(-3,-1+3)--(-3,-2+3);
            \draw(-3,-1+3)--(-2,0+3);
            \draw[](-2,0+3)--(0,0+3);
        \end{tikzpicture}&
        \begin{tikzpicture}
        \coordinate (1) at (0,0.25);
        \coordinate (2) at (1.5,0.25);
        \coordinate (3) at (0.75,1.75);
        \coordinate (4) at (0.75,0.75);
        \draw (1)--(2);
        \draw (1)--(3);
        \draw (1)--(4);
        \draw (3)--(4);
        \draw (3)--(2);
        \draw (4)--(2);
        \foreach \i in {1,2,3,4}{
            \vertex{\i}
        }
        \draw (0.75,0) node[] {v};
        \draw (0.35,0.55) node[] {x};
        \draw (1.15,0.55) node[] {z};
        \draw (0.9,1) node[] {y};
        \draw (0.1,0.8) node[] {u};
        \draw (1.4,0.8) node[] {w};
    \end{tikzpicture}
    \end{tikzcd}
    \caption{The triangulation realizing a graph of type (000) in $\mathbb F_1$, induced by the one realizing the same graph on a plane, with associated cone of maximal dimension.}
    \label{fg:triangulation_g3}
\end{figure}
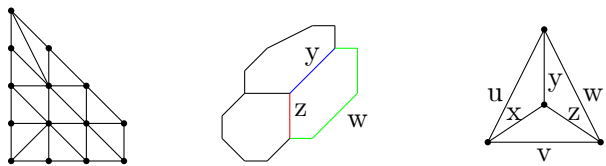

In particular, Figure \ref{fg:triangulation_g3} shows that a graph of type (000) realizable in $\mathbb F_1$ requires the edge length to satisfy $y+z\leq w,$ where $y,z,w$ denote its edge lengths as in Figure \ref{fg:genus3_maxtype}.

Then the analogue of \cite[Theorem 5.1] {BJMS15} is the following.

\begin{theorem}
    An abstract tropical curve of genus $3$ and maximal combinatorial type is realizable in $\mathbb F_1$ if and only if its combinatorial type is one of the ones in Figure \ref{fg:genus3_maxtype}.
    A tropical curve with combinatorial type (000) is realizable in $\mathbb F_1$ if and only if its edge lengths, as labeled in Figure \ref{fg:triangulation_g3}, satisfy, up to symmetry: 
    \[\operatorname{max}\{x,y\}\leq u,\quad \operatorname{max}\{x,z\}\leq v,\quad y+z\leq w.\]
    An abstract tropical curve of combinatorial type (020), (111), (212) is realizable in $\mathbb F_1$ if and only if is realizable as a plane curve of degree $4$, and in particular its edge lengths satisfy the same relations as in \cite[Theorem 5.1] {BJMS15}.
    
\end{theorem}

\subsection{Genus 4}\label{ssc:genus4}
In \cite{BJMS15}, the inequalities on the edge lengths imposed by the condition to be embeddable in the plane (with different degrees) are obtained computationally: by taking the secondary cone parametrizing all plane tropical curve dual to a given subdivision and computing the image of this cone inside the moduli space $\mathcal{M}_g$. The computational challenge in the case of tropical curves of genus $4$ is a lot bigger compared to genus $3$. For this reason, we do not go through all inequalities here. 
Nonetheless, we can provide necessary conditions on the edge lengths of an abstract tropical curve of genus $4$, of fixed maximal combinatorial type, that is realizable in $\mathbb F_0$ and $\mathbb F_2,$ respectively. 

From \cite[Table 3]{BJMS15}, notice first that there is no combinatorial type that is realizable in $\mathbb F_2$ and not in $\mathbb F_0$. In the following, we will consider combinatorial types of genus $4$ tropical curves of maximal combinatorial type realizable in both $\mathbb F_0$ and $\mathbb F_2$, these are represented in Figure \ref{fg:genus4_maxtype}.

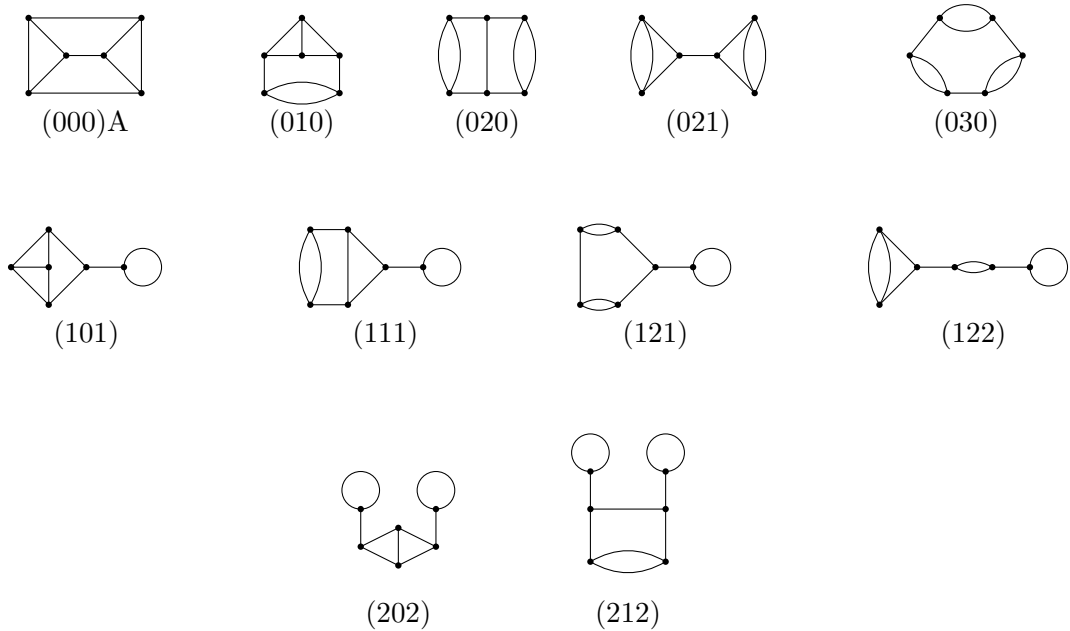
\begin{figure}[h]
\centering
\begin{tikzcd}
    \begin{tikzpicture}
        \coordinate (1) at (0,0);
        \coordinate (2) at (0,1);
        \coordinate (3) at (0.5,0.5);
        \coordinate (4) at (1,0.5);
        \coordinate (5) at (1.5,0);
        \coordinate (6) at (1.5,1);
        \draw (1)--(2);
        \draw (1)--(3);
        \draw (1)--(5);
        \draw (3)--(4);
        \draw (3)--(2);
        \draw (4)--(5);
        \draw (4)--(6);
        \draw (5)--(6);
        \draw (6)--(2);
        \foreach \i in {1,2,3,4,5,6}{
            \vertex{\i}
        }
        \draw (0.75,-0.5) node[] {(000)A};
    \end{tikzpicture}&
    \begin{tikzpicture}
        \coordinate (1) at (0,0);
        \coordinate (2) at (1,0);
        \coordinate (3) at (0,0.5);
        \coordinate (4) at (1,0.5);
        \coordinate (5) at (0.5,0.5);
        \coordinate (6) at (0.5,1);
        \draw (1)--(3);
        \draw (2)--(4);
        \draw (3)--(5);
        \draw (4)--(5);
        \draw (6)--(5);
        \draw (4)--(6);
        \draw (3)--(6);
        \draw(1)[] to [out=30, in=150] (2);
        \draw(1)[] to [out=330, in=210] (2);
        \foreach \i in {1,2,3,4,5,6}{
            \vertex{\i}
        }
        \draw (0.5,-0.5) node[] {(010)};
    \end{tikzpicture}\qquad\quad
        \begin{tikzpicture}
        \coordinate (1) at (0,0);
        \coordinate (2) at (0.5,0);
        \coordinate (3) at (1,0);
        \coordinate (4) at (0,1);
        \coordinate (5) at (1,1);
        \coordinate (6) at (0.5,1);
        \draw (1)--(2);
        \draw (2)--(6);
        \draw (2)--(3);
        \draw (4)--(6);
        \draw (6)--(5);
        \draw(1)[] to [out=120, in=240] (4);
        \draw(1)[] to [out=60, in=300] (4);
        \draw(3)[] to [out=120, in=240] (5);
        \draw(3)[] to [out=60, in=300] (5);
        \foreach \i in {1,2,3,4,5,6}{
            \vertex{\i}
        }
        \draw (0.5,-0.5) node[] {(020)};
    \end{tikzpicture}\qquad\quad
    \begin{tikzpicture}
        \coordinate (1) at (0,0);
        \coordinate (2) at (0.5,0.5);
        \coordinate (3) at (1.5,0);
        \coordinate (4) at (0,1);
        \coordinate (5) at (1.5,1);
        \coordinate (6) at (1,0.5);
        \draw (1)--(2);
        \draw (2)--(6);
        \draw (6)--(3);
        \draw (4)--(2);
        \draw (6)--(5);
        \draw(1)[] to [out=120, in=240] (4);
        \draw(1)[] to [out=60, in=300] (4);
        \draw(3)[] to [out=120, in=240] (5);
        \draw(3)[] to [out=60, in=300] (5);
        \foreach \i in {1,2,3,4,5,6}{
            \vertex{\i}
        }
        \draw (0.75,-0.5) node[] {(021)};
    \end{tikzpicture}&
    \begin{tikzpicture}
        \coordinate (1) at (-0.5,0.5);
        \coordinate (2) at (0,0);
        \coordinate (3) at (0.5,0);
        \coordinate (4) at (1,0.5);
        \coordinate (5) at (0.6,1);
        \coordinate (6) at (-0.1,1);
        \draw (2)--(3);
        \draw (1)--(6);
        \draw (4)--(5);
        \draw(1)[] to [out=350, in=100] (2);
        \draw(1)[] to [out=280, in=170] (2);
        \draw(3)[] to [out=10, in=260] (4);
        \draw(3)[] to [out=80, in=190] (4);
        \draw(6)[] to [out=60, in=120] (5);
        \draw(6)[] to [out=300, in=240] (5);
        \foreach \i in {1,2,3,4,5,6}{
            \vertex{\i}
        }
        \draw (0.25,-0.5) node[] {(030)};
    \end{tikzpicture}\\ 
    \begin{tikzpicture}
        \coordinate (1) at (-0.5,0.5);
        \coordinate (2) at (0,0);
        \coordinate (3) at (0,0.5);
        \coordinate (4) at (0,1);
        \coordinate (5) at (0.5,0.5);
        \coordinate (6) at (1,0.5);
        \draw (1)--(2);
        \draw (1)--(3);
        \draw (1)--(4);
        \draw (4)--(2);
        \draw (4)--(5);
        \draw (2)--(5);
        \draw (6)--(5);
        \draw (1.25,0.5) circle (0.25);
        \foreach \i in {1,2,3,4,5,6}{
            \vertex{\i}
        }
        \draw (0.5,-0.5) node[] {(101)};
    \end{tikzpicture}&
    \begin{tikzpicture}
        \coordinate (1) at (-0.5,0);
        \coordinate (2) at (0,0);
        \coordinate (3) at (-0.5,1);
        \coordinate (4) at (0,1);
        \coordinate (5) at (0.5,0.5);
        \coordinate (6) at (1,0.5);
        \draw (1)--(2);
        \draw (3)--(4);
        \draw (4)--(2);
        \draw (4)--(5);
        \draw (2)--(5);
        \draw (6)--(5);
        \draw(1)[] to [out=120, in=240] (3);
        \draw(1)[] to [out=60, in=300] (3);
        \draw (1.25,0.5) circle (0.25);
        \foreach \i in {1,2,3,4,5,6}{
            \vertex{\i}
        }
        \draw (0.5,-0.5) node[] {(111)};
    \end{tikzpicture}\qquad\qquad
    \begin{tikzpicture}
        \coordinate (1) at (-0.5,0);
        \coordinate (2) at (0,0);
        \coordinate (3) at (-0.5,1);
        \coordinate (4) at (0,1);
        \coordinate (5) at (0.5,0.5);
        \coordinate (6) at (1,0.5);
        \draw (4)--(5);
        \draw (2)--(5);
        \draw (6)--(5);
        \draw(1)--(3);
        \draw (1.25,0.5) circle (0.25);
        \draw(3)[] to [out=30, in=150] (4);
        \draw(3)[] to [out=330, in=210] (4);
        \draw(1)[] to [out=30, in=150] (2);
        \draw(1)[] to [out=330, in=210] (2);
        \foreach \i in {1,2,3,4,5,6}{
            \vertex{\i}
        }
        \draw (0.5,-0.5) node[] {(121)};
    \end{tikzpicture}&
    \begin{tikzpicture}
        \coordinate (1) at (-0.5,0);
        \coordinate (2) at (0,0.5);
        \coordinate (3) at (-0.5,1);
        \coordinate (4) at (0.5,0.5);
        \coordinate (5) at (1,0.5);
        \coordinate (6) at (1.5,0.5);
        \draw (3)--(2);
        \draw (1)--(2);
        \draw (2)--(4);
        \draw (5)--(6);
        \draw(1)[] to [out=120, in=240] (3);
        \draw(1)[] to [out=60, in=300] (3);
        \draw (1.75,0.5) circle (0.25);
        \draw(4)[] to [out=30, in=150] (5);
        \draw(4)[] to [out=330, in=210] (5);
        \foreach \i in {1,2,3,4,5,6}{
            \vertex{\i}
        }
        \draw (0.75,-0.5) node[] {(122)};
    \end{tikzpicture}
 \\
    &\begin{tikzpicture}
        \coordinate (1) at (-0.5,0.5);
        \coordinate (2) at (0,0.25);
        \coordinate (3) at (0.5,0.5);
        \coordinate (4) at (0,0.75);
        \coordinate (5) at (-0.5,1);
        \coordinate (6) at (0.5,1);
        \draw (4)--(2);
        \draw (1)--(2);
        \draw (1)--(4);
        \draw (3)--(4);
        \draw (3)--(2);
        \draw (1)--(5);
        \draw (3)--(6);
        \draw (-0.5,1.25) circle (0.25);
        \draw (0.5,1.25) circle (0.25);
        \foreach \i in {1,2,3,4,5,6}{
            \vertex{\i}
        }
        \draw (0,-0.5) node[] {(202)};
    \end{tikzpicture}\qquad\qquad
    \begin{tikzpicture}
        \coordinate (1) at (0,0.3);
        \coordinate (2) at (1,0.3);
        \coordinate (3) at (0,1);
        \coordinate (4) at (1,1);
        \coordinate (5) at (0,1.5);
        \coordinate (6) at (1,1.5);
        \draw (4)--(2);
        \draw (1)--(3);
        \draw (3)--(4);
        \draw (3)--(5);
        \draw (4)--(6);
        \draw (0,1.75) circle (0.25);
        \draw (1,1.75) circle (0.25);
        \foreach \i in {1,2,3,4,5,6}{
            \vertex{\i}
        }
        \draw (0.5,-0.5) node[] {(212)};
        \draw(1)[] to [out=30, in=150] (2);
        \draw(1)[] to [out=330, in=210] (2);
    \end{tikzpicture}&
\end{tikzcd}\caption{Maximal combinatorial types of genus $4$ curves realizable in $\mathbb F_0$ and $\mathbb F_2$.}\label{fg:genus4_maxtype}
\end{figure}

For each of the combinatorial type we can study the necessary conditions that the edge lengths of corresponding tropical curve have to satisfy in order to be embedded in $\mathbb F_0$ and $\mathbb F_2,$ respectively.
Let us first observe that, when studying an embedded tropical curve, we have the following correspondences between the combinatorial graph and the triangulation realizing the curve:
\begin{align*}\text{cycle}&\leftrightarrow \text{interior point}\\
\text{two adjacent cycles}&\leftrightarrow \text{1-simplex joining two interior points}\\
\text{cycles connected via a bridge}&\leftrightarrow \text{1-simplex separating two interior points.}\end{align*}

Such relations will partially determine the corresponding triangulation and this will allow us to determine necessary edge lengths conditions.
We show this for the combinatorial type (000)A. For all the other combinatorial type in Figure \ref{fg:genus4_maxtype} we refer to Appendix \ref{sc:appendix}.

Let us consider first the embedding in $\mathbb F_0$. The incidence correspondences between the cycles in the combinatorial graph force the triangulation to be as in Figure \ref{fg:(000)F_0}.
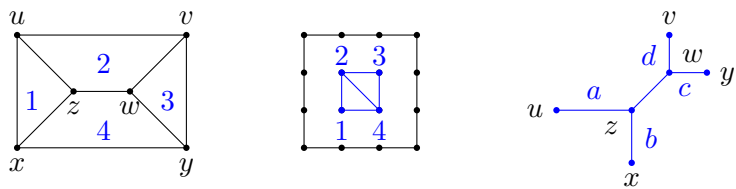
\begin{figure}[h]
\begin{tikzcd}
    \begin{tikzpicture}
        \coordinate (1) at (0,0);
        \coordinate (2) at (0,1.5);
        \coordinate (3) at (0.75,0.75);
        \coordinate (4) at (1.5,0.75);
        \coordinate (5) at (2.25,0);
        \coordinate (6) at (2.25,1.5);
        \draw (1)--(2);
        \draw (1)--(3);
        \draw (1)--(5);
        \draw (3)--(4);
        \draw (3)--(2);
        \draw (4)--(5);
        \draw (4)--(6);
        \draw (5)--(6);
        \draw (6)--(2);
        \foreach \i in {1,2,3,4,5,6}{
            \vertex{\i}
        }
        \draw[blue] (0.2,0.5) node[] {1};
        \draw[blue] (2,0.5) node[] {3};
        \draw[blue] (1.15,1) node[] {2};
        \draw[blue] (1.15,0.1) node[] {4};
        \draw (0,0) node[anchor=north] {$x$};
        \draw (2.25,0) node[anchor=north] {$y$};
        \draw(0,1.5) node[anchor=south] {$u$};
        \draw (2.25,1.5) node[anchor=south] {$v$};
        \draw(0.75,0.3) node[anchor=south] {$z$};
        \draw (1.5,0.3) node[anchor=south] {$w$};
    \end{tikzpicture}&
    \begin{tikzpicture}
        \foreach \i in {0,0.5,1,1.5}{
            \foreach \j in {0,0.5,1,1.5}{
            \vertex{\i,\j}
            }
    }
    \vertexblue{0.5,0.5}\vertexblue{1,0.5}
    \vertexblue{1,1}
    \vertexblue{0.5,1}
    \draw(0,0)--(1.5,0);
    \draw(0,0)--(0,1.5);
    \draw(1.5,1.5)--(0,1.5);
    \draw(1.5,1.5)--(1.5,0);
    \draw[blue] (0.5,0.1) node[] {1};
    \draw[blue] (1,0.1) node[] {4};
    \draw[blue] (0.5,1.1) node[] {2};
    \draw[blue] (1,1.1) node[] {3};
    \draw[blue] (0.5,0.5)--(1,0.5);
    \draw[blue] (0.5,1)--(1,1);
    \draw[blue] (0.5,1)--(1,0.5);
    \draw[blue] (0.5,1)--(0.5,0.5);
    \draw[blue] (1,1)--(1,0.5);
    \end{tikzpicture}&
    \begin{tikzpicture}
    \draw[blue](0.5,0.5)--(1,1);
    \draw[blue](0.5,0.5)--(0.5,-0.2);
    \draw[blue](0.5,0.5)--(-0.5,0.5);
    \draw[blue](1,1)--(1.5,1);
    \draw[blue](1,1)--(1,1.5);
    \vertexblue{-0.5,0.5}\vertexblue{0.5,-0.2}
    \vertexblue{0.5,0.5}
    \vertexblue{1,1}
    \vertexblue{1.5,1}
    \vertexblue{1,1.5}
    \draw[] (0.5,0.5) node[anchor=north east] {$z$};
    \draw[] (1,1) node[anchor=south west] {$w$};
    \draw[] (-0.5,0.5) node[anchor= east] {$u$};
    \draw[] (0.5,-0.2) node[anchor= north] {$x$};
    \draw[] (1.5,1) node[anchor= west] {$y$};
    \draw[] (1,1.5) node[anchor= south] {$v$};
    \draw[blue] (0,0.5) node[anchor= south] {$a$};
    \draw[blue] (0.5,0.1) node[anchor= west] {$b$};
    \draw[blue] (1.2,1) node[anchor= north] {$c$};
    \draw[blue] (1,1.2) node[anchor= east] {$d$};
    \end{tikzpicture}
    \end{tikzcd}\caption{Combinatorial type, partial triangulation of $\mathbb F_0$ and partial embedded tropical curve.}\label{fg:(000)F_0}
    \end{figure}

The lengths of the edges $a,b,c,d$ are non-trivial which forces the edge $zw$ plus either $a$ or $d,$ resp. $b$ or $c$, to have length strictly smaller than that of $uv$, resp. $xy.$

Instead, if we consider the embedding in $\mathbb F_2,$ the triangulation has to be as the one represented in Figure \ref{fg:(000)F_2}.

\begin{figure}[h]
\begin{tikzcd}
    \begin{tikzpicture}
        \foreach \j in {0,0.5,1,1.5,2,2.5,3}{
            \vertex{0,\j}
        }
        \foreach \j in {0,0.5,1,1.5,2}{
            \vertex{0.5,\j}
        }
        \foreach \j in {0,0.5,1}{
            \vertex{1,\j}
        }
        \vertex{1.5,0}
        \draw (0,0)--(1.5,0);
        \draw (0,3)--(1.5,0);
        \draw (0,3)--(0,0);
    \vertexblue{0.5,0.5}\vertexblue{1,0.5}
    \vertexblue{0.5,1.5}
    \vertexblue{0.5,1}
    \draw[blue] (0.5,0.1) node[] {1};
    \draw[blue] (1,0.1) node[] {4};
    \draw[blue] (0.5,1.1) node[anchor=east] {2};
    \draw[blue] (0.5,1.6) node[anchor=east] {3};
    \draw[blue] (0.5,0.5)--(1,0.5);
    \draw[blue] (0.5,1)--(0.5,1.5);
    \draw[blue] (0.5,1)--(1,0.5);
    \draw[blue] (0.5,1)--(0.5,0.5);
    \draw[blue] (0.5,1.5)--(1,0.5);
    \end{tikzpicture}&
    \begin{tikzpicture}
    \draw[blue](0.5,0.5)--(1.5,1.5);
    \draw[blue](0.5,0.5)--(0.5,-0.2);
    \draw[blue](0.5,0.5)--(-0.5,0.5);
    \draw[blue](1.5,1.5)--(0.5,1.5);
    \draw[blue](1.5,1.5)--(2.5,2);
    \vertexblue{-0.5,0.5}\vertexblue{0.5,-0.2}
    \vertexblue{0.5,0.5}
    \vertexblue{1.5,1.5}
    \vertexblue{2.5,2}
    \vertexblue{0.5,1.5}
    \draw[] (0.5,0.5) node[anchor=north east] {$z$};
    \draw[] (1.5,1.5) node[anchor=south] {$w$};
    \draw[] (-0.5,0.5) node[anchor= east] {$u$};
    \draw[] (0.5,-0.2) node[anchor= north] {$x$};
    \draw[] (2.5,2) node[anchor= west] {$y$};
    \draw[] (0.5,1.5) node[anchor= east] {$v$};
    \draw[blue] (0.5,0.1) node[anchor= west] {$b$};
    \draw[blue] (2,1.8) node[anchor= north] {$c$};
    \end{tikzpicture}
    \end{tikzcd}\caption{Partial triangulation of $\mathbb F_2$ and partial embedded tropical curve.}\label{fg:(000)F_2}
    \end{figure}
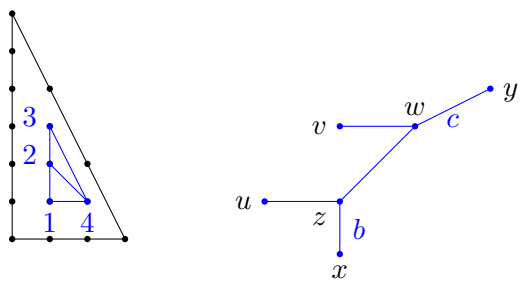
Again, the edges $b,c$ have non-trivial length, therefore $wz+wy<xy$ and $wz+xz<xy$. 
Moreover, the points $u,z$ share the same $y$-coordinates, and similarly the points $v,w$ also have the same $y$-coordinate. 
Then any triangulation completing the partial one represented in Figure \ref{fg:(000)F_2}, will yield edges connecting $u$ to $v$ with direction vector $(\alpha, 1),$ for some $\alpha \in\mathbb Z_{\geq 0}$, therefore $wz=uv$. 

This proves the following.
\begin{lemma}
    Let $\Gamma$ be a tropical curve of genus $4$ of combinatorial type (000)A.
    \begin{itemize}
        \item If $\Gamma$ is realizable in $\mathbb F_0$, then its edge lengths satisfy, up to symmetry:
        \begin{equation}
            \begin{cases}
                wz+vw<uv,\\
                wz+uz<uv,\\
                wz+wy<xy,\\
                wz+xz<xy.
            \end{cases}
        \end{equation}
        \item If $\Gamma$ is realizable in $\mathbb F_2$, then its edge lengths satisfy, up to symmetry:
        \begin{equation}
            \begin{cases}
                wz=uv,\\
                wz+wy<xy,\\
                wz+xz<xy.
            \end{cases}
        \end{equation}
    \end{itemize}
\end{lemma}

As observed in the beginning of Section \ref{ssc: genus3}, the integer $n$ for which a tropical curve of genus $3,4$ and of maximal combinatorial type is realizable in $\mathbb F_n$, satisfies the same conditions \eqref{def:Maroni} that the Maroni invariant of an algebraic curve of the same genus has to satisfy.
Moreover, the above result (and the ones in Appendix \ref{sc:appendix}) also show that, for any of the combinatorial types that we have considered, the edge length conditions for the realizability in $\mathbb F_0$ are not compatible with the ones for the realizability in $\mathbb F_2.$ In other words, an abstract tropical curve of genus $4$ (and $3$) cannot be embedded simultaneously in Hirzebruch surfaces of different degrees.

\begin{definition}\label{def:Maroni}
    Let $\Gamma$ be a tropical curve of genus $3,4$ of maximal combinatorial type that is realizable in $\mathbb F_n$. Then the integer $n$ is called the \emph{tropical Maroni invariant} of $\Gamma$.
\end{definition}

Jensen and Lehmann in \cite{JL} have defined \emph{scrollar invariants}, or more precisely \emph{composite scrollar invariants}, for tropical curves. In the algebraic case, the difference between the second and the first composite scrollar invariants for a trigonal curve determines the Maroni invariant. Therefore one could attempt to define the Maroni invariant for tropical curves via composite scrollar invariants in the same way, as the difference between the second (which is always equal to the algebraic one for gonality $3$) and the first one. More precisely, from \cite{JL}, given a metric graph $\Gamma$ and a divisor $D$ of rank $1$ on $\Gamma,$ the difference between the second and first composite scrollar invariants is 
$\sigma_2(\Gamma,D)-\sigma_1(\Gamma,D)=m+1,$ where $m$ is the smaller integer such that $K_{\Gamma}-mD$ is not effective. Here $K_{\Gamma}$ denotes the canonical divisor on $\Gamma.$ 
It is also a consequence that of \cite[Theorem 1.1]{JL} that if $\Gamma$ is the skeleton of a trigonal algebraic curve $X,$ then the above number defines also an upper bound for the Maroni invariant of $X.$

However, we do not expect such a number to be related to the tropical Maroni invariant that we have defined above. Indeed, the integer $n$ clearly depends on the divisor $D$ on the tropical curve, which, unlike the algebraic case, is not unique: a tropical curve, realizable in some Hirzebruch surface, might admit two non-linearly independent divisors of degree $3$ and rank $1,$ which might define different composite scrollar invariants.

\section{Obstructions for smooth embeddings into $\mathbb{R}^2$}\label{sec-obstructions}

In Proposition \ref{lm:realizability_cover} we have shown that if a tropical curve $\Gamma$ is realizable in $\mathbb F_n,$ then the natural projection from $\mathbb F_n$ to the tropical line induces a degree $3$ well-contracted cover from $\Gamma$, up to tropical modifications.
We would like now to understand the opposite direction of this statemenent, when restricting to genus $g$ curves of maximal combinatorial type, when $g=3,4$.

For instance, if we consider genus $3$ tropical curves, we know from the previous section that those with combinatorial type in Figure \ref{fg:genus3_maxtype} are realizable, up to some edge length conditions. Indeed, each of these curves admits, up to tropical modifications, a well-contracted degree $3$ cover, as in Figure \ref{fig:g3_tropical_covers}.

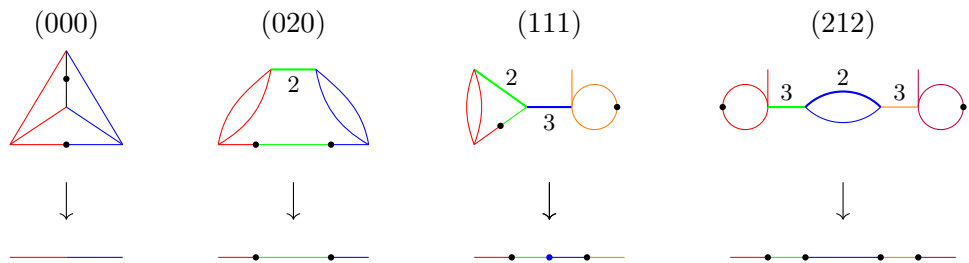
\begin{figure}[h]
        \centering
        \begin{tikzcd}
            \begin{tikzpicture}
                \draw[](0,0)--(0,0.75);\vertex{0,0.375}
                \draw[red](-0.75,-0.5)--(0,0);
                \draw[blue](0.75,-0.5)--(0,0);
                \draw[blue](0.75,-0.5)--(0,0.75);
                \draw[red](-0.75,-0.5)--(0,0.75);
                \draw[red](-0.75,-0.5)--(0,-0.5);
                \draw[blue](0.75,-0.5)--(0,-0.5);
                \vertex{0,-0.5}
                \draw(0,1) node{(000)};
                \draw[->](0,-1)--(0,-1.5);
                \draw[red](-0.75,-2)--(0,-2);
                \draw[blue](0.75,-2)--(0,-2);
            \end{tikzpicture}&
            \begin{tikzpicture}
                \draw[red](-1,-0.5)--(-0.5,-0.5);
                \draw[green](0.5,-0.5)--(-0.5,-0.5);
                \draw[blue](1,-0.5)--(0.5,-0.5);\vertex{-0.5,-0.5}\vertex{0.5,-0.5}
                \draw[thick, green](-0.3,0.5)--(0.3,0.5);
                \draw(0,0.2) node{\footnotesize{2}};
                \draw(0,1) node{(020)};
                \draw[->](0,-1)--(0,-1.5);
                \draw[red](-1,-0.5) to [out=80, in=220](-0.3,0.5);
                \draw[red](-1,-0.5) to [out=30, in=260](-0.3,0.5);
                \draw[blue](1,-0.5) to [out=100, in=320](0.3,0.5);
                \draw[blue](1,-0.5) to [out=150, in=280](0.3,0.5);
                \draw[red](-1,-2)--(-0.5,-2);
                \draw[green](0.5,-2)--(-0.5,-2);
                \draw[blue](1,-2)--(0.5,-2);\vertex{-0.5,-2}\vertex{0.5,-2}
            \end{tikzpicture}&
            \begin{tikzpicture}
                \draw[thick, green](-1,0.5)--(-0.3,0);
                \draw(-0.5,0.3) node{\footnotesize{2}};
                \draw[green](-0.65,-0.25)--(-0.3,0);
                \draw[red](-1,-0.5)--(-0.65,-0.25);
                \vertex{-0.65,-0.25}
                \draw[thick, blue](-0.3,0)--(0.3,0);
                \draw(0,-0.3) node{\footnotesize{3}};
                \draw[orange](0.6,0) circle (0.3);\vertex{0.9,0}
                \draw[orange] (0.3,0)--(0.3,0.5) ;
                \draw[->](0,-1)--(0,-1.5);
                \draw[red](-1,-0.5) to [out=70, in=290](-1,0.5);
                \draw[red](-1,-0.5) to [out=110, in=250](-1,0.5);
                \draw(0,1) node{(111)};
                \draw[->](0,-1)--(0,-1.5);
                \draw[red](-1,-2)--(-0.5,-2);
                \draw[green](0,-2)--(-0.5,-2);
                \draw[blue](0.5,-2)--(0,-2);
                \draw[orange](1,-2)--(0.5,-2);\vertex{-0.5,-2}\vertexblue{0,-2}\vertex{0.5,-2}
                
            \end{tikzpicture}&\begin{tikzpicture}
                \draw[red](-1.3,0) circle (0.3);\draw[red](-1,0)--(-1,0.5);\vertex{-1.6,0}
                \draw[thick, green](-1,0)--(-0.5,0);
                \draw(-0.75,0.1) node{\footnotesize{3}};
                \draw[purple](1.3,0) circle (0.3);\draw[purple](1,0)--(1,0.5);\vertex{1.6,0}
                \draw[orange](1,0)--(0.5,0);
                \draw(0.75,0.1) node{\footnotesize{3}};
                \draw[thick, blue] (-0.5,0) to [out=45, in=135] (0.5,0);
                \draw(0,0.3) node{\footnotesize{2}};
                \draw[blue] (-0.5,0) to [out=315, in=225] (0.5,0);
                \draw(0,1) node{(212)};
                \draw[->](0,-1)--(0,-1.5);
                \draw[red](-1.5,-2)--(-1,-2);
                \draw[green](-0.5,-2)--(-1,-2);
                \draw[blue](-0.5,-2)--(0.5,-2);
                \draw[orange](0.5,-2)--(1,-2);
                \draw[purple](1.5,-2)--(1,-2);
                \vertex{-1,-2}\vertex{-0.5,-2}\vertex{0.5,-2}\vertex{1,-2}
            \end{tikzpicture}
        \end{tikzcd}
        \caption{Examples of degree $3$ well-contracted covers.}
        \label{fig:g3_tropical_covers}
    \end{figure}

Let us instead consider the only maximal combinatorial type of a curve of genus $3,$ not realizable in $\mathbb R^2,$ that is the graph defined by three disjoint loops, each with an edge and the three edges meeting at a common vertex.
One can check that no tropical curve, whose canonical model has such a combinatorial type, admits a degree $3$ well-contracted cover. This will be made more precise later.

Also genus $4$ tropical curves of maximal combinatorial type whose underlying graph is planar follow the same behavior: the ones which are not realizable in $\mathbb F_n,$ do not admit a degree 3 well-contracted cover, up to tropical modifications.

\begin{theorem}\label{th:realizability_genus_3_4}
   Let $\Gamma$ be a tropical curve of genus $g=3,4$ of planar maximal combinatorial type. Then (a tropical modification of) $\Gamma$ admits a degree $3$ well-contracted cover if and only if $\Gamma$ is realizable in $\FF_n$, up to contractions not changing the combinatorial type.
\end{theorem}

\begin{proof}
    One direction of the statement is precisely Proposition \ref{lm:realizability_cover}. 
    In order to complete the proof is then sufficient to show that the tropical curves with maximal combinatorial types which are never realizable in $\FF_n$ cannot admit a well-contracted cover of degree 3, up to tropical modifications.
    From the result in Section \ref{sec-maroni}, the tropical curves which are not realizable in $\FF_n$ have combinatorial type denoted by (303) for genus $3$ and (213), (303), (314), (405) for genus $4$. 
    The combinatorial types (303) (for genus $3$), (213),(314) and (405) contain a forbidden pattern, called a \emph{sprawling node}, and we will see later in Lemma \ref{lm:Forbidden1} that containing such a pattern represents an obstruction to admitting well-contracted degree $3$ covers. Therefore we only need to check that a genus $4$ curve with combinatorial type (303) does not admit a well-contracted degree $3$ cover. This follows because the only degree $3$ tropical cover with such a graph (or a tropical modification of it) as a source is the one represented in Figure \ref{fig:(303)_tropical_cover}, which is not well-contracted. 
    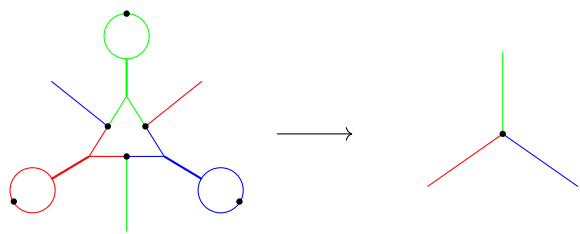
\begin{figure}[h]
            \centering
            \begin{tikzpicture}
                \draw[red](-0.5,0)--(0,0);
                \draw[blue](0,0)--(0.5,0);
                \draw[green](0,0)--(0,-1);
                \vertex{0,0}
                \draw[red](-0.5,0)--(-0.25,0.4);
                \draw[green](-0.25,0.4)--(0,0.8);
                \draw[blue](-0.25,0.4)--(-1,1);
                \draw[red](0.25,0.4)--(1,1);
                \vertex{-0.25,0.4}
                \draw[green](0.25,0.4)--(0,0.8);
                \draw[blue](0.25,0.4)--(0.5,0);\vertex{0.25,0.4}
                \draw[blue, thick](0.5,0)--(1,-0.3);
                \draw[red,thick](-0.5,0)--(-1,-0.3);
                \draw[green, thick](0,1.3)--(0,0.8);
                \draw[green](0,1.6) circle (0.3);
                \draw[red](-1.25,-0.45) circle (0.3);
                \draw[blue](1.25,-0.45) circle (0.3);
                \vertex{0,1.9}\vertex{-1.5,-0.6}\vertex{1.5,-0.6}
                \draw[->](2,0.3)--(3,0.3);
                \draw[green](5,0.3)--(5,1.4);
                \draw[blue](5,0.3)--(6,-0.4);
                \draw[red](5,0.3)--(4,-0.4);\vertex{5,0.3}
            \end{tikzpicture}
            \caption{A degree $3$ tropical cover, which is not well-contracted.}
            \label{fig:(303)_tropical_cover}
        \end{figure}
\end{proof}

Let us observe that in the above theorem we require the underlying graph of the tropical curves to be planar. This is because of the unique non-planar graph of genus $4$, whose combinatorial type is (000)B, as denoted in \cite{BJMS15}. A tropical curve of combinatorial type (000)B does admit a well-contracted cover of degree $3$, however it is not realizable in $\mathbb F_n$. Nonetheless, we will see later in Subsection \ref{ssc:unfolding(000)B} that we can embed it in $\mathbb F_0$ as a non-smooth tropical curve.

Here we continue to focus on tropical curves whose underlying graph is planar. The proof of Theorem \ref{th:realizability_genus_3_4} relies on the fact that some non-realizable maximal combinatorial types of genus $g=3,4$ contain a \emph{sprawling node} as in \cite[Figure 1]{JT21}.

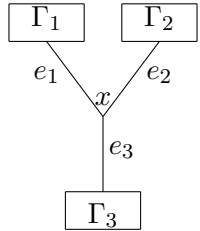
\begin{figure}[h]
\begin{tikzcd}
    \begin{tikzpicture}
        \coordinate (1) at (-0.75,1);
        \coordinate (2) at (0.75,1);
        \coordinate (3) at (0,-1);
        \coordinate (4) at (0,0);
        \draw (1)--(4);
        \draw (2)--(4);
        \draw (3)--(4);
        \draw(-1.25,1)--(-0.25,1);
        \draw(-1.25,1.5)--(-0.25,1.5);
        \draw(-0.25,1)--(-0.25,1.5);
        \draw(-1.25,1.5)--(-1.25,1);

        \draw(1.25,1)--(0.25,1);
        \draw(1.25,1.5)--(0.25,1.5);
        \draw(0.25,1)--(0.25,1.5);
        \draw(1.25,1.5)--(1.25,1);

        \draw (-0.5,-1)--(0.5,-1);
        \draw (-0.5,-1.5)--(0.5,-1.5);
        \draw (-0.5,-1.5)--(-0.5,-1);
        \draw (0.5,-1.5)--(0.5,-1);

        \draw[] (1) node[anchor=south] {$\Gamma_1$};
        \draw[] (2) node[anchor=south] {$\Gamma_2$};
        \draw[] (3) node[anchor=north] {$\Gamma_3$};
        \draw[] (4) node[anchor=south] {$x$};
        \draw[] (-0.75,0.5) node {$e_1$};
        \draw[] (0.75,0.5) node {$e_2$};
        \draw[] (0.25,-0.5) node {$e_3$};
    \end{tikzpicture}
    \end{tikzcd}\caption{A graph with a sprawling node.}\label{fg:Forbidden}
    \end{figure}

We will now complete the proof of Theorem \ref{th:realizability_genus_3_4} by showing that graphs containing a sprawling node cannot admit a degree $3$ well-contracted cover. The following result also holds in higher genus.

\begin{lemma}\label{lm:Forbidden1}
    Let $\Gamma$ be a tropical curve of genus $g\geq 3$ with a forbidden pattern which is a sprawling node. Then $\Gamma$ doesn't admit a degree $3$ well-contracted cover, up to tropical modifications. 
\end{lemma}

\begin{proof}
    Let $e_1,e_2,e_3$ denote the three edges defining the sprawling  node in $\Gamma,$ as in Figure \ref{fg:Forbidden}, and $x=\overline e_1\cap \overline e_2\cap \overline e_3.$
    Assume by contradiction that $\Gamma$ admits a non-degenerate degree $3$ tropical cover to a path $T$. Then either one of the edges $e_i$ is contracted or at least two of them have the same image through $\varphi.$ 
    
    If $e_i$ is contracted for some $i\in\{1,2,3\},$ then, since $\Gamma_i$ has positive genus, the restriction of $\varphi$ to $\Gamma_i$ has to have degree at least $2$. This would then force $\varphi$ to restrict over some $\Gamma_j$ with $j\neq i$ to either a contraction or a degree $1$ morphism, both of which are not possible.
    
    If instead $\varphi(e_i)=e\in E(T)$ for $i=1,2$ (and eventually $i=3$), then the subcurves $\Gamma_i$ also admit a well-contracted cover and $\varphi$ restricted to (at least) one of $\Gamma_i$ would be a degree 1 morphism to a sub-path in $T$, giving again a contradiction.
    
\end{proof}

Let us also consider the other two forbidden patterns in \cite[Figure 1]{JT21}, which appear for graphs of genus $g\geq5$. The above result indeed generalizes for such graphs: metric graphs with such forbidden patterns do not admit degree $3$ well-contracted covers. 
\begin{lemma}
    Let $\Gamma$ be a tropical curve with a forbidden pattern which is a \emph{crowded graph}. Then $\Gamma$ doesn't admit a degree 3 well-contracted cover, up to tropical modifications.
\end{lemma}

\begin{proof}
    Assume by contradiction that there exists a degree 3 well-contracted cover, hence a non-degenerate harmonic morphism of the same degree to a path, up to tropical modification. As recalled in Section \ref{ssc:def_gonality_and_relations}, such a cover determines a degree $3$ divisor $D$ of rank $1.$   Then we may assume $D=v_1+y_1+y_2$ for some $y_i\in\Gamma$ and $v_1$ as in Figure \ref{fg:Forbidden1}.
    
\begin{figure}[h]
\begin{tikzcd}
    \begin{tikzpicture}
        \coordinate (1) at (0,0);
        \coordinate (2) at (-1,1);
        \coordinate (3) at (0,1);
        \coordinate (4) at (1,1);
        \coordinate (22) at (-1,2);
        \coordinate (33) at (0,2);
        \coordinate (44) at (1,2);
        \coordinate (11) at (0,3);
        \draw (1)--(2);
        \draw (1)--(3);
        \draw (1)--(4);
        \draw (11)--(22);
        \draw (11)--(33);
        \draw (11)--(44);
        \draw(-1.25,1)--(-0.75,1);
        \draw(-0.25,1)--(0.25,1);
        \draw(1.25,1)--(0.75,1);
        \draw(-1.25,2)--(-0.75,2);
        \draw(-0.25,2)--(0.25,2);
        \draw(1.25,2)--(0.75,2);
        \draw(-1.25,2)--(-1.25,1);\draw(-0.75,1)--(-0.75,2);
        \draw(-0.25,2)--(-0.25,1);\draw(0.25,1)--(0.25,2);
        \draw(1.25,2)--(1.25,1);\draw(0.75,1)--(0.75,2);
        \draw[] (1) node[anchor=north] {$v_1$};
        \draw[] (11) node[anchor=south] {$v_2$};
        \draw[] (2) node[anchor=south] {$\Gamma_1$};
        \draw[] (3) node[anchor=south] {$\Gamma_2$};
        \draw[] (4) node[anchor=south] {$\Gamma_3$};
        \draw[] (-0.3,0.5) node {$e_1$};
        \draw[] (0.25,0.5) node {$e_2$};
        \draw[] (0.85,0.5) node {$e_3$};
        \draw[] (-0.3,2.4) node {$e_1'$};
        \draw[] (0.25,2.4) node {$e_2'$};
        \draw[] (0.85,2.4) node {$e_3'$};
    \end{tikzpicture}
    \end{tikzcd}\caption{A crowded graph.}\label{fg:Forbidden1}
    \end{figure}
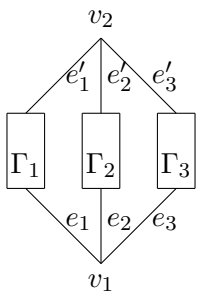

    In the following we refer to vertices and edges of the canonical model of $\Gamma.$ Without loss of generality we also assume no two points in the support of $D$ are in the interior of the same edge. 
    Notice also that $\Gamma$ is bridgeless.
    Then, using Dhar's burning algorithm, the points $y_i$ must identify one or two edges which form with some edge incident to $v_1$, say $e_1$, a $k$-edge cut, with $k=2,3.$ 
    We have the following possibilities.
    \begin{enumerate}
    \item Suppose $k=3$, then we have the following possibilities.
    \begin{enumerate}
        \item The points $y_1,y_2$ are contained in a $2$-edge cut of $\Gamma_1$. Starting a fire for instance in $\Gamma_2$ burns $v_1$ and the fire spreads in $\Gamma_1,$ stopping at $y_1,y_2.$ However the fire spreads also through $v_2$, reaching $y_1,y_2$ also from the other direction and thus burning the entire graph.

        \item If $y_2\in e_2\cup e_2'$, $y_3\in e_3\cup e_3'$, consider the rational function with slope $-1$ from the three points, along the $3$-edge cut, until one of the points reaches a vertex in some $\Gamma_i$ (up to linear equivalence this can always happen, otherwise we are in case (2) which will be considered later). Start a fire at any other point of $\Gamma_i.$  this burns first the vertex in $\Gamma_i$ and then the other two by connectivity.  
        \item Finally, if for some $j\in\{2,3\},$ $y_j$ lies in a bridge of $\Gamma_j$, then the previous argument applies, by replacing the edges among $e_2,e_2',e_3,e_3'$ forming the $3$-edge cut, with the bridge in $\Gamma_j$ and $\Gamma_j$ with one of the connected component in $\Gamma_j$ with the bridge removed.
    \end{enumerate}
    \item If $k=2$, then $y_2\in e$; $e\in\{e_1',b\}$, where $b$ is a bridge in $\Gamma_1.$ Starting Dhar's burning algorithm from any point in $\Gamma\setminus\{e_1,e\}\cup\Gamma_1$ would then burn the graph unless $y_3=v_1$ or it belongs to a edge $e'\subset\Gamma_1$ which forms with either $e_1$ or $e$ a $2$-edge cut. However in this case we could consider the rational function with slope $-1$ from any pair of point identifying a $2$-edge cut, toward the opposite direction of $v_1,v_2,$ until the first one reaches a vertex. Starting again Dhar's burning form any other vertex in $\Gamma_1$ then would burn the graph. 
        Notice that the graph doesn't burn only if $\Gamma_1$ has no other vertices, meaning that $\Gamma_1$ is made of parallel edges, and the same argument can be repeated for $\Gamma_2,\Gamma_3$. In particular the divisor must be of the form $2v_1+v_2$. Then the morphism $\varphi$, which induces such a divisor (via pull-back of any point in the metric tree), is such that any two incident edges to the same vertex $v_i$ cannot be identified. In particular the target tree cannot be a path.
    \end{enumerate}
\end{proof}

\begin{lemma}
    Let $\Gamma$ be a tropical curve with a forbidden pattern which is a \emph{TIE-fighter graph}. Then $\Gamma$ doesn't admit a degree 3 well-contracted cover, up to tropical modifications.
\end{lemma}

\begin{proof}
Let us assume again by contradiction that there exists a degree $3$ well-contracted cover $\varphi:\Gamma\to \Gamma',$ with $\Gamma$ a TIE-fighter graph, as in Figure \ref{fg:Forbidden2}.
Using an argument similar to the one in the proof of Lemma \ref{lm:Forbidden1}, one can check that the two bridges $e_1,e_4$ are sent to distinct edges in $\Gamma'$.
This in particular implies that any edge in the connected component of $\Gamma\setminus (e_1\cup e_4)$ containing $\Gamma_2$ and $\Gamma_3$ cannot be sent to a leaf-edge in $\Gamma'.$
\begin{figure}[h]
\begin{tikzcd}
    \begin{tikzpicture}
        \draw (0,-0.5)--(0,0.5);
        \draw (0,-0.5)--(-0.5,-0.5);
        \draw (0,0.5)--(-0.5,0.5);
        \draw (-0.5,-0.5)--(-0.5,0.5);
        \draw (0,0)--(1,0);
        \draw (1,0)--(1.5,0.5);
        \draw (1,0)--(1.5,-0.5);
        \draw (1.5,-0.5)--(1.5,-0.75);\draw (1.5,-0.5)--(1.5,-0.25);
        \draw (1.5,-0.75)--(2.5,-0.75);\draw (2.5,-0.25)--(1.5,-0.25);
        \draw (2.5,-0.5)--(2.5,-0.75);\draw (2.5,-0.5)--(2.5,-0.25);
        \draw (1.5,0.5)--(1.5,0.75);\draw (1.5,0.5)--(1.5,0.25);
        \draw (1.5,0.75)--(2.5,0.75);\draw (2.5,0.25)--(1.5,0.25);
        \draw (2.5,0.5)--(2.5,0.75);\draw (2.5,0.5)--(2.5,0.25);
        \draw (3,0)--(2.5,0.5);
        \draw (3,0)--(2.5,-0.5);
        \draw(3,0)--(4,0);
        \draw (4,-0.5)--(4,0.5);
        \draw (4,-0.5)--(4.5,-0.5);
        \draw (4,0.5)--(4.5,0.5);
        \draw (4.5,-0.5)--(4.5,0.5);
        \draw[] (-0.25,0) node {$\Gamma_1$};
        \draw[] (4.25,0) node {$\Gamma_4$};
        \draw[] (0.5,0.1) node {$e_1$};
        \draw[] (3.5,0.1) node {$e_4$};
        \draw[] (1.2,0.4) node {$e_2$};
        \draw[] (1.2,-0.5) node {$e_3$};
        \draw[] (2.8,0.4) node{$e_2'$};
        \draw[] (2.8,-0.5) node {$e_3'$};
        \draw[] (2,0.4) node{$\Gamma_2$};
        \draw[] (2,-0.6) node {$\Gamma_3$};
    \end{tikzpicture}
    \end{tikzcd}\caption{A TIE-fighter graph and the graph obtained by contracting all bridges.}\label{fg:Forbidden2}
    \end{figure}
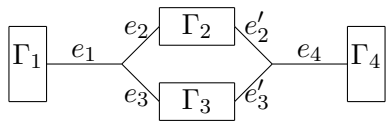

The morphism $\varphi$ determines a divisor $D$ of degree $3$ and rank at least $1,$ therefore we may write $D=v+y_2+y_3$ with $v=\Gamma_1\cap\overline{e}_1$.
If $y_2,y_3\notin\Gamma_1$, then they need to determine a $2$-edge cut of the graph. We may further assume that none of them is on a bridge, and that at least one is a vertex. Starting a fire from $\Gamma_1\setminus\{v\}$ burns the whole graph unless $y_2,y_3$ are contained in the closure of $e_2,e_3,$ respectively. Up to linear equivalence, then we may assume that at least one of them is then in $\Gamma_1$ and we refer to the following.

If $y_2,y_3$ in $\Gamma_1$, then consider for instance the case where $D\sim 3v$ (if the multiplicity is only $2$ then the argument follows even more easily). In particular we can consider the linear equivalent divisor moving $2v$ along $e_1$ and $e_2$ and $v$ on $e_1$ and $e_3$, until at least two points (with multiplicity) reach a vertex in $\Gamma_2$ or $\Gamma_3$. If they all reach a vertex, then starting the fire for instance at $\Gamma_4$ would burn the entire graph, unless the point $v$ moving on $e_3$ reaches the vertex $\Gamma_3\cap\overline{e}_3$ only after $2v$, moved first along $e_2$, does. 
This in particular means that there is some point $w\in\overline{e}_3$ such that $D\sim 3w.$ Since $e_3$ is not a bridge, then $\varphi(w)$ is a leaf, which contradicts $\Gamma'$ being a path.
\end{proof}

\section{Unfolding contracted features of a planar embedding of a trigonal $\Gamma$}\label{sec-unfolding}

So far we have focused on smooth plane tropical curves and their images under the forgetful map which are abstract tropical curves of maximal combinatorial type.

Assume $C$ is an abstract curve defined over $K$, and $C$ comes embedded into a toric surface. Assume its embedded tropicalization is not smooth. There is still a map from the abstract tropicalization $\Gamma$ to the embedded tropicalization (which in the smooth case can be obtained by sending $\Gamma$ identically onto the part of the plane tropical curve obtained by iteratively shrinking ends and leaf edges), but this map can now contract subgraphs of $\Gamma$, so that the embedded tropicalization does not necessarily reflect important features of the abstract tropicalization anymore.
For that reason, extending the forgetful map to non-smooth tropical plane curves does not make sense.

However, we can use a linear re-embedding of torus in such a way that the tropicalization of the newly embedded curve does not contract parts of $\Gamma$ anymore. This can be called an \emph{unfolding} of the skeleton.

The theoretical background for this is laid in \cite{BPR11a}, where a tropicalization that reflects all important features is called \emph{faithful}.

The concept of using linear re-embeddings to unfold was studied in various settings, e.g.\ in \cite{IMS09, BLdM11, Mi07, CM14}. Notice that this concept works for a \emph{tropicalized curve}, that is we always have preimages under tropicalization that are compatible with our tropical curves.

In this paper, we want to illustrate the power of re-embeddings in two settings:
\begin{enumerate}
    \item We show that the non-planar graph (000)B of genus $4$ can be the abstract tropicalization of a re-embedded plane tropical curve. 
    \item We show, exemplarily for the graph (000)A, that any abstract tropical curve $\Gamma$ of this type together with a trigonal morphism to $\mathbb{R}$, can be embedded as a re-embedded tropical plane curve in such a way that the projection to $\mathbb{R}$ yields the trigonal morphism.
\end{enumerate}

\subsection{Unfolding an abstract tropical curve of type (000)B}\label{ssc:unfolding(000)B}

\begin{figure}
    \centering

\tikzset{every picture/.style={line width=0.75pt}} 

\begin{tikzpicture}[x=0.75pt,y=0.75pt,yscale=-1,xscale=1]

\draw    (260,60) -- (250,70) ;
\draw    (200,110) -- (210,100) ;
\draw    (240,90) -- (250,80) ;
\draw    (220,90) -- (220,70) ;
\draw [color={rgb, 255:red, 208; green, 2; blue, 27 }  ,draw opacity=1 ]   (220,90) -- (210,90) ;
\draw    (220,70) -- (210,60) ;
\draw    (210,60) -- (200,60) ;
\draw    (200,130) -- (200,110) ;
\draw [color={rgb, 255:red, 65; green, 117; blue, 5 }  ,draw opacity=1 ]   (240,90) -- (220,90) ;
\draw    (200,80) -- (200,60) ;
\draw [color={rgb, 255:red, 74; green, 144; blue, 226 }  ,draw opacity=1 ]   (220,100) -- (220,90) ;
\draw    (240,110) -- (250,120) ;
\draw    (250,70) -- (220,70) ;
\draw    (200,60) -- (190,50) ;
\draw    (210,60) -- (210,40) ;
\draw    (250,80) -- (250,70) ;
\draw    (190,50) -- (180,50) ;
\draw    (190,40) -- (190,50) ;
\draw    (200,80) -- (180,80) ;
\draw    (210,90) -- (200,80) ;
\draw [color={rgb, 255:red, 208; green, 2; blue, 27 }  ,draw opacity=1 ]   (220,100) -- (210,100) ;
\draw [color={rgb, 255:red, 208; green, 2; blue, 27 }  ,draw opacity=1 ]   (210,90) -- (210,100) ;
\draw [color={rgb, 255:red, 74; green, 144; blue, 226 }  ,draw opacity=1 ]   (230,110) -- (220,100) ;
\draw [color={rgb, 255:red, 74; green, 144; blue, 226 }  ,draw opacity=1 ]   (240,110) -- (230,110) ;
\draw [color={rgb, 255:red, 74; green, 144; blue, 226 }  ,draw opacity=1 ]   (240,90) -- (240,110) ;
\draw    (270,60) -- (260,60) ;
\draw    (260,40) -- (260,60) ;
\draw    (270,80) -- (250,80) ;
\draw    (200,110) -- (180,110) ;
\draw    (230,130) -- (230,110) ;
\draw    (250,130) -- (250,120) ;
\draw    (270,120) -- (250,120) ;
\draw    (310,80) -- (340,50) ;
\draw    (340,80) -- (340,50) ;
\draw    (370,80) -- (340,50) ;
\draw [color={rgb, 255:red, 208; green, 2; blue, 27 }  ,draw opacity=1 ]   (340,110) -- (310,80) ;
\draw [color={rgb, 255:red, 74; green, 144; blue, 226 }  ,draw opacity=1 ]   (340,110) -- (340,80) ;
\draw [color={rgb, 255:red, 189; green, 16; blue, 224 }  ,draw opacity=1 ]   (380,100) -- (340,80) ;
\draw [color={rgb, 255:red, 74; green, 144; blue, 226 }  ,draw opacity=1 ]   (340,110) -- (370,80) ;
\draw [color={rgb, 255:red, 208; green, 2; blue, 27 }  ,draw opacity=1 ]   (380,100) -- (310,80) ;
\draw [color={rgb, 255:red, 65; green, 117; blue, 5 }  ,draw opacity=1 ]   (380,100) -- (370,80) ;
\draw [color={rgb, 255:red, 208; green, 2; blue, 27 }  ,draw opacity=1 ]   (310,130) -- (300,120) ;
\draw [color={rgb, 255:red, 189; green, 16; blue, 224 }  ,draw opacity=1 ]   (310,130) -- (310,120) ;
\draw [color={rgb, 255:red, 65; green, 117; blue, 5 }  ,draw opacity=1 ]   (320,140) -- (310,130) ;
\draw [color={rgb, 255:red, 65; green, 117; blue, 5 }  ,draw opacity=1 ]   (330,120) -- (320,140) ;
\draw    (320,150) -- (320,140) ;

\end{tikzpicture}

    \caption{Unfolding a graph type (000)B.}
    \label{fig-unfolding000B}
\end{figure}
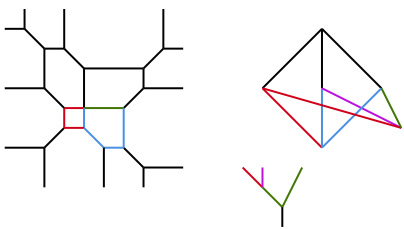

Consider the non-smooth tropical plane curve in $\mathbb{F}_0$ depicted in Figure \ref{fig-unfolding000B}. By Theorem 3.2 \cite{CM14}, we can re-embed the curve in such a way that we can unfold edges of arbitrary length separating the edges that form a crossing dual to a parallelogram. The picture shows on the bottom the new piece of the tropical plane curve that appears in the modified plane we obtain by re-embedding: it is a tropical plane consisting of three half-planes which are glued along a line. In one cell, we see the left part of the tropical plane curve (everything to the left of the crossing), in a second cell, we see the right part, and in the new, third, cell, we see the new piece. That means, the red and the green edge passing through the crossing are lowered into the new cell, and the pink edge appears as a new edge. On the right, the picture shows the abstract tropical curve we obtain when iteratively shrinking ends and leaf edges. It is remarkable that we can obtain a non-planar graph as abstract tropical curve underlying a re-embedding of a plane tropical curve. Similar examples can be found e.g.\ in \cite{Kra13}.

\subsection{Unfolding a trigonal morphism of an abstract tropical curve of type (000)A}
Assume $\Gamma$ is an abstract tropical curve of type (000)A. Let $f:\Gamma\rightarrow \mathbb{R}$ be a trigonal morphism for $\Gamma$. 
Start from a non-smooth tropical plane curve in $\FF_n$  such that iteratively shrinking ends and leaf edges yields $\tilde{\Gamma}$, where $\tilde{\Gamma}$ can be obtained from $\Gamma$ by contracting edges.
We demonstrate how we can use linear re-embeddings to unfold $\Gamma$ from the non-smooth tropical plane curve, in such a way that the  contraction $\tilde{\Gamma}\rightarrow \Gamma$ composed with the projection from $\FF_n$ yield the trigonal morphism of $\Gamma$.

For type (000)A, there exist $6$ combinatorial types of trigonal morphisms (see Figure \ref{fig:sixmorphisms}).

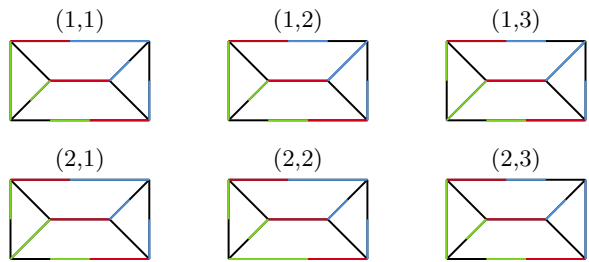
\begin{figure}[h]
    \centering
\tikzset{every picture/.style={line width=0.75pt}} 

\begin{tikzpicture}[x=0.75pt,y=0.75pt,yscale=-1,xscale=1]

\draw (165,50) node[][font=\footnotesize]{(1,1)};
\draw (275,50) node[][font=\footnotesize]{(1,2)};
\draw (385,50) node[][font=\footnotesize]{(1,3)};
\draw (165,120) node[][font=\footnotesize]{(2,1)};
\draw (275,120) node[][font=\footnotesize]{(2,2)};
\draw (385,120) node[][font=\footnotesize]{(2,3)};
\draw    (130,60) -- (200,60) ;
\draw    (130,60) -- (150,80) ;
\draw    (180,80) -- (200,100) ;
\draw    (180,80) -- (200,60) ;
\draw    (130,100) -- (150,80) ;
\draw    (130,100) -- (200,100) ;
\draw    (150,80) -- (180,80) ;
\draw    (200,60) -- (200,100) ;
\draw    (130,60) -- (130,100) ;
\draw    (240,60) -- (310,60) ;
\draw    (240,60) -- (260,80) ;
\draw    (290,80) -- (310,100) ;
\draw    (290,80) -- (310,60) ;
\draw    (240,100) -- (260,80) ;
\draw    (240,100) -- (310,100) ;
\draw    (260,80) -- (290,80) ;
\draw    (310,60) -- (310,100) ;
\draw    (240,60) -- (240,100) ;
\draw    (350,60) -- (420,60) ;
\draw    (350,60) -- (370,80) ;
\draw    (400,80) -- (420,100) ;
\draw    (400,80) -- (420,60) ;
\draw    (350,100) -- (370,80) ;
\draw    (350,100) -- (420,100) ;
\draw    (370,80) -- (400,80) ;
\draw    (420,60) -- (420,100) ;
\draw    (350,60) -- (350,100) ;
\draw    (130,120+10) -- (200,120+10) ;
\draw    (130,120+10) -- (150,140+10) ;
\draw    (180,140+10) -- (200,160+10) ;
\draw    (180,140+10) -- (200,120+10) ;
\draw    (130,160+10) -- (150,140+10) ;
\draw    (130,160+10) -- (200,160+10) ;
\draw    (150,140+10) -- (180,140+10) ;
\draw    (200,120+10) -- (200,160+10) ;
\draw    (130,120+10) -- (130,160+10) ;
\draw    (240,120+10) -- (310,120+10) ;
\draw    (240,120+10) -- (260,140+10) ;
\draw    (290,140+10) -- (310,160+10) ;
\draw    (290,140+10) -- (310,120+10) ;
\draw    (240,160+10) -- (260,140+10) ;
\draw    (240,160+10) -- (310,160+10) ;
\draw    (260,140+10) -- (290,140+10) ;
\draw    (310,120+10) -- (310,160+10) ;
\draw    (240,120+10) -- (240,160+10) ;
\draw    (350,120+10) -- (420,120+10) ;
\draw    (350,120+10) -- (370,140+10) ;
\draw    (400,140+10) -- (420,160+10) ;
\draw    (400,140+10) -- (420,120+10) ;
\draw    (350,160+10) -- (370,140+10) ;
\draw    (350,160+10) -- (420,160+10) ;
\draw    (370,140+10) -- (400,140+10) ;
\draw    (420,120+10) -- (420,160+10) ;
\draw    (350,120+10) -- (350,160+10) ;
\draw [color={rgb, 255:red, 208; green, 2; blue, 27 }  ,draw opacity=1 ]   (150,80) -- (180,80) ;
\draw [color={rgb, 255:red, 208; green, 2; blue, 27 }  ,draw opacity=1 ]   (170,100) -- (200,100) ;
\draw [color={rgb, 255:red, 208; green, 2; blue, 27 }  ,draw opacity=1 ]   (130,60) -- (160,60) ;
\draw [color={rgb, 255:red, 208; green, 2; blue, 27 }  ,draw opacity=1 ]   (260,80) -- (290,80) ;
\draw [color={rgb, 255:red, 208; green, 2; blue, 27 }  ,draw opacity=1 ]   (240,60) -- (270,60) ;
\draw [color={rgb, 255:red, 208; green, 2; blue, 27 }  ,draw opacity=1 ]   (280,100) -- (310,100) ;
\draw [color={rgb, 255:red, 208; green, 2; blue, 27 }  ,draw opacity=1 ]   (350,60) -- (380,60) ;
\draw [color={rgb, 255:red, 208; green, 2; blue, 27 }  ,draw opacity=1 ]   (370,80) -- (400,80) ;
\draw [color={rgb, 255:red, 208; green, 2; blue, 27 }  ,draw opacity=1 ]   (390,100) -- (420,100) ;
\draw [color={rgb, 255:red, 208; green, 2; blue, 27 }  ,draw opacity=1 ]   (130,120+10) -- (160,120+10) ;
\draw [color={rgb, 255:red, 208; green, 2; blue, 27 }  ,draw opacity=1 ]   (150,140+10) -- (180,140+10) ;
\draw [color={rgb, 255:red, 208; green, 2; blue, 27 }  ,draw opacity=1 ]   (170,160+10) -- (200,160+10) ;
\draw [color={rgb, 255:red, 208; green, 2; blue, 27 }  ,draw opacity=1 ]   (240,120+10) -- (270,120+10) ;
\draw [color={rgb, 255:red, 208; green, 2; blue, 27 }  ,draw opacity=1 ]   (260,140+10) -- (290,140+10) ;
\draw [color={rgb, 255:red, 208; green, 2; blue, 27 }  ,draw opacity=1 ]   (280,160+10) -- (310,160+10) ;
\draw [color={rgb, 255:red, 208; green, 2; blue, 27 }  ,draw opacity=1 ]   (350,120+10) -- (380,120+10) ;
\draw [color={rgb, 255:red, 208; green, 2; blue, 27 }  ,draw opacity=1 ]   (370,140+10) -- (400,140+10) ;
\draw [color={rgb, 255:red, 208; green, 2; blue, 27 }  ,draw opacity=1 ]   (390,160+10) -- (420,160+10) ;
\draw [color={rgb, 255:red, 74; green, 144; blue, 226 }  ,draw opacity=1 ]   (160,60) -- (200,60) ;
\draw [color={rgb, 255:red, 74; green, 144; blue, 226 }  ,draw opacity=1 ]   (180,80) -- (190,70) ;
\draw [color={rgb, 255:red, 74; green, 144; blue, 226 }  ,draw opacity=1 ]   (200,100) -- (200,80) ;
\draw [color={rgb, 255:red, 74; green, 144; blue, 226 }  ,draw opacity=1 ]   (200,160+10) -- (200,140+10) ;
\draw [color={rgb, 255:red, 74; green, 144; blue, 226 }  ,draw opacity=1 ]   (310,100) -- (310,80) ;
\draw [color={rgb, 255:red, 74; green, 144; blue, 226 }  ,draw opacity=1 ]   (310,160+10) -- (310,140+10) ;
\draw [color={rgb, 255:red, 74; green, 144; blue, 226 }  ,draw opacity=1 ]   (420,100) -- (420,80) ;
\draw [color={rgb, 255:red, 74; green, 144; blue, 226 }  ,draw opacity=1 ]   (420,160+10) -- (420,120+10) ;
\draw [color={rgb, 255:red, 74; green, 144; blue, 226 }  ,draw opacity=1 ]   (180,140+10) -- (190,130+10) ;
\draw [color={rgb, 255:red, 74; green, 144; blue, 226 }  ,draw opacity=1 ]   (290,80) -- (310,60) ;
\draw [color={rgb, 255:red, 74; green, 144; blue, 226 }  ,draw opacity=1 ]   (290,140+10) -- (300,130+10) ;
\draw [color={rgb, 255:red, 74; green, 144; blue, 226 }  ,draw opacity=1 ]   (400,80) -- (420,60) ;
\draw [color={rgb, 255:red, 74; green, 144; blue, 226 }  ,draw opacity=1 ]   (400,140+10) -- (410,130+10) ;
\draw [color={rgb, 255:red, 74; green, 144; blue, 226 }  ,draw opacity=1 ]   (160,120+10) -- (200,120+10) ;
\draw [color={rgb, 255:red, 74; green, 144; blue, 226 }  ,draw opacity=1 ]   (270,60) -- (290,60) ;
\draw [color={rgb, 255:red, 74; green, 144; blue, 226 }  ,draw opacity=1 ]   (270,120+10) -- (310,120+10) ;
\draw [color={rgb, 255:red, 74; green, 144; blue, 226 }  ,draw opacity=1 ]   (380,60) -- (400,60) ;
\draw [color={rgb, 255:red, 74; green, 144; blue, 226 }  ,draw opacity=1 ]   (380,120+10) -- (400,120+10) ;
\draw [color={rgb, 255:red, 126; green, 211; blue, 33 }  ,draw opacity=1 ]   (130,60) -- (130,100) ;
\draw [color={rgb, 255:red, 126; green, 211; blue, 33 }  ,draw opacity=1 ]   (240,60) -- (240,100) ;
\draw [color={rgb, 255:red, 126; green, 211; blue, 33 }  ,draw opacity=1 ]   (130,120+10) -- (130,140+10) ;
\draw [color={rgb, 255:red, 126; green, 211; blue, 33 }  ,draw opacity=1 ]   (240,120+10) -- (240,140+10) ;
\draw [color={rgb, 255:red, 126; green, 211; blue, 33 }  ,draw opacity=1 ]   (350,60) -- (350,80) ;
\draw [color={rgb, 255:red, 126; green, 211; blue, 33 }  ,draw opacity=1 ]   (350,120+10) -- (350,160+10) ;
\draw [color={rgb, 255:red, 126; green, 211; blue, 33 }  ,draw opacity=1 ]   (150,80) -- (140,90) ;
\draw [color={rgb, 255:red, 126; green, 211; blue, 33 }  ,draw opacity=1 ]   (260,80) -- (250,90) ;
\draw [color={rgb, 255:red, 126; green, 211; blue, 33 }  ,draw opacity=1 ]   (260,140+10) -- (250,150+10) ;
\draw [color={rgb, 255:red, 126; green, 211; blue, 33 }  ,draw opacity=1 ]   (370,140+10) -- (360,150+10) ;
\draw [color={rgb, 255:red, 126; green, 211; blue, 33 }  ,draw opacity=1 ]   (370,80) -- (350,100) ;
\draw [color={rgb, 255:red, 126; green, 211; blue, 33 }  ,draw opacity=1 ]   (150,140+10) -- (130,160+10) ;
\draw [color={rgb, 255:red, 126; green, 211; blue, 33 }  ,draw opacity=1 ]   (170,100) -- (150,100) ;
\draw [color={rgb, 255:red, 126; green, 211; blue, 33 }  ,draw opacity=1 ]   (170,160+10) -- (150,160+10) ;
\draw [color={rgb, 255:red, 126; green, 211; blue, 33 }  ,draw opacity=1 ]   (280,100) -- (260,100) ;
\draw [color={rgb, 255:red, 126; green, 211; blue, 33 }  ,draw opacity=1 ]   (390,100) -- (370,100) ;
\draw [color={rgb, 255:red, 126; green, 211; blue, 33 }  ,draw opacity=1 ]   (390,160+10) -- (370,160+10) ;
\draw [color={rgb, 255:red, 126; green, 211; blue, 33 }  ,draw opacity=1 ]   (280,160+10) -- (240,160+10) ;

\end{tikzpicture}
\caption{Six combinatorial types of trigonal morphisms for a graph of type (000)A. For the three colors blue, red and green, edges of the same color are supposed to have the same length (and map to the same image edge via the morphism). Black edges are contracted via the morphism and can have any length.}
    \label{fig:sixmorphisms}
\end{figure}

We consider four combinatorial types of non-smooth tropical curves in $\FF_0$, see Figure \ref{fig:4nonsmooth}.

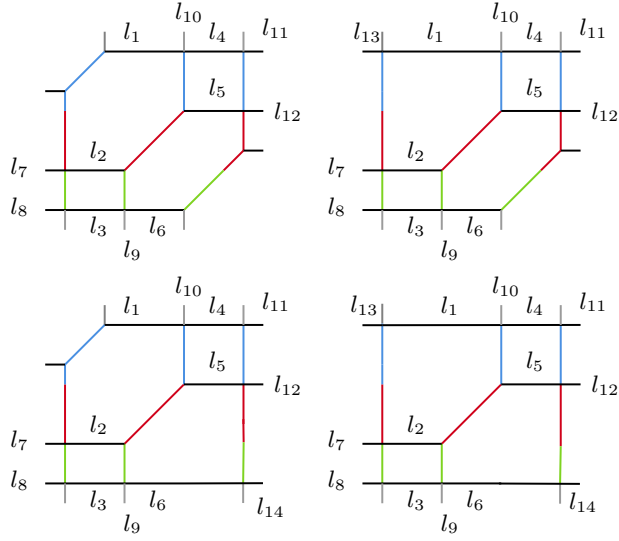
\begin{figure}[h]
    \centering

\tikzset{every picture/.style={line width=0.75pt}} 

\begin{tikzpicture}[x=0.75pt,y=0.75pt,yscale=-1,xscale=1]

\draw [color={rgb, 255:red, 128; green, 128; blue, 128 }  ,draw opacity=1 ]   (150,190) -- (150,180) ;
\draw [color={rgb, 255:red, 208; green, 2; blue, 27 }  ,draw opacity=1 ]   (160,250) -- (190,220) ;
\draw [color={rgb, 255:red, 74; green, 144; blue, 226 }  ,draw opacity=1 ]   (190,220) -- (190,190) ;
\draw [color={rgb, 255:red, 126; green, 211; blue, 33 }  ,draw opacity=1 ]   (130,250) -- (130,270) ;
\draw [color={rgb, 255:red, 208; green, 2; blue, 27 }  ,draw opacity=1 ]   (130,250) -- (130,220) ;
\draw [color={rgb, 255:red, 208; green, 2; blue, 27 }  ,draw opacity=1 ]   (210,250) -- (220,240) ;
\draw [color={rgb, 255:red, 208; green, 2; blue, 27 }  ,draw opacity=1 ]   (220,240) -- (220,220) ;
\draw [color={rgb, 255:red, 126; green, 211; blue, 33 }  ,draw opacity=1 ]   (160,250) -- (160,270) ;
\draw [color={rgb, 255:red, 126; green, 211; blue, 33 }  ,draw opacity=1 ]   (210,250) -- (190,270) ;
\draw [color={rgb, 255:red, 74; green, 144; blue, 226 }  ,draw opacity=1 ]   (220,220) -- (220,190) ;
\draw [color={rgb, 255:red, 74; green, 144; blue, 226 }  ,draw opacity=1 ]   (130,220) -- (130,210) ;
\draw [color={rgb, 255:red, 74; green, 144; blue, 226 }  ,draw opacity=1 ]   (130,210) -- (150,190) ;
\draw    (230,190) -- (150,190) ;
\draw    (190,270) -- (120,270) ;
\draw [color={rgb, 255:red, 155; green, 155; blue, 155 }  ,draw opacity=1 ]   (190,190) -- (190,180) ;
\draw [color={rgb, 255:red, 155; green, 155; blue, 155 }  ,draw opacity=1 ]   (220,190) -- (220,180) ;
\draw    (160,250) -- (120,250) ;
\draw    (230,220) -- (190,220) ;
\draw    (130,210) -- (120,210) ;
\draw [color={rgb, 255:red, 155; green, 155; blue, 155 }  ,draw opacity=1 ]   (130,280) -- (130,270) ;
\draw [color={rgb, 255:red, 155; green, 155; blue, 155 }  ,draw opacity=1 ]   (160,280) -- (160,270) ;
\draw [color={rgb, 255:red, 155; green, 155; blue, 155 }  ,draw opacity=1 ][fill={rgb, 255:red, 155; green, 155; blue, 155 }  ,fill opacity=1 ]   (190,280) -- (190,270) ;
\draw    (230,240) -- (220,240) ;
\draw [color={rgb, 255:red, 128; green, 128; blue, 128 }  ,draw opacity=1 ]   (290,190) -- (290,180) ;
\draw [color={rgb, 255:red, 208; green, 2; blue, 27 }  ,draw opacity=1 ]   (320,250) -- (350,220) ;
\draw [color={rgb, 255:red, 74; green, 144; blue, 226 }  ,draw opacity=1 ]   (350,220) -- (350,190) ;
\draw [color={rgb, 255:red, 126; green, 211; blue, 33 }  ,draw opacity=1 ]   (290,250) -- (290,270) ;
\draw [color={rgb, 255:red, 208; green, 2; blue, 27 }  ,draw opacity=1 ]   (290,250) -- (290,220) ;
\draw [color={rgb, 255:red, 208; green, 2; blue, 27 }  ,draw opacity=1 ]   (370,250) -- (380,240) ;
\draw [color={rgb, 255:red, 208; green, 2; blue, 27 }  ,draw opacity=1 ]   (380,240) -- (380,220) ;
\draw [color={rgb, 255:red, 126; green, 211; blue, 33 }  ,draw opacity=1 ]   (320,250) -- (320,270) ;
\draw [color={rgb, 255:red, 126; green, 211; blue, 33 }  ,draw opacity=1 ]   (370,250) -- (350,270) ;
\draw [color={rgb, 255:red, 74; green, 144; blue, 226 }  ,draw opacity=1 ]   (380,220) -- (380,190) ;
\draw [color={rgb, 255:red, 74; green, 144; blue, 226 }  ,draw opacity=1 ]   (290,220) -- (290,210) ;
\draw [color={rgb, 255:red, 74; green, 144; blue, 226 }  ,draw opacity=1 ]   (290,210) -- (290,190) ;
\draw    (390,190) -- (290,190) ;
\draw    (350,270) -- (280,270) ;
\draw [color={rgb, 255:red, 155; green, 155; blue, 155 }  ,draw opacity=1 ]   (350,190) -- (350,180) ;
\draw [color={rgb, 255:red, 155; green, 155; blue, 155 }  ,draw opacity=1 ]   (380,190) -- (380,180) ;
\draw    (320,250) -- (280,250) ;
\draw    (390,220) -- (350,220) ;
\draw    (290,190) -- (280,190) ;
\draw [color={rgb, 255:red, 155; green, 155; blue, 155 }  ,draw opacity=1 ]   (290,280) -- (290,270) ;
\draw [color={rgb, 255:red, 155; green, 155; blue, 155 }  ,draw opacity=1 ]   (320,280) -- (320,270) ;
\draw [color={rgb, 255:red, 155; green, 155; blue, 155 }  ,draw opacity=1 ][fill={rgb, 255:red, 155; green, 155; blue, 155 }  ,fill opacity=1 ]   (350,280) -- (350,270) ;
\draw    (390,240) -- (380,240) ;
\draw [color={rgb, 255:red, 128; green, 128; blue, 128 }  ,draw opacity=1 ]   (150,328) -- (150,318) ;
\draw [color={rgb, 255:red, 208; green, 2; blue, 27 }  ,draw opacity=1 ]   (160,388) -- (190,358) ;
\draw [color={rgb, 255:red, 74; green, 144; blue, 226 }  ,draw opacity=1 ]   (190,358) -- (190,328) ;
\draw [color={rgb, 255:red, 126; green, 211; blue, 33 }  ,draw opacity=1 ]   (130,388) -- (130,408) ;
\draw [color={rgb, 255:red, 208; green, 2; blue, 27 }  ,draw opacity=1 ]   (130,388) -- (130,358) ;
\draw [color={rgb, 255:red, 208; green, 2; blue, 27 }  ,draw opacity=1 ]   (220.08,387.22) -- (219.92,375.55) ;
\draw [color={rgb, 255:red, 208; green, 2; blue, 27 }  ,draw opacity=1 ]   (220,378) -- (220,358) ;
\draw [color={rgb, 255:red, 126; green, 211; blue, 33 }  ,draw opacity=1 ]   (160,388) -- (160,408) ;
\draw [color={rgb, 255:red, 126; green, 211; blue, 33 }  ,draw opacity=1 ]   (219.83,407.88) -- (220.08,387.22) ;
\draw [color={rgb, 255:red, 74; green, 144; blue, 226 }  ,draw opacity=1 ]   (220,358) -- (220,328) ;
\draw [color={rgb, 255:red, 74; green, 144; blue, 226 }  ,draw opacity=1 ]   (130,358) -- (130,348) ;
\draw [color={rgb, 255:red, 74; green, 144; blue, 226 }  ,draw opacity=1 ]   (130,348) -- (150,328) ;
\draw    (230,328) -- (150,328) ;
\draw    (190,408) -- (120,408) ;
\draw [color={rgb, 255:red, 155; green, 155; blue, 155 }  ,draw opacity=1 ]   (190,328) -- (190,318) ;
\draw [color={rgb, 255:red, 155; green, 155; blue, 155 }  ,draw opacity=1 ]   (220,328) -- (220,318) ;
\draw    (160,388) -- (120,388) ;
\draw    (230,358) -- (190,358) ;
\draw    (130,348) -- (120,348) ;
\draw [color={rgb, 255:red, 155; green, 155; blue, 155 }  ,draw opacity=1 ]   (130,418) -- (130,408) ;
\draw [color={rgb, 255:red, 155; green, 155; blue, 155 }  ,draw opacity=1 ]   (160,418) -- (160,408) ;
\draw [color={rgb, 255:red, 155; green, 155; blue, 155 }  ,draw opacity=1 ][fill={rgb, 255:red, 155; green, 155; blue, 155 }  ,fill opacity=1 ]   (219.83,417.88) -- (219.83,407.88) ;
\draw    (229.83,407.88) -- (190,408) ;
\draw [color={rgb, 255:red, 128; green, 128; blue, 128 }  ,draw opacity=1 ]   (290,327.3) -- (290,317.3) ;
\draw [color={rgb, 255:red, 208; green, 2; blue, 27 }  ,draw opacity=1 ]   (320,388) -- (350,358) ;
\draw [color={rgb, 255:red, 74; green, 144; blue, 226 }  ,draw opacity=1 ]   (350,358) -- (350,328) ;
\draw [color={rgb, 255:red, 126; green, 211; blue, 33 }  ,draw opacity=1 ]   (290,388) -- (290,408) ;
\draw [color={rgb, 255:red, 208; green, 2; blue, 27 }  ,draw opacity=1 ]   (290,388) -- (290,358) ;
\draw [color={rgb, 255:red, 208; green, 2; blue, 27 }  ,draw opacity=1 ]   (380,389.22) -- (380,378) ;
\draw [color={rgb, 255:red, 208; green, 2; blue, 27 }  ,draw opacity=1 ]   (380,378) -- (380,358) ;
\draw [color={rgb, 255:red, 126; green, 211; blue, 33 }  ,draw opacity=1 ]   (320,388) -- (320,408) ;
\draw [color={rgb, 255:red, 126; green, 211; blue, 33 }  ,draw opacity=1 ]   (379.67,408.22) -- (380,389.22) ;
\draw [color={rgb, 255:red, 74; green, 144; blue, 226 }  ,draw opacity=1 ]   (380,358) -- (380,328) ;
\draw [color={rgb, 255:red, 74; green, 144; blue, 226 }  ,draw opacity=1 ]   (290,358) -- (290,348) ;
\draw [color={rgb, 255:red, 74; green, 144; blue, 226 }  ,draw opacity=1 ]   (290,348) -- (290,327.3) ;
\draw    (390,328) -- (290,328.13) ;
\draw    (350,408) -- (280,408) ;
\draw [color={rgb, 255:red, 155; green, 155; blue, 155 }  ,draw opacity=1 ]   (350,328) -- (350,318) ;
\draw [color={rgb, 255:red, 155; green, 155; blue, 155 }  ,draw opacity=1 ]   (380,328) -- (380,318) ;
\draw    (320,388) -- (280,388) ;
\draw    (390,358) -- (350,358) ;
\draw    (290,328.13) -- (280,328.13) ;
\draw [color={rgb, 255:red, 155; green, 155; blue, 155 }  ,draw opacity=1 ]   (290,418) -- (290,408) ;
\draw [color={rgb, 255:red, 155; green, 155; blue, 155 }  ,draw opacity=1 ]   (320,418) -- (320,408) ;
\draw [color={rgb, 255:red, 155; green, 155; blue, 155 }  ,draw opacity=1 ][fill={rgb, 255:red, 155; green, 155; blue, 155 }  ,fill opacity=1 ]   (379.67,418.22) -- (379.67,408.22) ;
\draw    (389.67,408.22) -- (348.83,408.05) ;

\draw (158,174) node [anchor=north west][inner sep=0.75pt]  [font=\footnotesize] [align=left] {$\displaystyle l_{1}$};
\draw (201,174) node [anchor=north west][inner sep=0.75pt]  [font=\footnotesize] [align=left] {$\displaystyle l_{4}$};
\draw (141,234) node [anchor=north west][inner sep=0.75pt]  [font=\footnotesize] [align=left] {$\displaystyle l_{2}$};
\draw (141,272) node [anchor=north west][inner sep=0.75pt]  [font=\footnotesize] [align=left] {$\displaystyle l_{3}$};
\draw (201,202) node [anchor=north west][inner sep=0.75pt]  [font=\footnotesize] [align=left] {$\displaystyle l_{5}$};
\draw (171,272) node [anchor=north west][inner sep=0.75pt]  [font=\footnotesize] [align=left] {$\displaystyle l_{6}$};
\draw (101,242) node [anchor=north west][inner sep=0.75pt]  [font=\footnotesize] [align=left] {$\displaystyle l_{7}$};
\draw (101,262) node [anchor=north west][inner sep=0.75pt]  [font=\footnotesize] [align=left] {$\displaystyle l_{8}$};
\draw (158,284) node [anchor=north west][inner sep=0.75pt]  [font=\footnotesize] [align=left] {$\displaystyle l_{9}$};
\draw (184,164) node [anchor=north west][inner sep=0.75pt]  [font=\footnotesize] [align=left] {$\displaystyle l_{10}$};
\draw (228,172) node [anchor=north west][inner sep=0.75pt]  [font=\footnotesize] [align=left] {$\displaystyle l_{11}$};
\draw (234,214) node [anchor=north west][inner sep=0.75pt]  [font=\footnotesize] [align=left] {$\displaystyle l_{12}$};
\draw (311,174) node [anchor=north west][inner sep=0.75pt]  [font=\footnotesize] [align=left] {$\displaystyle l_{1}$};
\draw (361,174) node [anchor=north west][inner sep=0.75pt]  [font=\footnotesize] [align=left] {$\displaystyle l_{4}$};
\draw (301,234) node [anchor=north west][inner sep=0.75pt]  [font=\footnotesize] [align=left] {$\displaystyle l_{2}$};
\draw (301,272) node [anchor=north west][inner sep=0.75pt]  [font=\footnotesize] [align=left] {$\displaystyle l_{3}$};
\draw (361,202) node [anchor=north west][inner sep=0.75pt]  [font=\footnotesize] [align=left] {$\displaystyle l_{5}$};
\draw (331,272) node [anchor=north west][inner sep=0.75pt]  [font=\footnotesize] [align=left] {$\displaystyle l_{6}$};
\draw (261,242) node [anchor=north west][inner sep=0.75pt]  [font=\footnotesize] [align=left] {$\displaystyle l_{7}$};
\draw (261,262) node [anchor=north west][inner sep=0.75pt]  [font=\footnotesize] [align=left] {$\displaystyle l_{8}$};
\draw (318,284) node [anchor=north west][inner sep=0.75pt]  [font=\footnotesize] [align=left] {$\displaystyle l_{9}$};
\draw (344,164) node [anchor=north west][inner sep=0.75pt]  [font=\footnotesize] [align=left] {$\displaystyle l_{10}$};
\draw (388,172) node [anchor=north west][inner sep=0.75pt]  [font=\footnotesize] [align=left] {$\displaystyle l_{11}$};
\draw (394,214) node [anchor=north west][inner sep=0.75pt]  [font=\footnotesize] [align=left] {$\displaystyle l_{12}$};
\draw (158,312) node [anchor=north west][inner sep=0.75pt]  [font=\footnotesize] [align=left] {$\displaystyle l_{1}$};
\draw (201,312) node [anchor=north west][inner sep=0.75pt]  [font=\footnotesize] [align=left] {$\displaystyle l_{4}$};
\draw (141,372) node [anchor=north west][inner sep=0.75pt]  [font=\footnotesize] [align=left] {$\displaystyle l_{2}$};
\draw (141,410) node [anchor=north west][inner sep=0.75pt]  [font=\footnotesize] [align=left] {$\displaystyle l_{3}$};
\draw (201,340) node [anchor=north west][inner sep=0.75pt]  [font=\footnotesize] [align=left] {$\displaystyle l_{5}$};
\draw (171,410) node [anchor=north west][inner sep=0.75pt]  [font=\footnotesize] [align=left] {$\displaystyle l_{6}$};
\draw (101,380) node [anchor=north west][inner sep=0.75pt]  [font=\footnotesize] [align=left] {$\displaystyle l_{7}$};
\draw (101,400) node [anchor=north west][inner sep=0.75pt]  [font=\footnotesize] [align=left] {$\displaystyle l_{8}$};
\draw (158,422) node [anchor=north west][inner sep=0.75pt]  [font=\footnotesize] [align=left] {$\displaystyle l_{9}$};
\draw (184,302) node [anchor=north west][inner sep=0.75pt]  [font=\footnotesize] [align=left] {$\displaystyle l_{10}$};
\draw (228,310) node [anchor=north west][inner sep=0.75pt]  [font=\footnotesize] [align=left] {$\displaystyle l_{11}$};
\draw (234,352) node [anchor=north west][inner sep=0.75pt]  [font=\footnotesize] [align=left] {$\displaystyle l_{12}$};
\draw (318,312) node [anchor=north west][inner sep=0.75pt]  [font=\footnotesize] [align=left] {$\displaystyle l_{1}$};
\draw (361,312) node [anchor=north west][inner sep=0.75pt]  [font=\footnotesize] [align=left] {$\displaystyle l_{4}$};
\draw (301,372) node [anchor=north west][inner sep=0.75pt]  [font=\footnotesize] [align=left] {$\displaystyle l_{2}$};
\draw (301,410) node [anchor=north west][inner sep=0.75pt]  [font=\footnotesize] [align=left] {$\displaystyle l_{3}$};
\draw (361,340) node [anchor=north west][inner sep=0.75pt]  [font=\footnotesize] [align=left] {$\displaystyle l_{5}$};
\draw (331,410) node [anchor=north west][inner sep=0.75pt]  [font=\footnotesize] [align=left] {$\displaystyle l_{6}$};
\draw (261,380) node [anchor=north west][inner sep=0.75pt]  [font=\footnotesize] [align=left] {$\displaystyle l_{7}$};
\draw (261,400) node [anchor=north west][inner sep=0.75pt]  [font=\footnotesize] [align=left] {$\displaystyle l_{8}$};
\draw (318,422) node [anchor=north west][inner sep=0.75pt]  [font=\footnotesize] [align=left] {$\displaystyle l_{9}$};
\draw (344,302) node [anchor=north west][inner sep=0.75pt]  [font=\footnotesize] [align=left] {$\displaystyle l_{10}$};
\draw (388,310) node [anchor=north west][inner sep=0.75pt]  [font=\footnotesize] [align=left] {$\displaystyle l_{11}$};
\draw (394,352) node [anchor=north west][inner sep=0.75pt]  [font=\footnotesize] [align=left] {$\displaystyle l_{12}$};
\draw (274,174) node [anchor=north west][inner sep=0.75pt]  [font=\footnotesize] [align=left] {$\displaystyle l_{13}$};
\draw (225.17,413.8) node [anchor=north west][inner sep=0.75pt]  [font=\footnotesize] [align=left] {$\displaystyle l_{14}$};
\draw (272.5,311.67) node [anchor=north west][inner sep=0.75pt]  [font=\footnotesize] [align=left] {$\displaystyle l_{13}$};
\draw (382.83,411.13) node [anchor=north west][inner sep=0.75pt]  [font=\footnotesize] [align=left] {$\displaystyle l_{14}$};
\end{tikzpicture}

    \caption{Four combinatorial types of non-smooth embeddings into $\FF_0$. The blue edges have length $L_1$, the red edges have length $L_2$ and the green edges have length $L_3$. The black edges are contracted, their lengths are denoted by $l_1,\ldots, l_6$. The lengths $l_7\ldots,l_{14}$ denote the lengths of the edges we unfold from the $4$-valent vertices.}
    \label{fig:4nonsmooth}
\end{figure}

By Theorem 3.2 \cite{CM14}, we can re-embed the curve in such a way that we can unfold edges of arbitrary length separating the edges (resp.\ the vertex and edge) that form a crossing dual to a parallelogram (resp.\ dual to a quadrangle). The effect of the projection morphism is depicted in Figure \ref{fig:projmorphism}, where we incorporate all $4$ combinatorial types in one picture, which is possible since we can just set the corresponding length of unfolded edges to $0$.

\begin{figure}[h]
    \centering
\tikzset{every picture/.style={line width=0.75pt}} 

\begin{tikzpicture}[x=0.75pt,y=0.75pt,yscale=-1,xscale=1]

\draw    (210,80) -- (350,80) ;
\draw    (210,80) -- (250,120) ;
\draw    (310,120) -- (350,160) ;
\draw    (310,120) -- (350,80) ;
\draw    (210,160) -- (250,120) ;
\draw    (210,160) -- (350,160) ;
\draw    (250,120) -- (310,120) ;
\draw    (350,80) -- (350,160) ;
\draw    (210,80) -- (210,160) ;
\draw [color={rgb, 255:red, 208; green, 2; blue, 27 }  ,draw opacity=1 ]   (250,120) -- (310,120) ;
\draw [color={rgb, 255:red, 208; green, 2; blue, 27 }  ,draw opacity=1 ]   (290,160) -- (350,160) ;
\draw [color={rgb, 255:red, 208; green, 2; blue, 27 }  ,draw opacity=1 ]   (210,80) -- (270,80) ;
\draw [color={rgb, 255:red, 74; green, 144; blue, 226 }  ,draw opacity=1 ]   (270,80) -- (310,80) ;
\draw [color={rgb, 255:red, 74; green, 144; blue, 226 }  ,draw opacity=1 ]   (310,120) -- (330,100) ;
\draw [color={rgb, 255:red, 74; green, 144; blue, 226 }  ,draw opacity=1 ]   (350,160) -- (350,120) ;
\draw [color={rgb, 255:red, 126; green, 211; blue, 33 }  ,draw opacity=1 ]   (210,80) -- (210,130) ;
\draw [color={rgb, 255:red, 126; green, 211; blue, 33 }  ,draw opacity=1 ]   (250,120) -- (230,140) ;
\draw [color={rgb, 255:red, 126; green, 211; blue, 33 }  ,draw opacity=1 ]   (290,160) -- (250,160) ;

\draw (320,60) node [anchor=north west][inner sep=0.75pt]  [font=\footnotesize] [align=left] {$\displaystyle l_{1} +l_{13}$};
\draw (311,84) node [anchor=north west][inner sep=0.75pt]  [font=\footnotesize] [align=left] {$\displaystyle l_{10}$};
\draw (351,92) node [anchor=north west][inner sep=0.75pt]  [font=\footnotesize] [align=left] {$\displaystyle l_{4} +l_{11}$};
\draw (231,142) node [anchor=north west][inner sep=0.75pt]  [font=\footnotesize] [align=left] {$\displaystyle l_{9}$};
\draw (281,132) node [anchor=north west][inner sep=0.75pt]  [font=\footnotesize] [align=left] {$\displaystyle l_{5} +l_{12}$};
\draw (241,92) node [anchor=north west][inner sep=0.75pt]  [font=\footnotesize] [align=left] {$\displaystyle l_{2} +l_{7}$};
\draw (171,134) node [anchor=north west][inner sep=0.75pt]  [font=\footnotesize] [align=left] {$\displaystyle l_{3} +l_{8}$};
\draw (221,162) node [anchor=north west][inner sep=0.75pt]  [font=\footnotesize] [align=left] {$\displaystyle l_{6} +l_{14}$};

\end{tikzpicture}

    \caption{The effect of the projection morphism for the $4$ non-smooth embeddings into $\FF_0$.}
    \label{fig:projmorphism}
\end{figure}
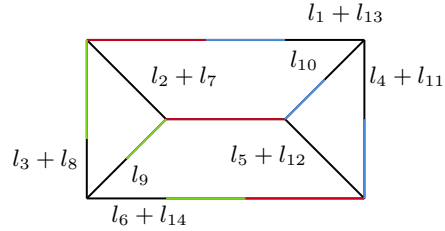

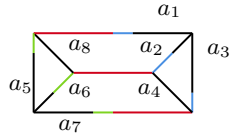
\begin{figure}[h]
    \centering

\tikzset{every picture/.style={line width=0.75pt}} 

\begin{tikzpicture}[x=0.75pt,y=0.75pt,yscale=-1,xscale=1]

\draw [color={rgb, 255:red, 208; green, 2; blue, 27 }  ,draw opacity=1 ]   (150,490) -- (190,490) ;
\draw [color={rgb, 255:red, 126; green, 211; blue, 33 }  ,draw opacity=1 ]   (150,490) -- (140,500) ;
\draw [color={rgb, 255:red, 74; green, 144; blue, 226 }  ,draw opacity=1 ]   (190,490) -- (200,480) ;
\draw    (150,490) -- (130,470) ;
\draw [color={rgb, 255:red, 208; green, 2; blue, 27 }  ,draw opacity=1 ]   (170,510) -- (210,510) ;
\draw [color={rgb, 255:red, 208; green, 2; blue, 27 }  ,draw opacity=1 ]   (130,470) -- (170,470) ;
\draw [color={rgb, 255:red, 74; green, 144; blue, 226 }  ,draw opacity=1 ]   (170,470) -- (180,470) ;
\draw [color={rgb, 255:red, 74; green, 144; blue, 226 }  ,draw opacity=1 ]   (210,510) -- (210,500) ;
\draw [color={rgb, 255:red, 126; green, 211; blue, 33 }  ,draw opacity=1 ]   (170,510) -- (160,510) ;
\draw [color={rgb, 255:red, 126; green, 211; blue, 33 }  ,draw opacity=1 ]   (130,470) -- (130,480) ;
\draw    (130,510) -- (130,480) ;
\draw    (130,510) -- (140,500) ;
\draw    (130,510) -- (160,510) ;
\draw    (190,490) -- (210,510) ;
\draw    (210,470) -- (210,500) ;
\draw    (180,470) -- (210,470) ;
\draw    (200,480) -- (210,470) ;

\draw (191,454) node [anchor=north west][inner sep=0.75pt]  [font=\footnotesize] [align=left] {$\displaystyle a_{1}$};
\draw (182,473) node [anchor=north west][inner sep=0.75pt]  [font=\footnotesize] [align=left] {$\displaystyle a_{2}$};
\draw (216,474) node [anchor=north west][inner sep=0.75pt]  [font=\footnotesize] [align=left] {$\displaystyle a_{3}$};
\draw (181,494) node [anchor=north west][inner sep=0.75pt]  [font=\footnotesize] [align=left] {$\displaystyle a_{4}$};
\draw (116,492) node [anchor=north west][inner sep=0.75pt]  [font=\footnotesize] [align=left] {$\displaystyle a_{5}$};
\draw (146,494) node [anchor=north west][inner sep=0.75pt]  [font=\footnotesize] [align=left] {$\displaystyle a_{6}$};
\draw (141,512) node [anchor=north west][inner sep=0.75pt]  [font=\footnotesize] [align=left] {$\displaystyle a_{7}$};
\draw (146,472) node [anchor=north west][inner sep=0.75pt]  [font=\footnotesize] [align=left] {$\displaystyle a_{8}$};

\end{tikzpicture}

    \caption{Labeling of the contracted edges for a graph of type (000)A.}
    \label{fig:000labels}
\end{figure}

In Figure \ref{fig:000labels}, we fix a uniform labeling of the lengths of the contracted edges. The blue edges have length $L_1$, the red edges $L_2$ and the green edges $L_3$.
If we are in type (1,1) in Figure \ref{fig:sixmorphisms}, then $a_1=a_5=0$. In type (1,2), $a_2=a_5=0$. In type (1,3), $a_2=a_6=0$. In type (2,1), $a_1=a_6=0$. In type (2,2), $a_1=a_7=0$. In type (2,3) finally, $a_3=a_5=0$.

In type (1,1), let $l_3=l_8=0$. Then also $l_2=0$. 
Let $l_4=l_5=\min\{a_3,a_4\}$.

If $a_7<l_5+L_2$, pick the embedding on the top left, i.e.\ $l_{13}$ and $l_{14}$ do not exist. 
Else, pick the embedding on the bottom left, i.e.\ $l_{13}$ does not exist.

Let $l_1=0$. Pick $l_7=a_8$, $l_9=a_6$, $l_{10}=a_2$. Let $l_4=l_5=\min\{a_3,a_4\}$. If $\min\{a_3,a_4\}=a_3$, let $l_{11}=0$ and $l_{12}=a_4-a_3$. Else, let $l_{12}=0$ and $l_{11}=a_3-a_4$.
If $a_7<l_5+L_2$, let $l_6=a_7$. Else, $l_6=l_5+L_2$, choose $l_{14}=a_7-l_6$.

In type (1,2), let $l_3=l_8=l_{10}=0$. Then also $l_2=0$. Let $l_9=a_6$, $l_7=a_8$.

If $a_1<l_2+L_2$, pick one of the embeddings on the left, i.e.\ $l_{13}$ does not exist. Let $a_1=l_1$. Else, pick one of the embeddings on the right. Then $l_1=l_2+L_2$. Let $l_{13}=a_1-l_1$.

Let $l_4=l_5=\min\{a_3,a_4\}$. If $\min\{a_3,a_4\}=a_3$, let $l_{11}=0$ and $l_{12}=a_4-a_3$. Else, let $l_{12}=0$ and $l_{11}=a_3-a_4$.

If $a_7<l_5+L_2$, pick one of the embeddings on the top, i.e.\ $l_{14}$ does not exist. Let $a_7=l_6$. Else, pick one of the embeddings on the bottom. Then $l_6=l_5+L_2$. Let $l_{14}=a_7-l_6$.

In type (1,3), let $l_9=l_{10}=0$.
Let $l_2=l_3=\min\{a_5,a_8\}$.
If $\min\{a_5,a_8\}=a_5$, let $l_8=0$ and $l_7=a_8-a_5$. Else, let $l_7=0$ and $l_8=a_5-a_8$.

If $a_1<l_2+L_2$, pick one of the embeddings on the left, i.e.\ $l_{13}$ does not exist. Let $a_1=l_1$. Else, pick one of the embeddings on the right. Then $l_1=l_2+L_2$. Let $l_{13}=a_1-l_1$.

Let $l_4=l_5=\min\{a_3,a_4\}$. If $\min\{a_3,a_4\}=a_3$, let $l_{11}=0$ and $l_{12}=a_4-a_3$. Else, let $l_{12}=0$ and $l_{11}=a_3-a_4$.

If $a_7<l_5+L_2$, pick one of the embeddings on the top, i.e.\ $l_{14}$ does not exist. Let $a_7=l_6$. Else, pick one of the embeddings on the bottom. Then $l_6=l_5+L_2$. Let $l_{14}=a_7-l_6$.

In type (2,1), let $l_4=l_5=\min\{a_3,a_4\}$. If $\min\{a_3,a_4\}=a_3$, let $l_{11}=0$ and $l_{12}=a_4-a_3$. Else, let $l_{12}=0$ and $l_{11}=a_3-a_4$.

Let $l_2=l_3=\min\{a_5,a_8\}$.
If $\min\{a_5,a_8\}=a_5$, let $l_8=0$ and $l_7=a_8-a_5$. Else, let $l_7=0$ and $l_8=a_5-a_8$.

If $a_7<l_5+L_2$, pick the embedding on the top left, i.e.\ $l_{13}$ and $l_{14}$ do not exist. Let $a_7=l_6$.
Else, pick the embedding on the bottom left, i.e.\ $l_{13}$ does not exist. Then $l_6=l_5+L_2$. Let $l_{14}=a_7-l_6$.

Let $l_1=l_9=0$. Let $l_{10}=a_2$.

In type (2,2), let $l_1=l_6=0$. Pick the upper left embedding, i.e.\ $l_{13}, l_{14}$ do not exist.

Let $l_2=l_3=\min\{a_5,a_8\}$.
If $\min\{a_5,a_8\}=a_5$, let $l_8=0$ and $l_7=a_8-a_5$. Else, let $l_7=0$ and $l_8=a_5-a_8$.

Let $l_4=l_5=\min\{a_3,a_4\}$. If $\min\{a_3,a_4\}=a_3$, let $l_{11}=0$ and $l_{12}=a_4-a_3$. Else, let $l_{12}=0$ and $l_{11}=a_3-a_4$.

Let $l_{10}=a_1$, $l_9=a_6$.

In type (2,3), let $l_2=l_3=l_8=l_4=l_5=l_{11}=0$.

If $a_1<l_2+L_2$, pick one of the embeddings on the left, i.e.\ $l_{13}$ does not exist. Let $a_1=l_1$. Else, pick one of the embeddings on the right. Then $l_1=l_2+L_2$. Let $l_{13}=a_1-l_1$.

If $a_7<l_5+L_2$, pick one of the embeddings on the top, i.e.\ $l_{14}$ does not exist. Let $a_7=l_6$. Else, pick one of the embeddings on the bottom. Then $l_6=l_5+L_2$. Let $l_{14}=a_7-l_6$.

Let $l_7=a_8$, $l_9=a_6$, $l_{10}=a_2$ and $l_{12}=a_4$.

In any case, we manage to unfold $\Gamma$ in such a way that the trigonal morphism is respected.
\appendix
\section{Embeddings of genus 4 curves of maximal combinatorial type in Hirzebruch surfaces}
Here we complete the discussion started in Section \ref{ssc:genus4} for all the remaining combinatorial types in Figure \ref{fg:genus4_maxtype}.
\subsubsection*{Type (010)}\label{sc:appendix}
Let us label the cycles and the vertices as in Figure \ref{fg:(010)F_0}.
The cycle $3$ is adjacent to all the others, and the cycle $4$ does not meet $1$ and $2$. Then the triangulation describing the embedding in $\mathbb F_0$ has to contain the blue edges as in the second figure in Figure \ref{fg:(010)F_0}, and since we want the tropical plane curve to be smooth, then it will contain also the black edges in the third figure.
\begin{figure}[h]
\begin{tikzcd}
    \begin{tikzpicture}[scale=0.75]
        \coordinate (1) at (0,0);
        \coordinate (2) at (2,0);
        \coordinate (3) at (0,1);
        \coordinate (4) at (2,1);
        \coordinate (5) at (1,1);
        \coordinate (6) at (1,2);
        \draw (1)--(3);
        \draw (2)--(4);
        \draw (3)--(5);
        \draw (4)--(5);
        \draw (6)--(5);
        \draw (4)--(6);
        \draw (3)--(6);
        \draw(1)[] to [out=30, in=150] (2);
        \draw(1)[] to [out=330, in=210] (2);
        \foreach \i in {1,2,3,4,5,6}{
            \vertex{\i}
        }

        \draw[blue] (0.5,1.1) node[] {1};
        \draw[blue] (1.5,0.5) node[] {3};
        \draw[blue] (1.5,1.1) node[] {2};
        \draw[blue] (1,-0.1) node[] {4};
        \draw (0,0) node[anchor=north] {$x$};
        \draw (2,0) node[anchor=north] {$y$};
        \draw(0,1) node[anchor=south] {$u$};
        \draw (2,1) node[anchor=south] {$v$};
        \draw(1,0.5) node[anchor=south] {$z$};
        \draw (1,2) node[anchor=south] {$w$};
    \end{tikzpicture}&
    \begin{tikzpicture}
        \foreach \i in {0,0.5,1,1.5}{
            \foreach \j in {0,0.5,1,1.5}{
            \vertex{\i,\j}
            }
    }
    \vertexblue{0.5,0.5}\vertexblue{1,0.5}
    \vertexblue{1,1}
    \vertexblue{0.5,1}
    \draw(0,0)--(1.5,0);
    \draw(0,0)--(0,1.5);
    \draw(1.5,1.5)--(0,1.5);
    \draw(1.5,1.5)--(1.5,0);
    \draw[blue] (0.5,0.1) node[] {1};
    \draw[blue] (1,0.1) node[] {2};
    \draw[blue] (0.5,1.1) node[] {3};
    \draw[blue] (1,1.1) node[] {4};
    \draw[blue] (0.5,0.5)--(1,0.5);
    \draw[blue] (0.5,1)--(1,0.5);
    \draw[blue] (0.5,1)--(1,1);
    \draw[blue] (0.5,1)--(0.5,0.5);
    \draw[blue] (0.5,1)--(1.5,0.5);
    \end{tikzpicture}
    &
    \begin{tikzpicture}
        \foreach \i in {0,0.5,1,1.5}{
            \foreach \j in {0,0.5,1,1.5}{
            \vertex{\i,\j}
            }
    }
    \vertexblue{0.5,0.5}\vertexblue{1,0.5}
    \vertexblue{1,1}
    \vertexblue{0.5,1}
    \draw(0,0)--(1.5,0);
    \draw(0,0)--(0,1.5);
    \draw(1.5,1.5)--(0,1.5);
    \draw(1.5,1.5)--(1.5,0);
    \draw[blue] (0.5,0.1) node[] {1};
    \draw[blue] (1,0.1) node[] {2};
    \draw[blue] (0.5,1.1) node[] {3};
    \draw[blue] (1,1.1) node[] {4};
    \draw[] (1.5,0.5)--(1,0.5);
    \draw(1,1)--(1.5,0.5);
    \draw[blue] (0.5,0.5)--(1,0.5);
    \draw[blue] (0.5,1)--(1,0.5);
    \draw[blue] (0.5,1)--(1,1);
    \draw[blue] (0.5,1)--(0.5,0.5);
    \draw[blue] (0.5,1)--(1.5,0.5);
    \end{tikzpicture}&
    \begin{tikzpicture}[scale=0.75]
    \draw[blue](0.5,0.5)--(1,1);
    \draw[blue](1,1)--(1.5,2);
    \draw[blue](1.5,2)--(1.5,2.5);
    \draw[](1.5,2)--(2,2.5);
    \draw[blue](0.5,0.5)--(0.5,-0.2);
    \draw[blue](0.5,0.5)--(-0.5,0.5);
    \draw[](1,1)--(1,0);
    \vertexblue{-0.5,0.5}\vertexblue{0.5,-0.2}
    \vertexblue{0.5,0.5}
    \vertexblue{1,1}
    \vertexblue{1.5,2}
    \vertexblue{1.5,2.5}
    \draw[] (0.5,0.5) node[anchor=north east] {$z$};
    \draw[] (1,1) node[anchor=south east] {$v$};
    \draw[] (-0.5,0.5) node[anchor= east] {$u$};
    \draw[] (0.5,-0.2) node[anchor= north] {$w$};
    \draw[] (1.5,2) node[anchor= north west] {$y$};
    \draw[] (1.5,2.5) node[anchor= south] {$x$};
    \end{tikzpicture}
    \end{tikzcd}\caption{Combinatorial type, partial triangulations of $\mathbb F_0$ and partial embedded tropical curve.}\label{fg:(010)F_0}
    \end{figure}
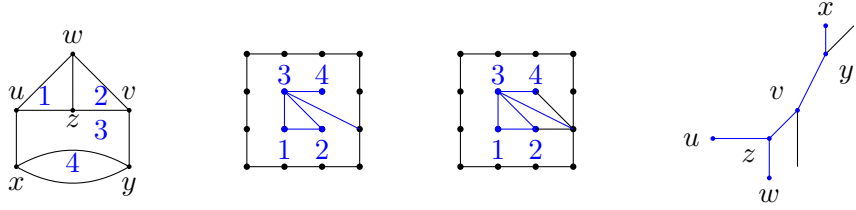

Comparing $x$-coordinates, resp. $y$-coordinates, of $ ux$ and of the edges $uz,zv,vy,xy$ gives inequality $uz+vz+vy<ux,$ resp. $vz+2vy+xy<ux$.

Similarly, possible triangulations of $\mathbb F_2,$ have to be as in Figure \ref{fg:(010)F_2a}. 
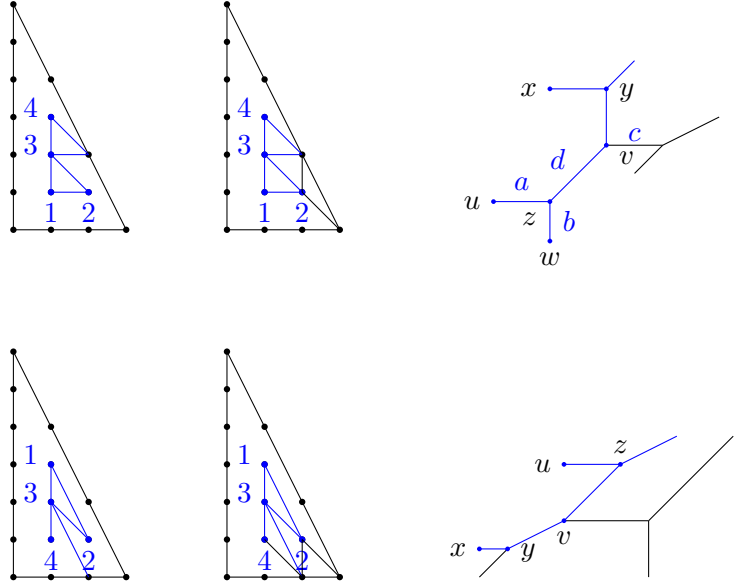
\begin{figure}[h]
\begin{tikzcd}
    \begin{tikzpicture}
        \foreach \j in {0,0.5,1,1.5,2,2.5,3}{
            \vertex{0,\j}
        }
        \foreach \j in {0,0.5,1,1.5,2}{
            \vertex{0.5,\j}
        }
        \foreach \j in {0,0.5,1}{
            \vertex{1,\j}
        }
        \vertex{1.5,0}
        \draw (0,0)--(1.5,0);
        \draw (0,3)--(1.5,0);
        \draw (0,3)--(0,0);
    \vertexblue{0.5,0.5}\vertexblue{1,0.5}
    \vertexblue{0.5,1.5}
    \vertexblue{0.5,1}
    \draw[blue] (0.5,0.1) node[] {1};
    \draw[blue] (1,0.1) node[] {2};
    \draw[blue] (0.5,1.1) node[anchor=east] {3};
    \draw[blue] (0.5,1.6) node[anchor=east] {4};
    \draw[blue] (0.5,0.5)--(1,0.5);
    \draw[blue] (1,1)--(0.5,1.5);
    \draw[blue] (0.5,1)--(0.5,1.5);
    \draw[blue] (0.5,1)--(1,0.5);
    \draw[blue] (0.5,1)--(0.5,0.5);
    \draw[blue] (0.5,1)--(1,1);
    \end{tikzpicture}&
    \begin{tikzpicture}
        \foreach \j in {0,0.5,1,1.5,2,2.5,3}{
            \vertex{0,\j}
        }
        \foreach \j in {0,0.5,1,1.5,2}{
            \vertex{0.5,\j}
        }
        \foreach \j in {0,0.5,1}{
            \vertex{1,\j}
        }
        \vertex{1.5,0}
        \draw (0,0)--(1.5,0);
        \draw (0,3)--(1.5,0);
        \draw (0,3)--(0,0);
    \vertexblue{0.5,0.5}\vertexblue{1,0.5}
    \vertexblue{0.5,1.5}
    \vertexblue{0.5,1}
    \draw[blue] (0.5,0.1) node[] {1};
    \draw[blue] (1,0.1) node[] {2};
    \draw[blue] (0.5,1.1) node[anchor=east] {3};
    \draw[blue] (0.5,1.6) node[anchor=east] {4};
    \draw[blue] (0.5,0.5)--(1,0.5);
    \draw[blue] (1,1)--(0.5,1.5);
    \draw[blue] (0.5,1)--(0.5,1.5);
    \draw[blue] (0.5,1)--(1,0.5);
    \draw[blue] (0.5,1)--(0.5,0.5);
    \draw[blue] (0.5,1)--(1,1);
    \draw(1,0.5)--(1,1);
    \draw(1,0.5)--(1.5,0); 
    \end{tikzpicture}&
    \begin{tikzpicture}[scale=0.75]
    \draw[blue](0.5,0.5)--(1.5,1.5);
    \draw[blue](0.5,0.5)--(0.5,-0.2);
    \draw[blue](0.5,0.5)--(-0.5,0.5);
    \draw[blue](1.5,1.5)--(1.5,2.5);
    \draw[blue](1.5,2.5)--(2,3);
    \draw[blue](0.5,2.5)--(1.5,2.5);
    \draw[](1.5,1.5)--(2.5,1.5);
    \draw[](2.5,1.5)--(3.5,2);
    \draw[](2.5,1.5)--(2,1);
    \vertexblue{-0.5,0.5}\vertexblue{0.5,-0.2}
    \vertexblue{0.5,0.5}
    \vertexblue{1.5,1.5}
    \vertexblue{0.5,2.5}
    \vertexblue{1.5,2.5}
    \draw[] (0.5,0.5) node[anchor=north east] {$z$};
    \draw[] (1.5,2.5) node[anchor=west] {$y$};
    \draw[] (-0.5,0.5) node[anchor= east] {$u$};
    \draw[] (0.5,-0.2) node[anchor= north] {$w$};
    \draw[] (1.5,1.3) node[anchor= west] {$v$};
    \draw[] (0.5,2.5) node[anchor= east] {$x$};
    \draw[blue] (0,0.5) node[anchor= south] {$a$};
    \draw[blue] (0.5,0.1) node[anchor= west] {$b$};
    \draw[blue] (2,2) node[anchor= north] {$c$};
    \draw[blue] (1,1.2) node[anchor= east] {$d$};
    \end{tikzpicture}\\
    \begin{tikzpicture}
        \foreach \j in {0,0.5,1,1.5,2,2.5,3}{
            \vertex{0,\j}
        }
        \foreach \j in {0,0.5,1,1.5,2}{
            \vertex{0.5,\j}
        }
        \foreach \j in {0,0.5,1}{
            \vertex{1,\j}
        }
        \vertex{1.5,0}
        \draw (0,0)--(1.5,0);
        \draw (0,3)--(1.5,0);
        \draw (0,3)--(0,0);
    \vertexblue{0.5,0.5}\vertexblue{1,0.5}
    \vertexblue{0.5,1.5}
    \vertexblue{0.5,1}
    \draw[blue] (0.5,0.1) node[] {4};
    \draw[blue] (1,0.1) node[] {2};
    \draw[blue] (0.5,1.1) node[anchor=east] {3};
    \draw[blue] (0.5,1.6) node[anchor=east] {1};
    \draw[blue] (0.5,1)--(0.5,1.5);
    \draw[blue] (0.5,1)--(1,0.5);
    \draw[blue] (0.5,1)--(1,0);
    \draw[blue] (0.5,1)--(0.5,0.5);
    \draw[blue] (0.5,1.5)--(1,0.5);
    \end{tikzpicture}&\begin{tikzpicture}
        \foreach \j in {0,0.5,1,1.5,2,2.5,3}{
            \vertex{0,\j}
        }
        \foreach \j in {0,0.5,1,1.5,2}{
            \vertex{0.5,\j}
        }
        \foreach \j in {0,0.5,1}{
            \vertex{1,\j}
        }
        \vertex{1.5,0}
        \draw (0,0)--(1.5,0);
        \draw (0,3)--(1.5,0);
        \draw (0,3)--(0,0);
    \vertexblue{0.5,0.5}\vertexblue{1,0.5}
    \vertexblue{0.5,1.5}
    \vertexblue{0.5,1}
    \draw[blue] (0.5,0.1) node[] {4};
    \draw[blue] (1,0.1) node[] {2};
    \draw[blue] (0.5,1.1) node[anchor=east] {3};
    \draw[blue] (0.5,1.6) node[anchor=east] {1};
    \draw[blue] (0.5,1)--(0.5,1.5);
    \draw[blue] (0.5,1)--(1,0.5);
    \draw[blue] (0.5,1)--(1,0);
    \draw[blue] (0.5,1)--(0.5,0.5);
    \draw[blue] (0.5,1.5)--(1,0.5);
    \draw(1,0.5)--(1.5,0); 
    \draw(1,0)--(1,0.5); 
    \draw(0.5,0.5)--(1,0); 
    \end{tikzpicture}&
   \begin{tikzpicture}[scale=0.75]
    \draw[](2.5,2.5)--(1,1);
    \draw[](1,0)--(1,1);
    \draw[] (1,1)--(-0.5,1);
    \draw[blue] (0.5,2)--(1.5,2.5);
    \draw[blue] (0.5,2)--(-0.5,2);
    \draw[blue] (0.5,2)--(-0.5,1);
    \draw[blue] (-1.5,0.5)--(-0.5,1);
    \draw[blue] (-1.5,0.5)--(-2,0.5);
    \draw[] (-1.5,0.5)--(-2,0);
    \vertexblue{-2,0.5}\vertexblue{-1.5,0.5}
    \vertexblue{-0.5,2}
    \vertexblue{-0.5,1}
    \vertexblue{0.5,2}
    \draw[] (-2,0.5) node[anchor= east] {$x$};
    \draw[] (-1.5,0.5) node[anchor=west] {$y$};
    \draw[] (-0.5,2) node[anchor= east] {$u$};
    \draw[] (-0.5,1) node[anchor= north] {$v$};
    \draw[] (0.5,2) node[anchor= south] {$z$};
    \end{tikzpicture}
    \end{tikzcd}\caption{Partial triangulations of $\mathbb F_2$ and partial embedded tropical curves.}\label{fg:(010)F_2a}
\end{figure}
    
In both cases, any completion of such triangulations would would give segments joining $u$ and $x$ of the form $\alpha, 1),$ for some $\alpha \in\mathbb Z$.
Therefore the length of the edge $ux$ will be equal to the difference between the $y$-coordinates of the points $x$ and $u.$ 
This is also true for the edges joining the points $v,y,z$ since they are also defined by segments of the form $(\alpha', 1)$, for some $\alpha'\in\ZZ$. Moreover $y,z$ have same $y$-coordinate as $x,u,$ respectively, thus $vz+vy=ux.$
This proves the following.
\begin{lemma}
    Let $\Gamma$ be a tropical curve of genus $4$ of combinatorial type (010).
    \begin{itemize}
        \item If $\Gamma$ is realizable in $\mathbb F_0,$ then its edge lengths satisfy, up to symmetry:
        \begin{equation}
        \begin{cases}
           uz+vz+vy<ux;\\
           vz+2vy+xy<ux.\\
           \end{cases}
        \end{equation}
        \item If $\Gamma$ is realizable in $\mathbb F_2,$ then its edge lengths satisfy, up to symmetry:
        \begin{equation}
               vz+vy=ux.
        \end{equation}
    \end{itemize}
\end{lemma}

\subsubsection*{Type (020)}

Let us now consider the combinatorial type (020). The triangulation describing the embedding in $\mathbb F_0$ is represented in Figure \ref{fg:(020)F_0}.
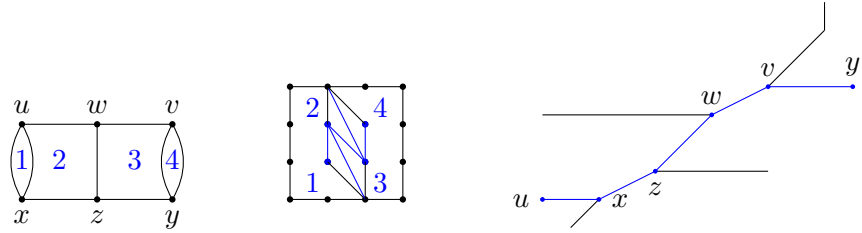
\begin{figure}[h]
\begin{tikzcd}
      \begin{tikzpicture}
        \coordinate (1) at (0,0);
        \coordinate (2) at (1,0);
        \coordinate (3) at (2,0);
        \coordinate (4) at (0,1);
        \coordinate (5) at (2,1);
        \coordinate (6) at (1,1);
        \draw (1)--(2);
        \draw (2)--(6);
        \draw (2)--(3);
        \draw (4)--(6);
        \draw (6)--(5);
        \draw(1)[] to [out=120, in=240] (4);
        \draw(1)[] to [out=60, in=300] (4);
        \draw(3)[] to [out=120, in=240] (5);
        \draw(3)[] to [out=60, in=300] (5);
        \foreach \i in {1,2,3,4,5,6}{
            \vertex{\i}
        }
        \draw (0,0) node[anchor=north] {$x$};
        \draw (2,0) node[anchor=north] {$y$};
        \draw(0,1) node[anchor=south] {$u$};
        \draw (2,1) node[anchor=south] {$v$};
        \draw(1,0) node[anchor=north] {$z$};
        \draw (1,1) node[anchor=south] {$w$};
        \draw[blue](0,0.4) node {1};
        \draw[blue] (2,0.4) node {4};
        \draw[blue](0.5,0.4) node {2};
        \draw[blue](1.5,0.4) node {3};
    \end{tikzpicture}&
    \begin{tikzpicture}
        \foreach \i in {0,0.5,1,1.5}{
            \foreach \j in {0,0.5,1,1.5}{
            \vertex{\i,\j}
            }
    }
    \vertexblue{0.5,0.5}\vertexblue{1,0.5}
    \vertexblue{1,1}
    \vertexblue{0.5,1}
    \draw(0,0)--(1.5,0);
    \draw(0,0)--(0,1.5);
    \draw(1.5,1.5)--(0,1.5);
    \draw(1.5,1.5)--(1.5,0);
    \draw[blue] (0.3,0.1) node[] {1};
    \draw[blue] (1.2,0.1) node[] {3};
    \draw[blue] (0.3,1.1) node[] {2};
    \draw[blue] (1.2,1.1) node[] {4};
    \draw[] (0.5,0.5)--(1,0);
    \draw[blue] (0.5,1)--(1,0);
    \draw[blue] (0.5,1)--(1,0.5);
    \draw[blue] (0.5,1.5)--(1,0.5);
    \draw[] (0.5,1.5)--(1,1);
    \draw[blue] (0.5,0.5)--(0.5,1);
    \draw[] (0.5,1.5)--(0.5,1);
    \draw[blue] (1,0.5)--(1,1);
    \draw[] (1,0.5)--(1,0);
    \end{tikzpicture}
    &
    \begin{tikzpicture}[scale=0.75]
    \draw[](0.5,0-0.5)--(1,0.5-0.5);
    \draw[blue](1,0.5-0.5)--(2,1-0.5);
    \draw[blue](2,1-0.5)--(3,2-0.5);
    \draw[blue](3,2-0.5)--(4,2.5-0.5);
    \draw[](0,2-0.5)--(3,2-0.5);
    \draw[](4,2.5-0.5)--(5,3.5-0.5);
    \vertexblue{3,2-0.5}
    \draw[blue](4,2.5-0.5)--(5.5,2.5-0.5);
    \draw[](2,1-0.5)--(4,1-0.5);
    \draw[blue](0,0.5-0.5)--(1,0.5-0.5);
    \draw[](5,3.5-0.5)--(5,4-0.5);
    \vertexblue{0,0.5-0.5}
    \vertexblue{2,1-0.5}
    \vertexblue{1,0.5-0.5}
    \vertexblue{5.5,2.5-0.5}
    \vertexblue{4,2.5-0.5}
    \draw[] (1,0.5-0.5) node[anchor=west] {$x$};
    \draw[] (4,2.5-0.5) node[anchor=south ] {$v$};
    \draw[] (0,0.5-0.5) node[anchor= east] {$u$};
    \draw[] (5.5,2.5-0.5) node[anchor= south] {$y$};
    \draw[] (2,1-0.5) node[anchor= north] {$z$};
    \draw[] (3,2-0.5) node[anchor= south] {$w$};
    \end{tikzpicture}
    \end{tikzcd}\caption{Combinatorial type, partial triangulations of $\mathbb F_0$ and partial embedded tropical curve. }\label{fg:(020)F_0}
    \end{figure}

Comparing the $x$-coordinates, resp. $y$-coordinates, of the points gives $2xz+zw<uw$ and $2vw+zw<yz,$ resp. $xz+zw<uw$ and $vw+zw<yz,$

Instead, for the triangulation of $\mathbb F_2,$ we have two possibilities, depicted in Figures \ref{fg:(020)F_2a} and \ref{fg:(020)F_2b}.

\begin{figure}[h]
\begin{tikzcd}
    \begin{tikzpicture}
        \foreach \j in {0,0.5,1,1.5,2,2.5,3}{
            \vertex{0,\j}
        }
        \foreach \j in {0,0.5,1,1.5,2}{
            \vertex{0.5,\j}
        }
        \foreach \j in {0,0.5,1}{
            \vertex{1,\j}
        }
        \vertex{1.5,0}
        \draw (0,0)--(1.5,0);
        \draw (0,3)--(1.5,0);
        \draw (0,3)--(0,0);
        \coordinate (1) at (1,0.5);
        \coordinate (2) at (0.5,0.5);
        \coordinate (3) at (0.5,1);
        \coordinate (4) at (0.5,1.5);
        \foreach \j in {1,2,3,4}{
            \vertexblue{\j}
        }
        \draw[blue] (1) node[anchor=north] {1};
        \draw[blue] (2) node[anchor=north] {2};
        \draw[blue] (3) node[anchor=east] {3};
        \draw[blue] (4) node[anchor=east] {4};
        \draw[blue] (1)--(2);
        \draw[blue] (2)--(4);
        \draw[blue] (2)--(1,1);
    \end{tikzpicture}&
    \begin{tikzpicture}
        \foreach \j in {0,0.5,1,1.5,2,2.5,3}{
            \vertex{0,\j}
        }
        \foreach \j in {0,0.5,1,1.5,2}{
            \vertex{0.5,\j}
        }
        \foreach \j in {0,0.5,1}{
            \vertex{1,\j}
        }
        \vertex{1.5,0}
        \draw (0,0)--(1.5,0);
        \draw (0,3)--(1.5,0);
        \draw (0,3)--(0,0);
        \coordinate (1) at (1,0.5);
        \coordinate (2) at (0.5,0.5);
        \coordinate (3) at (0.5,1);
        \coordinate (4) at (0.5,1.5);
        \foreach \j in {1,2,3,4}{
            \vertexblue{\j}
        }
        \draw[blue] (1) node[anchor=north] {1};
        \draw[blue] (2) node[anchor=north] {2};
        \draw[blue] (3) node[anchor=east] {3};
        \draw[blue] (4) node[anchor=east] {4};
        \draw[blue] (1)--(2);
        \draw[blue] (2)--(4);
        \draw[blue] (2)--(1,1);
        \draw (1)--(1,1);\draw (3)--(1,1); \draw (4)--(1,1);
    \draw(1,0.5)--(1.5,0); 
    \end{tikzpicture}&
    \begin{tikzpicture}[scale=0.75]
    \draw[blue](2.5,0.5)--(1.5,1.5);
    \draw[](2.5,0.5)--(3.5,0.5);
    \draw[](2.8,-0.2)--(3.5,0.5);
    \draw[](3.5,0.5)--(4.5,1);
    \draw[blue](2.5,0.5)--(2.5,0);
    \draw[](1.5,1.5)--(1.5,2.5);
    \draw[](1.5,2.5)--(2,3);
    \draw[blue](0.5,2.5)--(1.5,2.5);
    \draw[blue](1.5,1.5)--(1,1.5);
    \vertexblue{2.5,0.5}\vertexblue{1,1.5}
    \vertexblue{2.5,0}
    \vertexblue{1.5,1.5}
    \vertexblue{0.5,2.5}
    \vertexblue{1.5,2.5}
    \draw[] (1,1.5) node[anchor= east] {$w$};
    \draw[] (1.5,2.5) node[anchor=west] {$y$};
    \draw[] (2.5,0.5) node[anchor= south] {$x$};
    \draw[] (1.5,1.5) node[anchor= west] {$z$};
    \draw[] (2.5,0) node[anchor= north] {$u$};
    \draw[] (0.5,2.5) node[anchor= east] {$v$};
    \end{tikzpicture}
    \end{tikzcd}\caption{Partial triangulation of $\mathbb F_2$ and partial embedded tropical curve.}\label{fg:(020)F_2a}
\end{figure}
    
\begin{figure}[h]
\begin{tikzcd}
    \begin{tikzpicture}
        \foreach \j in {0,0.5,1,1.5,2,2.5,3}{
            \vertex{0,\j}
        }
        \foreach \j in {0,0.5,1,1.5,2}{
            \vertex{0.5,\j}
        }
        \foreach \j in {0,0.5,1}{
            \vertex{1,\j}
        }
        \vertex{1.5,0}
        \draw (0,0)--(1.5,0);
        \draw (0,3)--(1.5,0);
        \draw (0,3)--(0,0);
        \coordinate (2) at (1,0.5);
        \coordinate (1) at (0.5,0.5);
        \coordinate (3) at (0.5,1);
        \coordinate (4) at (0.5,1.5);
        \foreach \j in {1,2,3,4}{
            \vertexblue{\j}
        }
        \draw[blue] (1) node[anchor=north] {1};
        \draw[blue] (2) node[anchor=north] {2};
        \draw[blue] (3) node[anchor=east] {3};
        \draw[blue] (4) node[anchor=east] {4};
        \draw[blue] (2)--(3);
        \draw[blue] (1)--(2);
        \draw[blue] (3)--(4);
        \draw[blue] (2)--(0,1);
    \end{tikzpicture}&
    \begin{tikzpicture}
        \foreach \j in {0,0.5,1,1.5,2,2.5,3}{
            \vertex{0,\j}
        }
        \foreach \j in {0,0.5,1,1.5,2}{
            \vertex{0.5,\j}
        }
        \foreach \j in {0,0.5,1}{
            \vertex{1,\j}
        }
        \vertex{1.5,0}
         \draw (0,0)--(1.5,0);
        \draw (0,3)--(1.5,0);
        \draw (0,3)--(0,0);
        \coordinate (2) at (1,0.5);
        \coordinate (1) at (0.5,0.5);
        \coordinate (3) at (0.5,1);
        \coordinate (4) at (0.5,1.5);
        \foreach \j in {1,2,3,4}{
            \vertexblue{\j}
        }
        \draw[blue] (1) node[anchor=north] {1};
        \draw[blue] (2) node[anchor=north] {2};
        \draw[blue] (3) node[anchor=east] {3};
        \draw[blue] (4) node[anchor=east] {4};
        \draw[blue] (2)--(3);
        \draw[blue] (1)--(2);
        \draw[blue] (3)--(4);
        \draw[blue] (2)--(0,1);
        \draw (0,1)--(3);\draw (0,1)--(1); \draw (4)--(1,1);\draw (2)--(1,1);
    \draw(1,0.5)--(1.5,0); 
    \end{tikzpicture}&
    \begin{tikzpicture}[scale=0.75]
    \draw[](1.5,2.5+1)--(2,3+1);
    \draw[blue](0,2.5+1)--(1.5,2.5+1);
    \draw[](0,2.5+1)--(0,0.5+1);
    \draw[blue](-0.5,-0.5+1)--(0,0.5+1);
    \draw[blue](0,0.5+1)--(0.5,1+1);
    \draw[blue](-0.5,-0.5+1)--(-0.5,-1+1);
    \draw[](-1,-1+1)--(-0.5,-0.5+1);
    \draw[](0,-0.5+1)--(1.5,1+1);
    \draw[](0.5,1+1)--(1.5,1+1);
    \draw[](0,-0.5+1)--(0,-0.7+1);
    \draw[](2.5,1.5+1)--(1.5,1+1);
    \vertexblue{0,0.5+1}\vertexblue{0.5,1+1}
    \vertexblue{-0.5,-0.5+1}
    \vertexblue{1.5,1+1}
    \vertexblue{0,2.5+1}
    \vertexblue{1.5,2.5+1}
    \draw[] (0.5,1+1) node[anchor= south] {w};
    \draw[] (1.5,2.5+1) node[anchor=west] {v};
    \draw[] (-0.5,-0.5+1) node[anchor= south east] {x};
    \draw[] (0,0.5+1) node[anchor= east] {z};
    \draw[] (1.5,1+1) node[anchor= south] {u};
    \draw[] (0,2.5+1) node[anchor= east] {y};
    \end{tikzpicture}
    \end{tikzcd}\caption{Partial triangulation of $\mathbb F_2$ and partial embedded tropical curve.}\label{fg:(020)F_2b}
    \end{figure}
The first possibility in Figure \ref{fg:(020)F_2a} shows that the triangulation can be completed by joining the points $v,w$ with edges with corresponding segments $(\alpha, 1),$ thus $vw=yz.$

The triangulation depicted in Figure \ref{fg:(020)F_2b}, instead, can be completed by joining the points $z,y$ with edges defined by segment of type $(0,1)$ or $(2, 1).$
Both have same $y$-coordinate, therefore $yz=vw+wz.$

This proves the following lemma.
\begin{lemma}
    Let $\Gamma$ be a tropical curve of genus $4$ of combinatorial type (020).
    \begin{itemize}
        \item If $\Gamma$ is realizable in $\mathbb F_0,$ then its edge lengths satisfy, up to symmetry:
        \begin{equation}
        \begin{cases}
            2xz+wz<uw\\2vw+wz<yz.        \end{cases}
        \end{equation}
        \item If $\Gamma$ is realizable in $\mathbb F_2,$ then its edge lengths satisfy, up to symmetry:
        \begin{equation}
        \begin{cases}
            xz+wz<uw\\vw=yz;
        \end{cases}\text{or}\quad
        \begin{cases}
            2xz+wz<uw\\
            yz=vw+wz.
        \end{cases}
        \end{equation}
    \end{itemize}
\end{lemma}

\subsubsection*{Type (021)}

Let us now consider the combinatorial type (021). The triangulation describing the embedding in $\mathbb F_0$ is represented in Figure \ref{fg:(021)F_0}.
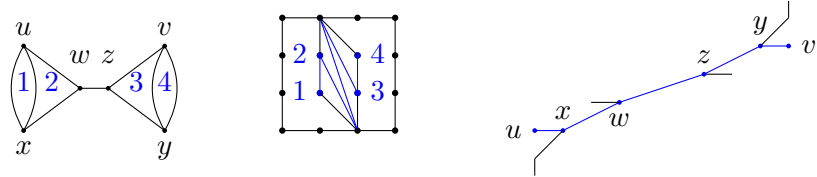
\begin{figure}[h]
\begin{tikzcd}
      \begin{tikzpicture}[scale=0.75]
        \coordinate (1) at (0,0);
        \coordinate (2) at (1,0.75);
        \coordinate (3) at (2.5,0);
        \coordinate (4) at (0,1.5);
        \coordinate (5) at (2.5,1.5);
        \coordinate (6) at (1.5,0.75);
        \draw (1)--(2);
        \draw (2)--(6);
        \draw (6)--(3);
        \draw (4)--(2);
        \draw (6)--(5);
        \draw(1)[] to [out=120, in=240] (4);
        \draw(1)[] to [out=60, in=300] (4);
        \draw(3)[] to [out=120, in=240] (5);
        \draw(3)[] to [out=60, in=300] (5);
        \foreach \i in {1,2,3,4,5,6}{
            \vertex{\i}
        }

        \draw (0,0) node[anchor=north] {$x$};
        \draw (2.5,0) node[anchor=north] {$y$};
        \draw(0,1.5) node[anchor=south] {$u$};
        \draw (2.5,1.5) node[anchor=south] {$v$};
        \draw(1.5,1) node[anchor=south] {$z$};
        \draw (1,1) node[anchor=south] {$w$};
        \draw[blue](0,0.7) node {1};
        \draw[blue] (2.5,0.7) node {4};
        \draw[blue](0.5,0.7) node {2};
        \draw[blue](2,0.7) node {3};
    \end{tikzpicture}&
    \begin{tikzpicture}
        \foreach \i in {0,0.5,1,1.5}{
            \foreach \j in {0,0.5,1,1.5}{
            \vertex{\i,\j}
            }
    }
    \coordinate (1) at (0.5,0.5);
    \coordinate (2) at (0.5,1);
    \coordinate (3) at (1,0.5);
    \coordinate (4) at (1,1);
        \foreach \j in {1,2,3,4}{
            \vertexblue{\j}
        }
    \draw[blue] (1) node[anchor=east] {1};\draw[blue] (2) node[anchor=east] {2};
    \draw[blue] (3) node[anchor=west] {3};\draw[blue] (4) node[anchor=west] {4};
    \draw(0,0)--(1.5,0);
    \draw(0,0)--(0,1.5);
    \draw(1.5,1.5)--(0,1.5);
    \draw(1.5,1.5)--(1.5,0);
    \draw[blue] (1)--(2);
    \draw[blue] (3)--(4);
    \draw[blue] (0.5,1.5)--(1,0);
    \draw[blue] (2)--(1,0);
    \draw[blue] (0.5,1.5)--(3);
    \draw[] (1)--(1,0);\draw[] (3)--(1,0);
    \draw[] (2)--(0.5,1.5);\draw[] (4)--(0.5,1.5);
    \end{tikzpicture}
    &
    \begin{tikzpicture}[scale=0.75]
    \draw[blue](0.5,0)--(1,0);
    \draw[blue](1,0)--(2,0.5);
    \draw[blue](2,0.5)--(3.5,1);
    \draw[blue](3.5,1)--(4.5,1.5);
    \vertexblue{2,0.5}
    \draw[](2,0.5)--(1.5,0.5);\draw (3.5,1)--(4,1);
    \vertexblue{3.5,1}\vertexblue{4.5,1.5}\vertexblue{5,1.5}
    \vertexblue{1,0}
    \vertexblue{0.5,0}
    \draw[blue](4.5,1.5)--(5,1.5);
    \draw[](4.5,1.5)--(5,2);
    \draw[](5,2.3)--(5,2);
    \draw[](1,0)--(0.5,-0.5);
    \draw[](0.5,-0.8)--(0.5,-0.5);
    \draw[] (1,0) node[anchor=south] {$x$};
    \draw[] (5,1.5) node[anchor=west ] {$v$};
    \draw[] (0.5,0) node[anchor= east] {$u$};
    \draw[] (4.5,1.5) node[anchor= south] {$y$};
    \draw[] (2,0.5) node[anchor= north] {$w$};
    \draw[] (3.5,1) node[anchor= south] {$z$};
    \end{tikzpicture}
    \end{tikzcd}\caption{Combinatorial type, partial triangulations of $\mathbb F_0$ and partial embedded tropical curve.}\label{fg:(021)F_0}
    \end{figure}

Comparing $x$-coordinates, resp. $y$-coordinates, of the points yields relations $ux+2wx<uw$ and $vy+2yz<vz$, resp.  
$wx<uw$ and $yz<vz.$
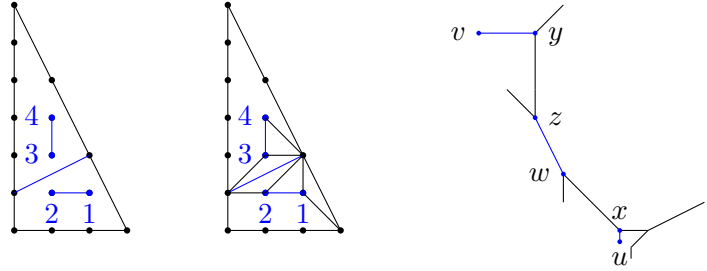
\begin{figure}[h]
\begin{tikzcd}
    \begin{tikzpicture}
        \foreach \j in {0,0.5,1,1.5,2,2.5,3}{
            \vertex{0,\j}
        }
        \foreach \j in {0,0.5,1,1.5,2}{
            \vertex{0.5,\j}
        }
        \foreach \j in {0,0.5,1}{
            \vertex{1,\j}
        }
        \vertex{1.5,0}
        \draw (0,0)--(1.5,0);
        \draw (0,3)--(1.5,0);
        \draw (0,3)--(0,0);
        \coordinate (1) at (1,0.5);
        \coordinate (2) at (0.5,0.5);
        \coordinate (3) at (0.5,1);
        \coordinate (4) at (0.5,1.5);
        \foreach \j in {1,2,3,4}{
            \vertexblue{\j}
        }
        \draw[blue] (1) node[anchor=north] {1};
        \draw[blue] (2) node[anchor=north] {2};
        \draw[blue] (3) node[anchor=east] {3};
        \draw[blue] (4) node[anchor=east] {4};
        \draw[blue] (1)--(2);
        \draw[blue] (3)--(4);
        \draw[blue] (0,0.5)--(1,1);
    \end{tikzpicture}&
    \begin{tikzpicture}
        \foreach \j in {0,0.5,1,1.5,2,2.5,3}{
            \vertex{0,\j}
        }
        \foreach \j in {0,0.5,1,1.5,2}{
            \vertex{0.5,\j}
        }
        \foreach \j in {0,0.5,1}{
            \vertex{1,\j}
        }
        \vertex{1.5,0}
        \draw (0,0)--(1.5,0);
        \draw (0,3)--(1.5,0);
        \draw (0,3)--(0,0);
        \coordinate (1) at (1,0.5);
        \coordinate (2) at (0.5,0.5);
        \coordinate (3) at (0.5,1);
        \coordinate (4) at (0.5,1.5);
        \foreach \j in {1,2,3,4}{
            \vertexblue{\j}
        }
        \draw[blue] (1) node[anchor=north] {1};
        \draw[blue] (2) node[anchor=north] {2};
        \draw[blue] (3) node[anchor=east] {3};
        \draw[blue] (4) node[anchor=east] {4};
         \draw[blue] (1)--(2);
        \draw[blue] (3)--(4);
        \draw[blue] (0,0.5)--(1,1);
        \draw[] (0,0.5)--(3);
        \draw[] (1,1)--(2);
        \draw[] (0,0.5)--(2);
        \draw (1)--(1,1);\draw (3)--(1,1); \draw (4)--(1,1);
    \draw(1,0.5)--(1.5,0); 
    \end{tikzpicture}&
    \begin{tikzpicture}[scale=0.75]
    \draw[](1.5,1.5+1)--(1.5,2.5+1);
    \draw[](1.5,2.5+1)--(2,3+1);
    \draw[blue](0.5,2.5+1)--(1.5,2.5+1);
    \draw[](1.5,1+1)--(1,1.5+1);
    \draw[](1.5,1.5+1)--(1.5,1+1);
    \draw[blue](1.5,1+1)--(2,0+1);
    \draw[](2,0+1)--(2,-0.5+1);
    \draw[](2,0+1)--(3,-1+1);
    \draw[blue](3,-1.2+1)--(3,-1+1);
    \draw[](3.5,-1+1)--(3,-1+1);
    \draw[](3.5,-1+1)--(4.5,-0.5+1);
    \draw[](3.5,-1+1)--(3.2,-1.3+1);
    \draw[](3.2,-1.5+1)--(3.2,-1.3+1);
    \vertexblue{3,-1+1}\vertexblue{2,0+1}
    \vertexblue{3,-1.2+1}
    \vertexblue{1.5,1+1}
    \vertexblue{0.5,2.5+1}
    \vertexblue{1.5,2.5+1}
    \draw[] (2,0+1) node[anchor= east] {$w$};
    \draw[] (1.5,2.5+1) node[anchor=west] {$y$};
    \draw[] (3,-1+1) node[anchor= south] {$x$};
    \draw[] (1.5,1+1) node[anchor= west] {$z$};
    \draw[] (3,-1.2+1) node[anchor= north] {$u$};
    \draw[] (0.5,2.5+1) node[anchor= east] {$v$};
    \end{tikzpicture}
    \end{tikzcd}\caption{Partial triangulation of $\mathbb F_2$ and partial embedded tropical curve.}\label{fg:(021)F_2}
\end{figure}

The triangulation in Figure \ref{fg:(021)F_2} can be completed by joining $z$ with $v$ only via segments of the form $(\alpha,1).$ Then $vz=yz.$
Instead, let us notice that the total length of the edges joining $u,w$ depends on how the triangulation is completed. In Figure \ref{fg:(021)codim} we exhibit a triangulation for which the edge lengths do not satisfy any additional equation.

\begin{figure}[h]
\begin{tikzcd}
    \begin{tikzpicture}
        \foreach \j in {0,0.5,1,1.5,2,2.5,3}{
            \vertex{0,\j}
        }
        \foreach \j in {0,0.5,1,1.5,2}{
            \vertex{0.5,\j}
        }
        \foreach \j in {0,0.5,1}{
            \vertex{1,\j}
        }
        \vertex{1.5,0}
        \draw (0,0)--(1.5,0);
        \draw (0,3)--(1.5,0);
        \draw (0,3)--(0,0);
        \coordinate (1) at (1,0.5);
        \coordinate (2) at (0.5,0.5);
        \coordinate (3) at (0.5,1);
        \coordinate (4) at (0.5,1.5);
        \foreach \j in {1,2,3,4}{
            \vertexblue{\j}
        }
        \draw[blue] (1) node[anchor=north] {1};
        \draw[blue] (2) node[anchor=north] {2};
        \draw[blue] (3) node[anchor=east] {3};
        \draw[blue] (4) node[anchor=east] {4};
         \draw[blue] (1)--(2);
        \draw[blue] (3)--(4);
        \draw[blue] (0,0.5)--(1,1);
        \draw[] (0,0.5)--(3);
        \draw[] (1,1)--(2);
        \draw[] (0,0.5)--(2);
        \draw (1)--(1,1);\draw (3)--(1,1); \draw (4)--(1,1);
    \draw(1,0.5)--(1.5,0); 
    \draw[red](0,0.5)--(0.5,0);
    \draw[red](0.5,0.5)--(0.5,0);
    \draw[red](0.5,0.5)--(1,0);
    \draw[red](1,0)--(1);
    \end{tikzpicture}&
    \begin{tikzpicture}[scale=0.75]
    \draw[](1.5,1.5+1.5)--(1.5,2.5+1.5);
    \draw[](1.5,2.5+1.5)--(2,3+1.5);
    \draw[blue](0.5,2.5+1.5)--(1.5,2.5+1.5);
    \draw[](1.5,1+1.5)--(1,1.5+1.5);
    \draw[](1.5,1.5+1.5)--(1.5,1+1.5);
    \draw[blue](1.5,1+1.5)--(2,0+1.5);
    \draw[](2,0+1.5)--(2,-1.7+1.5);
    \draw[](2,0+1.5)--(3,-1+1.5);
    \draw[blue](3,-1.2+1.5)--(3,-1+1.5);
    \draw[](3.5,-1+1.5)--(3,-1+1.5);
    \draw[](3.5,-1+1.5)--(4.5,-0.5+1.5);
    \draw[](3.5,-1+1.5)--(3.3,-1.2+1.5);
    \draw[](3.3,-1.5+1.5)--(3.3,-1.2+1.5);
    \draw[red](3,-1.2+1.5)--(3.3,-1.2+1.5);
    \draw[red](3,-1.2+1.5)--(2.5,-1.7+1.5);
    \draw(2.5,-2.2+1.5)--(2.5,-1.7+1.5);
    \draw[red](2.5,-1.7+1.5)--(2,-1.7+1.5);
    \draw[red](1.7,-2+1.5)--(2,-1.7+1.5);
    \draw(1.7,-2+1.5)--(1.5,-2+1.5);
    \draw(1.7,-2+1.5)--(1.7,-2.2+1.5);
    \vertexblue{3,-1+1.5}\vertexblue{2,0+1.5}
    \vertexblue{3,-1.2+1.5}
    \vertexblue{1.5,1+1.5}
    \vertexblue{0.5,2.5+1.5}
    \vertexblue{1.5,2.5+1.5}
    \draw[] (2,0+1.5) node[anchor= east] {$w$};
    \draw[] (1.5,2.5+1.5) node[anchor=west] {$y$};
    \draw[] (3,-1+1.5) node[anchor= south] {$x$};
    \draw[] (1.5,1+1.5) node[anchor= west] {$z$};
    \draw[] (3,-1.2+1.5) node[anchor= north] {$u$};
    \draw[] (0.5,2.5+1.5) node[anchor= east] {$v$};
    \end{tikzpicture}
    \end{tikzcd}\caption{A possible triangulation of $\mathbb F_2$ and the corresponding embedded curve.}\label{fg:(021)codim}
\end{figure}
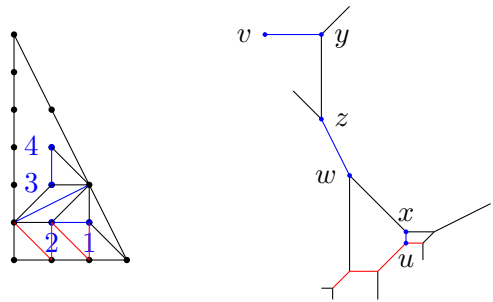

This proves the following Lemma.
\begin{lemma}
    Let $\Gamma$ be a tropical curve of genus $4$ of combinatorial type (021).
    \begin{itemize}
        \item If $\Gamma$ is realizable in $\mathbb F_0,$ then its edge lengths satisfy, up to symmetry:
        \begin{equation}
        \begin{cases}
            ux+2wx<uw;\\
            vy+2yz<vz.
        \end{cases}
        \end{equation}
        \item If $\Gamma$ is realizable in $\mathbb F_2,$ then its edge lengths satisfy, up to symmetry:
        \begin{equation}
        \begin{cases}
            yz=vz;
        \end{cases}
        \end{equation}
    \end{itemize}
\end{lemma}

\subsubsection*{Type (030)}

Let us now consider the combinatorial type (030). The triangulation describing the embedding in $\mathbb F_0$ is represented in Figure \ref{fg:(030)F_0}.
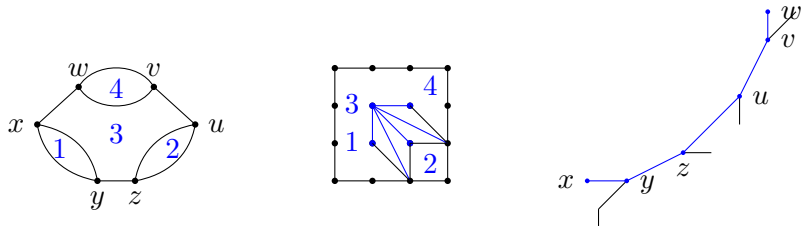
\begin{figure}[h]
\begin{tikzcd}
      \begin{tikzpicture}
        \coordinate (1) at (-0.3,0.75);
        \coordinate (2) at (0.5,0);
        \coordinate (3) at (1,0);
        \coordinate (4) at (1.8,0.75);
        \coordinate (5) at (1.25,1.25);
        \coordinate (6) at (0.25,1.25);
        \draw (2)--(3);
        \draw (1)--(6);
        \draw (4)--(5);
        \draw(1)[] to [out=350, in=100] (2);
        \draw(1)[] to [out=280, in=170] (2);
        \draw(3)[] to [out=10, in=260] (4);
        \draw(3)[] to [out=80, in=190] (4);
        \draw(6)[] to [out=60, in=120] (5);
        \draw(6)[] to [out=300, in=240] (5);
        \foreach \i in {1,2,3,4,5,6}{
            \vertex{\i}
        }
        \draw(1) node[anchor=east] {$x$};
        \draw(2) node[anchor=north] {$y$};
        \draw(3) node[anchor=north] {$z$};
        \draw(4) node[anchor=west] {$u$};
        \draw(5) node[anchor=south] {$v$};
        \draw(6) node[anchor=south] {$w$};
        \draw[blue](0,0.3) node {1};
        \draw[blue] (0.75,1.1) node {4};
        \draw[blue](1.5,0.3) node {2};
        \draw[blue](0.75,0.5) node {3};
    \end{tikzpicture}&
    \begin{tikzpicture}
        \foreach \i in {0,0.5,1,1.5}{
            \foreach \j in {0,0.5,1,1.5}{
            \vertex{\i,\j}
            }
    }
    \coordinate (1) at (0.5,0.5);
    \coordinate (3) at (0.5,1);
    \coordinate (2) at (1,0.5);
    \coordinate (4) at (1,1);
        \foreach \j in {1,2,3,4}{
            \vertexblue{\j}
        }
    \draw[blue] (1) node[anchor=east] {1};\draw[blue] (2) node[anchor=north west] {2};
    \draw[blue] (3) node[anchor=east] {3};\draw[blue] (4) node[anchor=south west] {4};
    \draw(0,0)--(1.5,0);
    \draw(0,0)--(0,1.5);
    \draw(1.5,1.5)--(0,1.5);
    \draw(1.5,1.5)--(1.5,0);
    \draw[blue] (3)--(2);
    \draw[blue] (3)--(4);
    \draw[blue] (3)--(1);
    \draw[blue] (3)--(1.5,0.5);
    \draw[blue] (3)--(1,0);
    \draw[] (4)--(1.5,0.5);
    \draw[] (2)--(1.5,0.5);
    \draw[] (2)--(1,0);
    \draw[] (1)--(1,0);
    \end{tikzpicture}
    &
    \begin{tikzpicture}[scale=0.75]
    \draw (-0.5,-0.5)--(-0.5,-0.8);
    \draw (-0.5,-0.5)--(0,0);
    \draw[blue] (0,0)--(1,0.5);\draw[blue] (0,0)--(-0.7,0);
    \draw[blue] (1,0.5)--(2,1.5);\draw[] (1,0.5)--(1.5,0.5);
    \draw[] (2,1.5)--(2,1);
    \draw[blue] (2,1.5)--(2.5,2.5);
    \draw[] (2.5,2.5)--(3,3);
    \draw[blue] (2.5,2.5)--(2.5,3);
    \vertexblue{1,0.5}\vertexblue{2,1.5}\vertexblue{2.5,2.5}
    \vertexblue{2.5,3}
    \vertexblue{0,0}
    \vertexblue{-0.7,0}
    \draw[] (1,0.5) node[anchor=north] {$z$};
    \draw[] (2,1.5) node[anchor=west] {$u$};
    \draw[] (2.5,2.5) node[anchor=west] {$v$};
    \draw[] (2.5,3) node[anchor=west] {$w$};
    \draw[] (0,0) node[anchor=west] {$y$};
    \draw[] (-0.7,0) node[anchor=east] {$x$};
    \end{tikzpicture}
    \end{tikzcd}\caption{Combinatorial type, partial triangulations of $\mathbb F_0$ and partial embedded tropical curve.}\label{fg:(030)F_0}
    \end{figure}

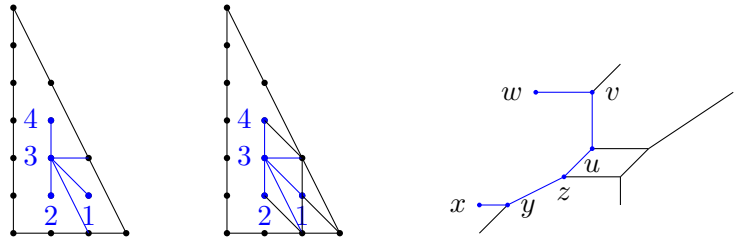
\begin{figure}[h]
\begin{tikzcd}
    \begin{tikzpicture}
        \foreach \j in {0,0.5,1,1.5,2,2.5,3}{
            \vertex{0,\j}
        }
        \foreach \j in {0,0.5,1,1.5,2}{
            \vertex{0.5,\j}
        }
        \foreach \j in {0,0.5,1}{
            \vertex{1,\j}
        }
        \vertex{1.5,0}
        \draw (0,0)--(1.5,0);
        \draw (0,3)--(1.5,0);
        \draw (0,3)--(0,0);
        \coordinate (1) at (1,0.5);
        \coordinate (2) at (0.5,0.5);
        \coordinate (3) at (0.5,1);
        \coordinate (4) at (0.5,1.5);
        \foreach \j in {1,2,3,4}{
            \vertexblue{\j}
        }
        \draw[blue] (1) node[anchor=north] {1};
        \draw[blue] (2) node[anchor=north] {2};
        \draw[blue] (3) node[anchor=east] {3};
        \draw[blue] (4) node[anchor=east] {4};
        \draw[blue] (3)--(2);
        \draw[blue] (3)--(1);
        \draw[blue] (3)--(4);
        \draw[blue] (3)--(1,0);
        \draw[blue] (3)--(1,1);
    \end{tikzpicture}&
    \begin{tikzpicture}
        \foreach \j in {0,0.5,1,1.5,2,2.5,3}{
            \vertex{0,\j}
        }
        \foreach \j in {0,0.5,1,1.5,2}{
            \vertex{0.5,\j}
        }
        \foreach \j in {0,0.5,1}{
            \vertex{1,\j}
        }
        \vertex{1.5,0}
        \draw (0,0)--(1.5,0);
        \draw (0,3)--(1.5,0);
        \draw (0,3)--(0,0);
        \coordinate (1) at (1,0.5);
        \coordinate (2) at (0.5,0.5);
        \coordinate (3) at (0.5,1);
        \coordinate (4) at (0.5,1.5);
        \foreach \j in {1,2,3,4}{
            \vertexblue{\j}
        }
        \draw[blue] (1) node[anchor=north] {1};
        \draw[blue] (2) node[anchor=north] {2};
        \draw[blue] (3) node[anchor=east] {3};
        \draw[blue] (4) node[anchor=east] {4};
        \draw[blue] (3)--(2);
        \draw[blue] (3)--(1);
        \draw[blue] (3)--(4);
        \draw[blue] (3)--(1,0);
        \draw[blue] (3)--(1,1);
        \draw[] (2)--(1,0);
        \draw[] (1)--(1,0);
        \draw (4)--(1,1);
    \draw(1,0.5)--(1.5,0); 
    \draw[] (1,1)--(1,0.5);
    \end{tikzpicture}&
    \begin{tikzpicture}[scale=0.75]
    \draw[blue](1.5,1.5)--(1.5,2.5);
    \draw[](1.5,2.5)--(2,3);
    \draw[blue](0.5,2.5)--(1.5,2.5);
    \draw[blue](1.5,1.5)--(1,1);
    \draw[blue](0,0.5)--(1,1);
    \draw[blue](0,0.5)--(-0.5,0.5);
    \draw[](0,0.5)--(-0.5,0);
    \draw[](1,1)--(2,1);
    \draw[](2,0.5)--(2,1);
    \draw[](2.5,1.5)--(2,1);\draw[](2.5,1.5)--(4,2.5);
    \draw[](2.5,1.5)--(1.5,1.5);
    \vertexblue{1,1}\vertexblue{0,0.5}
    \vertexblue{-0.5,0.5}
    \vertexblue{1.5,1.5}
    \vertexblue{0.5,2.5}
    \vertexblue{1.5,2.5}
    \draw[] (0,0.5) node[anchor= west] {$y$};
    \draw[] (1.5,2.5) node[anchor=west] {$v$};
    \draw[] (-0.5,0.5) node[anchor= east] {$x$};
    \draw[] (1,1) node[anchor= north] {$z$};
    \draw[] (1.5,1.5) node[anchor= north] {$u$};
    \draw[] (0.5,2.5) node[anchor= east] {$w$};
    \end{tikzpicture}
    \end{tikzcd}\caption{Partial triangulation of $\mathbb F_2$ and partial embedded tropical curve.}\label{fg:(030)F_2}
\end{figure}

Again, comparing $x$-coordinates and $y$-coordinates of the points yields the following.
\begin{lemma}
    Let $\Gamma$ be a tropical curve of genus $4$ of combinatorial type (030).
    \begin{itemize}
        \item If $\Gamma$ is realizable in $\mathbb F_0,$ then its edge lengths satisfy:
        \begin{equation}
        \begin{cases}
        2yz+uz+uv<xw;\\
        yz+uz+2uv<xw.\end{cases}
        \end{equation}
        \item If $\Gamma$ is realizable in $\mathbb F_0,$ then its edge lengths satisfy:
        \begin{equation}
            yz+uz+uv=xw.
        \end{equation}
    \end{itemize}
\end{lemma}

\subsubsection*{Type (101)}

Let us now consider the combinatorial type (101). The triangulation describing the embedding in $\mathbb F_0$ is represented in Figure \ref{fg:(101)F_0}.
\begin{figure}[h]
\begin{tikzcd}
      \begin{tikzpicture}
        \coordinate (1) at (-0.75,0.75);
        \coordinate (2) at (0,0);
        \coordinate (3) at (0,0.75);
        \coordinate (4) at (0,1.5);
        \coordinate (5) at (0.75,0.75);
        \coordinate (6) at (1.5,0.75);
        \draw (1)--(2);
        \draw (1)--(3);
        \draw (1)--(4);
        \draw (4)--(2);
        \draw (4)--(5);
        \draw (2)--(5);
        \draw (6)--(5);
        \draw (2,0.75) circle (0.5);
        \foreach \i in {1,2,3,4,5,6}{
            \vertex{\i}
        }
        \draw(1) node[anchor=east] {$x$};
        \draw(2) node[anchor=north] {$y$};
        \draw(3) node[anchor=west] {$z$};
        \draw(4) node[anchor=south] {$u$};
        \draw(5) node[anchor=south] {$v$};
        \draw(6)+(0.15,0) node[anchor=south east] {$w$};
        \draw[blue](-0.2,0.3) node {1};
        \draw[blue] (2,0.6) node {4};
        \draw[blue](-0.2,0.9) node {2};
        \draw[blue](0.5,0.6) node {3};
    \end{tikzpicture}&
    \begin{tikzpicture}
        \foreach \i in {0,0.5,1,1.5}{
            \foreach \j in {0,0.5,1,1.5}{
            \vertex{\i,\j}
            }
    }
    \coordinate (1) at (0.5,0.5);
    \coordinate (3) at (0.5,1);
    \coordinate (2) at (1,0.5);
    \coordinate (4) at (1,1);
        \foreach \j in {1,2,3,4}{
            \vertexblue{\j}
        }
    \draw[blue] (1) node[anchor=east] {1};\draw[blue] (2) node[anchor=north west] {2};
    \draw[blue] (3) node[anchor=east] {3};\draw[blue] (4) node[anchor=south west] {4};
    \draw(0,0)--(1.5,0);
    \draw(0,0)--(0,1.5);
    \draw(1.5,1.5)--(0,1.5);
    \draw(1.5,1.5)--(1.5,0);
    \draw[blue](0,1.5)--(1.5,0.5);
    \draw[blue] (3)--(2);
    \draw[blue] (3)--(1);
    \draw[blue] (2)--(1);
    \draw[] (4)--(1.5,0.5);\draw(4)--(0,1.5);
    \draw(3)--(0,1.5);\draw(3)--(1.5,0.5);
    \draw[] (2)--(1.5,0.5);
    \end{tikzpicture}
    &
\begin{tikzpicture}[scale=0.75]
    \draw[blue] (0,0)--(-1,0);\draw[blue] (0,0)--(0,-1);
    \draw[blue] (0,0)--(0.5,0.5);
    \draw[] (0.5,0.5)--(1,1.5);
    \draw[] (1,1.5)--(0.5,1);
    \draw[] (0.5,0)--(0.5,0.5);
    \draw[blue] (1,1.5)--(2,3);
    \draw[] (2,3)--(2.5,4);
    \draw[] (2,3)--(2.5,3.5);
    
    \vertexblue{1,1.5}\vertexblue{2,3}\vertexblue{0,0}\vertexblue{0.5,0.5}
    \vertexblue{-1,0}
    \draw[] (1,1.5) node[anchor=east] {$v$};
    \draw[] (2,3) node[anchor=west] {$w$};
    \draw[] (0,0) node[anchor=west] {$z$};
    \draw[] (0.5,0.5) node[anchor=west] {$u$};
    \draw[] (-1,0) node[anchor=east] {$y$};
    \end{tikzpicture}
    \end{tikzcd}\caption{Combinatorial type, partial triangulations of $\mathbb F_0$ and partial embedded tropical curve.}\label{fg:(101)F_0}
    \end{figure}
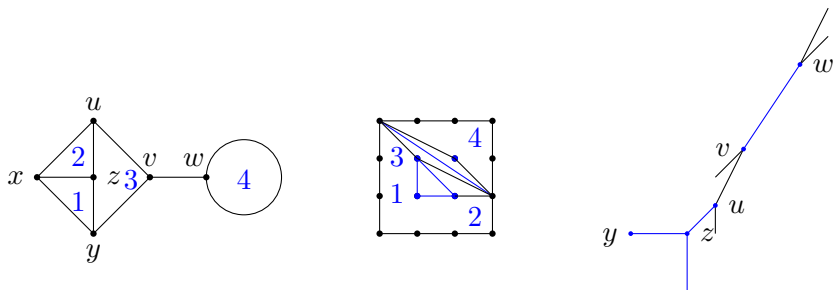

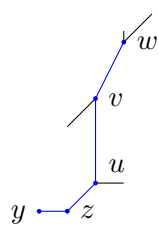
\begin{figure}[h]
\begin{tikzcd}
    \begin{tikzpicture}
        \foreach \j in {0,0.5,1,1.5,2,2.5,3}{
            \vertex{0,\j}
        }
        \foreach \j in {0,0.5,1,1.5,2}{
            \vertex{0.5,\j}
        }
        \foreach \j in {0,0.5,1}{
            \vertex{1,\j}
        }
        \vertex{1.5,0}
        \draw (0,0)--(1.5,0);
        \draw (0,3)--(1.5,0);
        \draw (0,3)--(0,0);
        \coordinate (1) at (1,0.5);
        \coordinate (2) at (0.5,0.5);
        \coordinate (3) at (0.5,1);
        \coordinate (4) at (0.5,1.5);
        \foreach \j in {1,2,3,4}{
            \vertexblue{\j}
        }
        \draw[blue] (1) node[anchor=north] {1};
        \draw[blue] (2) node[anchor=north] {2};
        \draw[blue] (3) node[anchor=east] {3};
        \draw[blue] (4) node[anchor=east] {4};
        \draw[blue] (3)--(2);
        \draw[blue] (3)--(1);
        \draw[blue] (0,1.5)--(1,1);
    \end{tikzpicture}&
    \begin{tikzpicture}
        \foreach \j in {0,0.5,1,1.5,2,2.5,3}{
            \vertex{0,\j}
        }
        \foreach \j in {0,0.5,1,1.5,2}{
            \vertex{0.5,\j}
        }
        \foreach \j in {0,0.5,1}{
            \vertex{1,\j}
        }
        \vertex{1.5,0}
        \draw (0,0)--(1.5,0);
        \draw (0,3)--(1.5,0);
        \draw (0,3)--(0,0);
        \coordinate (1) at (1,0.5);
        \coordinate (2) at (0.5,0.5);
        \coordinate (3) at (0.5,1);
        \coordinate (4) at (0.5,1.5);
        \foreach \j in {1,2,3,4}{
            \vertexblue{\j}
        }
        \draw[blue] (1) node[anchor=north] {1};
        \draw[blue] (2) node[anchor=north] {2};
        \draw[blue] (3) node[anchor=east] {3};
        \draw[blue] (4) node[anchor=south east] {4};
        \draw[blue] (3)--(2);
        \draw[blue] (3)--(1);
        \draw[blue] (0,1.5)--(1,1);
        \draw[blue] (3)--(1,1);
        \draw (4)--(1,1);\draw(4)--(0,1.5);\draw(3)--(0,1.5);
    \draw(1,0.5)--(1.5,0); 
    \draw[] (1,1)--(1,0.5);
    \end{tikzpicture}&
    \begin{tikzpicture}[scale=0.75]
    \draw[](1.5,2.5+1)--(2,3+1);
    \draw[blue](1,1.5+1)--(1.5,2.5+1);
    \draw[](1.5,2.7+1)--(1.5,2.5+1);
    \draw[](0.5,1+1)--(1,1.5+1);
    \draw[blue](1,0+1)--(1,1.5+1);
    \draw[blue](0.5,-0.5+1)--(1,0+1);
    \draw[blue](0.5,-0.5+1)--(0,-0.5+1);
    \draw[](1,0+1)--(1.5,0+1);
    \vertexblue{1,0+1}\vertexblue{1,1.5+1}
    \vertexblue{0,-0.5+1}
    \vertexblue{0.5,-0.5+1}
    \vertexblue{1.5,2.5+1}
    \draw[] (0.5,-0.5+1) node[anchor= west] {$z$};
    \draw[] (1.5,2.5+1) node[anchor=west] {$w$};
    \draw[] (0,-0.5+1) node[anchor= east] {$y$};
    \draw[] (1,0+1) node[anchor= south west] {$u$};
    \draw[] (1,1.5+1) node[anchor= west] {$v$};
    \end{tikzpicture}
    \end{tikzcd}\caption{Partial triangulation of $\mathbb F_2$ and partial embedded tropical curve.}\label{fg:(101)F_2}
\end{figure}

Both partial triangulation in Figures \ref{fg:(101)F_0}, \ref{fg:(101)F_2} can be completed by joining $y$ with $v$ either with one edge of segment $(1,1)$, two edges of segments $(1,1)$ and $(0,1)$ or three edges, with segments $(-1,1),$ $(0,1)$ and $(1,1)$, all with $y$-coordinate equal to $1.$
Comparing with the $y$-coordinates of the points $u,z$ gives respectively the equalities $vy=uz+2uv$ in $\FF_0$ and $vy=uv+uz$ in $\FF_2$. 

This proves the following Lemma.
\begin{lemma}
    Let $\Gamma$ be a tropical curve of genus $4$ of combinatorial type (101).
    \begin{itemize}
        \item If $\Gamma$ is realizable in $\mathbb F_0,$ then its edge lengths satisfy:
        \begin{equation}
        \begin{cases}
           2uv+uz=vy.
        \end{cases}
        \end{equation}
        \item If $\Gamma$ is realizable in $\mathbb F_2,$ then its edge lengths satisfy:
        \begin{equation}
        \begin{cases}
            uv+uz=vy.
        \end{cases}
        \end{equation}
    \end{itemize}
\end{lemma}

\subsubsection*{Type (111)}

Let us now consider the combinatorial type (111). The triangulation describing the embedding in $\mathbb F_0$ is represented in Figure \ref{fg:(111)F_0}.
\begin{figure}[h]
\begin{tikzcd}
      \begin{tikzpicture}
        \coordinate (1) at (-0.75,0);
        \coordinate (2) at (0,0);
        \coordinate (3) at (-0.75,1.5);
        \coordinate (4) at (0,1.5);
        \coordinate (5) at (0.75,0.75);
        \coordinate (6) at (1.5,0.75);
        \draw (1)--(2);
        \draw (3)--(4);
        \draw (4)--(2);
        \draw (4)--(5);
        \draw (2)--(5);
        \draw (6)--(5);
        \draw(1)[] to [out=120, in=240] (3);
        \draw(1)[] to [out=60, in=300] (3);
        \draw (2,0.75) circle (0.5);
        \foreach \i in {1,2,3,4,5,6}{
            \vertex{\i}
        }
        \draw(1) node[anchor=north] {$x$};
        \draw(2) node[anchor=north] {$y$};
        \draw(3) node[anchor=south] {$z$};
        \draw(4) node[anchor=south] {$u$};
        \draw(5) node[anchor=south] {$v$};
        \draw(6)+(0.15,0) node[anchor=south east] {$w$};
        \draw[blue](-0.75,0.6) node {1};
        \draw[blue] (2,0.6) node {4};
        \draw[blue](-0.2,0.6) node {2};
        \draw[blue](0.4,0.6) node {3};
    \end{tikzpicture}&
    \begin{tikzpicture}
        \foreach \i in {0,0.5,1,1.5}{
            \foreach \j in {0,0.5,1,1.5}{
            \vertex{\i,\j}
            }
    }
    \coordinate (1) at (0.5,0.5);
    \coordinate (3) at (0.5,1);
    \coordinate (2) at (1,0.5);
    \coordinate (4) at (1,1);
        \foreach \j in {1,2,3,4}{
            \vertexblue{\j}
        }
    \draw[blue] (1) node[anchor=east] {1};\draw[blue] (2) node[anchor=north west] {2};
    \draw[blue] (3) node[anchor=east] {3};\draw[blue] (4) node[anchor=south west] {4};
    \draw(0,0)--(1.5,0);
    \draw(0,0)--(0,1.5);
    \draw(1.5,1.5)--(0,1.5);
    \draw(1.5,1.5)--(1.5,0);
    \draw[blue](0,1.5)--(1.5,0.5);
    \draw[blue] (3)--(2);
    \draw[blue] (0,1)--(2);
    \draw[blue] (2)--(1);
    \draw[] (4)--(1.5,0.5);\draw(4)--(0,1.5);\draw(0,1)--(3);\draw(0,1)--(1);
    \draw(3)--(0,1.5);\draw(3)--(1.5,0.5);
    \draw[] (2)--(1.5,0.5);
    \end{tikzpicture}
    &
    \begin{tikzpicture}[scale=0.75]
    \draw[blue] (0,0)--(0.5,0.5);
    \draw[] (0.5,0.5)--(1,1.5);
    \draw[] (1,1.5)--(0,0.5);
    \draw[] (0.5,0)--(0.5,0.5);
    \draw[blue] (1,1.5)--(2,3);
    \draw[] (2,3)--(2.5,4);
    \draw[] (2,3)--(2.5,3.5);
    \draw[blue](0,0)--(-0.5,-1);
    \draw[](0,0)--(0,0.5);
    \draw[](-1,0.5)--(0,0.5);
    \draw[blue](-0.5,-1.5)--(-0.5,-1);
    \draw[](-1,-1.5)--(-0.5,-1);
    \vertexblue{1,1.5}\vertexblue{2,3}\vertexblue{0.5,0.5}
    \vertexblue{0,0}
    \draw[] (1,1.5) node[anchor=east] {$v$};
    \draw[] (2,3) node[anchor=west] {$w$};
    \draw[] (0.5,0.5) node[anchor=west] {$u$};
    \draw[] (0,0) node[anchor=east] {$y$};
    \end{tikzpicture}
    \end{tikzcd}\caption{Combinatorial type, partial triangulations of $\mathbb F_0$ and partial embedded tropical curve.}\label{fg:(111)F_0}
    \end{figure}

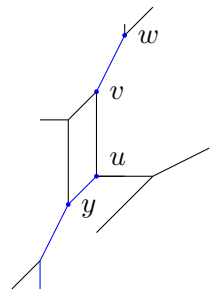
\begin{figure}[h]
\begin{tikzcd}
    \begin{tikzpicture}
        \foreach \j in {0,0.5,1,1.5,2,2.5,3}{
            \vertex{0,\j}
        }
        \foreach \j in {0,0.5,1,1.5,2}{
            \vertex{0.5,\j}
        }
        \foreach \j in {0,0.5,1}{
            \vertex{1,\j}
        }
        \vertex{1.5,0}
        \draw (0,0)--(1.5,0);
        \draw (0,3)--(1.5,0);
        \draw (0,3)--(0,0);
        \coordinate (2) at (1,0.5);
        \coordinate (1) at (0.5,0.5);
        \coordinate (3) at (0.5,1);
        \coordinate (4) at (0.5,1.5);
        \foreach \j in {1,2,3,4}{
            \vertexblue{\j}
        }
        \draw[blue] (1) node[anchor=north] {1};
        \draw[blue] (2) node[anchor=north] {2};
        \draw[blue] (3) node[anchor=east] {3};
        \draw[blue] (4) node[anchor=east] {4};
        \draw[blue] (3)--(2);
        \draw[blue] (1)--(2);
        \draw[blue] (0,1.5)--(1,1);
        \draw[blue] (0,1)--(1,0.5);
    \end{tikzpicture}&
    \begin{tikzpicture}
        \foreach \j in {0,0.5,1,1.5,2,2.5,3}{
            \vertex{0,\j}
        }
        \foreach \j in {0,0.5,1,1.5,2}{
            \vertex{0.5,\j}
        }
        \foreach \j in {0,0.5,1}{
            \vertex{1,\j}
        }
        \vertex{1.5,0}
        \draw (0,0)--(1.5,0);
        \draw (0,3)--(1.5,0);
        \draw (0,3)--(0,0);
        \coordinate (2) at (1,0.5);
        \coordinate (1) at (0.5,0.5);
        \coordinate (3) at (0.5,1);
        \coordinate (4) at (0.5,1.5);
        \foreach \j in {1,2,3,4}{
            \vertexblue{\j}
        }
        \draw[blue] (1) node[anchor=north] {1};
        \draw[blue] (2) node[anchor=north] {2};
        \draw[blue] (3) node[anchor=east] {3};
        \draw[blue] (4) node[anchor=south east] {4};
        \draw[blue] (3)--(2);
        \draw[blue] (1)--(2);
        \draw[blue] (0,1.5)--(1,1);
        \draw[blue] (0,1)--(1,0.5);\draw(0,1)--(1,1);\draw(0,1)--(1);
        \draw (4)--(1,1);\draw(4)--(0,1.5);\draw(3)--(0,1.5);
    \draw(1,0.5)--(1.5,0); 
    \draw[] (1,1)--(1,0.5);

    \end{tikzpicture}&
    \begin{tikzpicture}[scale=0.75]
    \draw[](1.5,2.5+1)--(2,3+1);
    \draw[blue](1,1.5+1)--(1.5,2.5+1); 
    \draw[](1.5,2.7+1)--(1.5,2.5+1);
    \draw[](0.5,1+1)--(1,1.5+1);
    \draw[](1,0+1)--(1,1.5+1);
    \draw[](1,0+1)--(1.5,0+1);
    \draw[](3,0.5+1)--(2,0+1);
    \draw[](1,-1+1)--(2,0+1);
    \draw[blue](0.5,-0.5+1)--(1,0+1);
    \draw[blue](0.5,-0.5+1)--(0,-1.5+1);
    \draw[](0.5,-0.5+1)--(0.5,1+1);
    \draw[](0,1+1)--(0.5,1+1);
    \draw[blue](0,-2+1)--(0,-1.5+1);
    \draw[](-0.5,-2+1)--(0,-1.5+1);
    \draw[](1,0+1)--(2,0+1);
    \vertexblue{1,0+1}
    \vertexblue{1,1.5+1}
    \vertexblue{0.5,-0.5+1}
    \vertexblue{1.5,2.5+1}
    \draw[] (0.5,-0.5+1) node[anchor= west] {$y$};
    \draw[] (1.5,2.5+1) node[anchor=west] {$w$};
    \draw[] (1,0+1) node[anchor= south west] {$u$};
    \draw[] (1,1.5+1) node[anchor= west] {$v$};
    \end{tikzpicture}
    \end{tikzcd}\caption{Partial triangulation of $\mathbb F_2$ and partial embedded tropical curve.}\label{fg:(111)F_2}
\end{figure}
 
Comparing again the $y$-coordinates of the points proves the following.
\begin{lemma}
    Let $\Gamma$ be a tropical curve of genus $4$ of combinatorial type (111).
    \begin{itemize}
        \item If $\Gamma$ is realizable in $\mathbb F_0,$ then its edge lengths satisfy:
        \begin{equation}
        \begin{cases}
           uy+2uv=vy.
        \end{cases}
        \end{equation}
        \item If $\Gamma$ is realizable in $\mathbb F_2,$ then its edge lengths satisfy:
        \begin{equation}
        \begin{cases}
            uy+uv=vy.
            
        \end{cases}
        \end{equation}
    \end{itemize}
\end{lemma}

\subsubsection*{Type (121)}

Let us now consider the combinatorial type (121). The triangulation describing the embedding in $\mathbb F_0$ is represented in Figure \ref{fg:(121)F_0}.
\begin{figure}[h]
\begin{tikzcd}
      \begin{tikzpicture}
        \coordinate (1) at (-0.75,0);
        \coordinate (2) at (0,0);
        \coordinate (3) at (-0.75,1.5);
        \coordinate (4) at (0,1.5);
        \coordinate (5) at (0.75,0.75);
        \coordinate (6) at (1.5,0.75);
        \draw (4)--(5);
        \draw (2)--(5);
        \draw (6)--(5);
        \draw(1)--(3);
        \draw (2,0.75) circle (0.5);
        \draw(3)[] to [out=30, in=150] (4);
        \draw(3)[] to [out=330, in=210] (4);
        \draw(1)[] to [out=30, in=150] (2);
        \draw(1)[] to [out=330, in=210] (2);
        \foreach \i in {1,2,3,4,5,6}{
            \vertex{\i}
        }
        \draw(1) node[anchor=north] {$x$};
        \draw(2) node[anchor=north] {$y$};
        \draw(3) node[anchor=south] {$z$};
        \draw(4) node[anchor=south] {$u$};
        \draw(5) node[anchor=south] {$v$};
        \draw(6)+(0.15,0) node[anchor=south east] {$w$};
        \draw[blue](-0.4,1.4) node {1};
        \draw[blue] (2,0.6) node {4};
        \draw[blue](-0.4,-0.1) node {2};
        \draw[blue](0,0.6) node {3};
    \end{tikzpicture}&
    \begin{tikzpicture}
        \foreach \i in {0,0.5,1,1.5}{
            \foreach \j in {0,0.5,1,1.5}{
            \vertex{\i,\j}
            }
    }
    \coordinate (1) at (0.5,0.5);
    \coordinate (3) at (0.5,1);
    \coordinate (2) at (1,0.5);
    \coordinate (4) at (1,1);
        \foreach \j in {1,2,3,4}{
            \vertexblue{\j}
        }
    \draw[blue] (1) node[anchor=east] {1};\draw[blue] (2) node[anchor=north west] {2};
    \draw[blue] (3) node[anchor=east] {3};\draw[blue] (4) node[anchor=south west] {4};
    \draw(0,0)--(1.5,0);
    \draw(0,0)--(0,1.5);
    \draw(1.5,1.5)--(0,1.5);
    \draw(1.5,1.5)--(1.5,0);
    \draw[blue](0,1.5)--(1.5,0.5);
    \draw[blue] (3)--(2);
    \draw[blue] (3)--(1);
    \draw[] (4)--(1.5,0.5);\draw(4)--(0,1.5);\draw[](3)--(0,1.5);\draw(3)--(1.5,0.5);\draw(2)--(1.5,0);\draw(2)--(1,0);\draw(1)--(1,0);
    \draw[blue](3)--(1,0);
    \draw[] (2)--(1.5,0.5);
    \end{tikzpicture}
    &
    \begin{tikzpicture}[scale=0.75]
    \draw[blue] (0,0)--(0.5,0.5);
    \draw[] (0.5,0.5)--(1,1.5);
    \draw[] (1,1.5)--(0.5,0.5);
    \draw[] (0.5,0.5)--(0.5,0);
    \draw[] (1,0)--(0.5,0);
    \draw[] (0,-0.5)--(0.5,0);
    \draw[] (0,-0.5)--(0,-1);
    \draw[] (-0.5,-0.5)--(0,-0.5);
    \draw[blue] (-0.5,-0.5)--(0.5,0.5);
    \draw[blue] (-1,-0.75)--(-0.5,-0.5);
    \draw[blue] (-1,-0.75)--(-1.2,-0.75);
    \draw[] (-1,-0.75)--(-1.5,-1.25);
    \draw[] (1,1.5)--(-1,-0.5);
    \draw[] (0.5,0)--(0.5,0.5);
    \draw[blue] (1,1.5)--(2,3);
    \draw[] (2,3)--(2.5,4);
    \draw[] (2,3)--(2.5,3.5);
    \vertexblue{1,1.5}\vertexblue{2,3}\vertexblue{0.5,0.5}
    \vertexblue{-0.5,-0.5}\vertexblue{-1,-0.75}\vertexblue{-1.2,-0.75}
    \draw[] (1,1.5) node[anchor=east] {$v$};
    \draw[] (2,3) node[anchor=west] {$w$};
    \draw[] (-1,-0.75) node[anchor=north] {$x$};
    \draw[] (0.5,0.5) node[anchor=west] {$u$};
    \draw[] (-0.5,-0.5) node[anchor=north] {$z$};
    \draw[] (-1.2,-0.75) node[anchor=east] {$y$};
    \end{tikzpicture}
    \end{tikzcd}\caption{Combinatorial type, partial triangulations of $\mathbb F_0$ and partial embedded tropical curve.}\label{fg:(121)F_0}
    \end{figure}

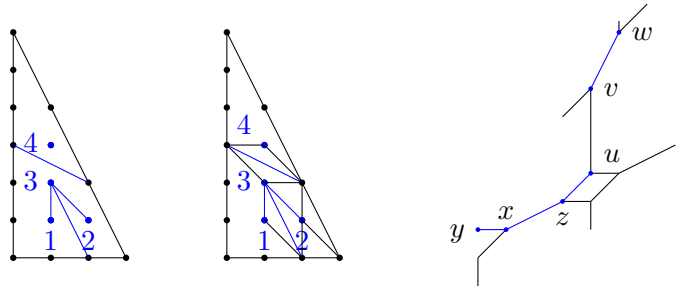
\begin{figure}[h]
\begin{tikzcd}
    \begin{tikzpicture}
        \foreach \j in {0,0.5,1,1.5,2,2.5,3}{
            \vertex{0,\j}
        }
        \foreach \j in {0,0.5,1,1.5,2}{
            \vertex{0.5,\j}
        }
        \foreach \j in {0,0.5,1}{
            \vertex{1,\j}
        }
        \vertex{1.5,0}
        \draw (0,0)--(1.5,0);
        \draw (0,3)--(1.5,0);
        \draw (0,3)--(0,0);
        \coordinate (2) at (1,0.5);
        \coordinate (1) at (0.5,0.5);
        \coordinate (3) at (0.5,1);
        \coordinate (4) at (0.5,1.5);
        \foreach \j in {1,2,3,4}{
            \vertexblue{\j}
        }
        \draw[blue] (1) node[anchor=north] {1};
        \draw[blue] (2) node[anchor=north] {2};
        \draw[blue] (3) node[anchor=east] {3};
        \draw[blue] (4) node[anchor=east] {4};
        \draw[blue] (3)--(2);
        \draw[blue] (1)--(3);
        \draw[blue] (3)--(1,0);
        \draw[blue] (0,1.5)--(1,1);
    \end{tikzpicture}&
    \begin{tikzpicture}
        \foreach \j in {0,0.5,1,1.5,2,2.5,3}{
            \vertex{0,\j}
        }
        \foreach \j in {0,0.5,1,1.5,2}{
            \vertex{0.5,\j}
        }
        \foreach \j in {0,0.5,1}{
            \vertex{1,\j}
        }
        \vertex{1.5,0}
        \draw (0,0)--(1.5,0);
        \draw (0,3)--(1.5,0);
        \draw (0,3)--(0,0);
        \coordinate (2) at (1,0.5);
        \coordinate (1) at (0.5,0.5);
        \coordinate (3) at (0.5,1);
        \coordinate (4) at (0.5,1.5);
        \foreach \j in {1,2,3,4}{
            \vertexblue{\j}
        }
        \draw[blue] (1) node[anchor=north] {1};
        \draw[blue] (2) node[anchor=north] {2};
        \draw[blue] (3) node[anchor=east] {3};
        \draw[blue] (4) node[anchor=south east] {4};
        \draw[blue] (3)--(2);
        \draw[blue] (1)--(3);
        \draw[blue] (3)--(1,0);
        \draw[blue] (0,1.5)--(1,1);\draw(3)--(1,1);\draw(1,0)--(1);
        \draw (4)--(1,1);\draw(4)--(0,1.5);\draw(3)--(0,1.5);\draw(2)--(1,0);
    \draw(1,0.5)--(1.5,0); 
    \draw[] (1,1)--(1,0.5);
    \end{tikzpicture}&
    \begin{tikzpicture}[scale=0.75]
    \draw[](1.5,2.5+1.5)--(2,3+1.5);
    \draw[blue](1,1.5+1.5)--(1.5,2.5+1.5); 
    \draw[](1.5,2.7+1.5)--(1.5,2.5+1.5);
    \draw[](0.5,1+1.5)--(1,1.5+1.5);
    \draw[](1,0+1.5)--(1,1.5+1.5);
    \draw[](1,0+1.5)--(1.5,0+1.5);
    \draw[](2.5,0.5+1.5)--(1.5,0+1.5);
    \draw[blue](0.5,-0.5+1.5)--(1,0+1.5);
    \draw[blue](-0.5,-1+1.5)--(0.5,-0.5+1.5);
    \draw[blue](-0.5,-1+1.5)--(-1,-1+1.5);
    \draw[](-0.5,-1+1.5)--(-1,-1.5+1.5);
    \draw[](-1,-2+1.5)--(-1,-1.5+1.5);
    \draw[](0.5,-0.5+1.5)--(1,-0.5+1.5);
    \draw[](1.5,0+1.5)--(1,-0.5+1.5);
    \draw[](1,-0.5+1.5)--(1,-1+1.5);
   \vertexblue{-0.5,-1+1.5}
    \vertexblue{1,0+1.5}\vertexblue{1,1.5+1.5}
    \vertexblue{-1,-1+1.5}
    \vertexblue{0.5,-0.5+1.5}
    \vertexblue{1.5,2.5+1.5}
    \draw[] (0.5,-0.5+1.5) node[anchor= north] {$z$};
    \draw[] (1.5,2.5+1.5) node[anchor=west] {$w$};
    \draw[] (-1,-1+1.5) node[anchor= east] {$y$};
    \draw[] (1,0+1.5) node[anchor= south west] {$u$};
    \draw[] (1,1.5+1.5) node[anchor= west] {$v$};
    \draw[] (-0.5,-1+1.5) node[anchor= south] {$x$};
    \end{tikzpicture}
    \end{tikzcd}\caption{Partial triangulation of $\mathbb F_2$ and partial embedded tropical curve.}\label{fg:(121)F_2}
\end{figure}

Comparing again the $y$-coordinates of the points proves the following.
\begin{lemma}
    Let $\Gamma$ be a tropical curve of genus $4$ of combinatorial type (121).
    \begin{itemize}
        \item If $\Gamma$ is realizable in $\mathbb F_0,$ then its edge lengths satisfy:
        \begin{equation}
        \begin{cases}
           xz+zu+2uv=vy.\\ 
        \end{cases}
        \end{equation}
        \item If $\Gamma$ is realizable in $\mathbb F_2$, then its edge lengths satisfy:
        \begin{equation}
        \begin{cases}
            xz+zu+uv=vy.
        \end{cases}
        \end{equation}
    \end{itemize}
\end{lemma}

\subsubsection*{Type (122)}

Let us now consider the combinatorial type (122). The triangulation describing the embedding in $\mathbb F_0$ and possible embeddings in $\FF_2$ are represented in Figures \ref{fg:(122)F_0},\ref{fg:(122)F_2a} and \ref{fg:(122)F_2b}.
\begin{figure}[h]
\begin{tikzcd}
      \begin{tikzpicture}
        \coordinate (1) at (-0.75,0);
        \coordinate (2) at (0,0.75);
        \coordinate (3) at (-0.75,1.5);
        \coordinate (4) at (0.75,0.75);
        \coordinate (5) at (1.5,0.75);
        \coordinate (6) at (2,0.75);
        \draw (3)--(2);
        \draw (1)--(2);
        \draw (2)--(4);
        \draw (5)--(6);
        \draw(1)[] to [out=120, in=240] (3);
        \draw(1)[] to [out=60, in=300] (3);
        \draw (2.5,0.75) circle (0.5);
        \draw(4)[] to [out=30, in=150] (5);
        \draw(4)[] to [out=330, in=210] (5);
        \foreach \i in {1,2,3,4,5,6}{
            \vertex{\i}
        }
        \draw(1) node[anchor=north] {$x$};
        \draw(2) node[anchor=north] {$y$};
        \draw(3) node[anchor=south] {$z$};
        \draw(4) node[anchor=south] {$u$};
        \draw(5) node[anchor=south] {$v$};
        \draw(6)+(0.15,0) node[anchor=south east] {$w$};
        \draw[blue](-0.75,0.6) node {1};
        \draw[blue] (2.5,0.6) node {4};
        \draw[blue](-0.4,0.6) node {2};
        \draw[blue](1.2,0.6) node {3};
    \end{tikzpicture}&
    \begin{tikzpicture}
        \foreach \i in {0,0.5,1,1.5}{
            \foreach \j in {0,0.5,1,1.5}{
            \vertex{\i,\j}
            }
    }
    \coordinate (1) at (0.5,0.5);
    \coordinate (3) at (0.5,1);
    \coordinate (2) at (1,0.5);
    \coordinate (4) at (1,1);
        \foreach \j in {1,2,3,4}{
            \vertexblue{\j}
        }
    \draw[blue] (1) node[anchor=east] {1};\draw[blue] (2) node[anchor=north west] {2};
    \draw[blue] (3) node[anchor=east] {3};\draw[blue] (4) node[anchor=south west] {4};
    \draw(0,0)--(1.5,0);
    \draw(0,0)--(0,1.5);
    \draw(1.5,1.5)--(0,1.5);
    \draw(1.5,1.5)--(1.5,0);
    \draw[blue](0,1.5)--(1.5,0.5);
    \draw[blue](0,1)--(1.5,0.5);
    \draw[blue] (1)--(2);
    \draw[] (4)--(1.5,0.5);\draw(4)--(0,1.5);
    \draw[](3)--(0,1.5);\draw(3)--(1.5,0.5);
    \draw (3)--(0,1);
    \draw(2)--(0,1);\draw(1)--(0,1);
    \draw[] (2)--(1.5,0.5);
    \end{tikzpicture}
    &
    \begin{tikzpicture}[scale=0.75]
    \draw[] (0.5,0.5)--(1,1.5);
    \draw[blue] (0,-1)--(0.5,0.5);
    \draw[] (0,-1)--(0,-2);
    \draw[] (0,-1)--(-0.5,-2);
    \draw[] (-1,-2.5)--(-0.5,-2);
    \draw[] (-1,-2.5)--(-1.5,-2.5);
    \draw[blue] (-0.5,-2.5)--(-0.5,-2);
    \draw[] (1,1.5)--(0.5,1);\draw[] (0,1)--(0.5,1);
    \draw[] (0.5,0.5)--(0.5,1);
   \draw[blue] (1,1.5)--(2,3);
    \draw[] (2,3)--(2.5,4);
    \draw[] (2,3)--(2.5,3.5);
    \vertexblue{1,1.5}\vertexblue{2,3}\vertexblue{0.5,0.5}
    \vertexblue{0,-1}\vertexblue{-0.5,-2}
    \draw[] (1,1.5) node[anchor=east] {$v$};
    \draw[] (2,3) node[anchor=west] {$w$};
    \draw[] (-0.5,-2) node[anchor=east] {$x$};
    \draw[] (0.5,0.5) node[anchor=west] {$u$};
    \draw[] (0,-1) node[anchor=east] {$y$};
    \end{tikzpicture}
    \end{tikzcd}\caption{Combinatorial type, partial triangulations of $\mathbb F_0$ and partial embedded tropical curve.}\label{fg:(122)F_0}
    \end{figure}
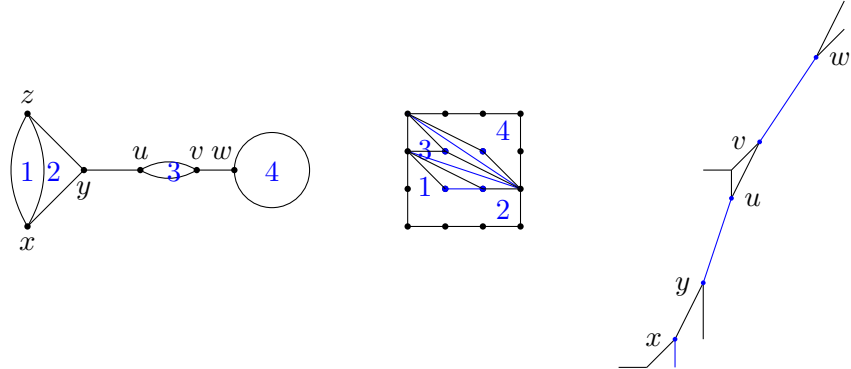

\begin{figure}[h]
\begin{tikzcd}
    \begin{tikzpicture}
        \foreach \j in {0,0.5,1,1.5,2,2.5,3}{
            \vertex{0,\j}
        }
        \foreach \j in {0,0.5,1,1.5,2}{
            \vertex{0.5,\j}
        }
        \foreach \j in {0,0.5,1}{
            \vertex{1,\j}
        }
        \vertex{1.5,0}
        \draw (0,0)--(1.5,0);
        \draw (0,3)--(1.5,0);
        \draw (0,3)--(0,0);
        \coordinate (2) at (1,0.5);
        \coordinate (1) at (0.5,0.5);
        \coordinate (3) at (0.5,1);
        \coordinate (4) at (0.5,1.5);
        \foreach \j in {1,2,3,4}{
            \vertexblue{\j}
        }
        \draw[blue] (1) node[anchor=north] {1};
        \draw[blue] (2) node[anchor=north] {2};
        \draw[blue] (3) node[anchor=east] {3};
        \draw[blue] (4) node[anchor=east] {4};
        \draw[blue](1)--(2);
        \draw[blue] (1,1)--(0,0.5);
        \draw[blue] (0,1.5)--(1,1);
    \end{tikzpicture}&
    \begin{tikzpicture}
        \foreach \j in {0,0.5,1,1.5,2,2.5,3}{
            \vertex{0,\j}
        }
        \foreach \j in {0,0.5,1,1.5,2}{
            \vertex{0.5,\j}
        }
        \foreach \j in {0,0.5,1}{
            \vertex{1,\j}
        }
        \vertex{1.5,0}
        \draw (0,0)--(1.5,0);
        \draw (0,3)--(1.5,0);
        \draw (0,3)--(0,0);
        \coordinate (2) at (1,0.5);
        \coordinate (1) at (0.5,0.5);
        \coordinate (3) at (0.5,1);
        \coordinate (4) at (0.5,1.5);
        \foreach \j in {1,2,3,4}{
            \vertexblue{\j}
        }
        \draw[blue] (1) node[anchor=north] {1};
        \draw[blue] (2) node[anchor=north] {2};
        \draw[blue] (3) node[anchor=east] {3};
        \draw[blue] (4) node[anchor=south east] {4};
        \draw[blue] (1)--(2);
        \draw[blue] (1,1)--(0,0.5);
        \draw[blue] (0,1.5)--(1,1);
        \draw(3)--(1,1);\draw(0,0.5)--(1);
        \draw(1,1)--(1);
        \draw (4)--(1,1);\draw(4)--(0,1.5);\draw(3)--(0,1.5);\draw(3)--(0,1);\draw(3)--(0,0.5);
    \draw(1,0.5)--(1.5,0); 
    \draw[] (1,1)--(1,0.5); 
    \end{tikzpicture}&
    \begin{tikzpicture}[scale=0.75]
    \draw[](1.5,2.5+1.5)--(2,3+1.5);
    \draw[blue](1,1.5+1.5)--(1.5,2.5+1.5); \draw[](1.5,2.7+1.5)--(1.5,2.5+1.5);
    \draw[](0.5,1+1.5)--(1,1.5+1.5);
    \draw[](0.5,1+1.5)--(0.5,0.5+1.5);\draw[](0.5,1+1.5)--(0,1+1.5);
    \draw[](1,0+1.5)--(0.5,0.5+1.5);\draw[](0.5,0.5+1.5)--(0,0.5+1.5);
    \draw[](1,0+1.5)--(1,1.5+1.5);

    \draw[blue](1,+1.5)--(1.5,-1+1.5);
    \vertexblue{1.5,-1+1.5}
    \draw[](1.5,-2+1.5)--(1.5,-1+1.5);
    \draw[](2,-1.5+1.5)--(1.5,-1+1.5);
    \draw[](2,-1.5+1.5)--(2.5,-1.5+1.5);
    \draw[](3.5,-1+1.5)--(2.5,-1.5+1.5);
    \draw[blue](2,-1.5+1.5)--(2,-2+1.5);
    \vertexblue{1,0+1.5}\vertexblue{1,1.5+1.5}
    \vertexblue{2,-1.5+1.5}
    \vertexblue{1.5,2.5+1.5}
    \draw[] (2,-1.5+1.5) node[anchor= south] {$x$};
    \draw[] (1.5,2.5+1.5) node[anchor=west] {$w$};
    \draw[] (1.5,-1+1.5) node[anchor= east] {$y$};
    \draw[] (1,0+1.5) node[anchor= south west] {$u$};
    \draw[] (1,1.5+1.5) node[anchor= west] {$v$};
    \end{tikzpicture}
    \end{tikzcd}\caption{Partial triangulation of $\mathbb F_2$ and partial embedded tropical curve.}\label{fg:(122)F_2a}
\end{figure}

\begin{figure}[h]
\begin{tikzcd}
    \begin{tikzpicture}
        \foreach \j in {0,0.5,1,1.5,2,2.5,3}{
            \vertex{0,\j}
        }
        \foreach \j in {0,0.5,1,1.5,2}{
            \vertex{0.5,\j}
        }
        \foreach \j in {0,0.5,1}{
            \vertex{1,\j}
        }
        \vertex{1.5,0}
        \draw (0,0)--(1.5,0);
        \draw (0,3)--(1.5,0);
        \draw (0,3)--(0,0);
        \coordinate (4) at (1,0.5);
        \coordinate (3) at (0.5,0.5);
        \coordinate (2) at (0.5,1);
        \coordinate (1) at (0.5,1.5);
        \foreach \j in {1,2,3,4}{
            \vertexblue{\j}
        }
        \draw[blue] (1) node[anchor=east] {1};
        \draw[blue] (2) node[anchor=east] {2};
        \draw[blue] (3) node[anchor=north] {3};
        \draw[blue] (4) node[anchor=north] {4};
        \draw[blue](1)--(2);
        \draw[blue] (1,1)--(0.5,0);
        \draw[blue] (0,0.5)--(1,1);
    \end{tikzpicture}
&
    \begin{tikzpicture}
        \foreach \j in {0,0.5,1,1.5,2,2.5,3}{
            \vertex{0,\j}
        }
        \foreach \j in {0,0.5,1,1.5,2}{
            \vertex{0.5,\j}
        }
        \foreach \j in {0,0.5,1}{
            \vertex{1,\j}
        }
        \vertex{1.5,0}
        \draw (0,0)--(1.5,0);
        \draw (0,3)--(1.5,0);
        \draw (0,3)--(0,0);
        \coordinate (4) at (1,0.5);
        \coordinate (3) at (0.5,0.5);
        \coordinate (2) at (0.5,1);
        \coordinate (1) at (0.5,1.5);
        \foreach \j in {1,2,3,4}{
            \vertexblue{\j}
        }
        \draw[blue] (1) node[anchor=east] {1};
        \draw[blue] (2) node[anchor=east] {2};
        \draw[blue] (3) node[anchor=north] {3};
        \draw[blue] (4) node[anchor=north] {4};
        \draw[blue](1)--(2);
        \draw[blue] (1,1)--(0.5,0);
        \draw[blue] (0,0.5)--(1,1);
        \draw(4)--(0.5,0); \draw(4)--(1.5,0); \draw(4)--(1,0);
        \draw(3)--(1,1);
        \draw(3)--(0,0.5);\draw(3)--(0.5,0);
        \draw(2)--(0,0.5);\draw(2)--(1,1);
        \draw(1)--(1,1);
        \draw[] (1,1)--(1,0.5);
    \end{tikzpicture}&
    \begin{tikzpicture}[scale=0.75]
    \draw(0.5,-0.5+1.5)--(1,-0.5+1.5);
    \draw(0,0+1.5)--(0,0.5+1.5);
    \draw[](0,0.5+1.5)--(1,-0.5+1.5);
    \draw[blue](0,0.5+1.5)--(-0.5,1.5+1.5); 
    \draw[](-0.5,1.5+1.5)--(-0.5,2.5+1.5);
    \draw[](-0.5,1.5+1.5)--(-1,2+1.5);
    \draw[blue](-0.5,2.5+1.5)--(-1,2.5+1.5);
    \draw[](-0.5,2.5+1.5)--(0,3+1.5);
    \draw[blue](1,-0.5+1.5)--(2,-1+1.5);
    \draw[](2,-1+1.5)--(3.5,-1+1.5);
    \draw[](2,-1+1.5)--(2.5,-1.5+1.5);
    \draw[](2.5,-1.5+1.5)--(3,-1.5+1.5);
    \draw[](3.5,-1+1.5)--(3,-1.5+1.5);
    \draw[](2.5,-1.5+1.5)--(2.5,-2+1.5);
    \draw[](3,-1.5+1.5)--(3,-2+1.5);
    \draw[](3.5,-1+1.5)--(4.3,-0.5+1.5);
    \vertexblue{2,-1+1.5}\vertexblue{1,-0.5+1.5}
    \vertexblue{0,0.5+1.5}
    \vertexblue{-0.5,1.5+1.5}\vertexblue{-0.5,2.5+1.5}
    \vertexblue{-1,2.5+1.5}
    \draw[] (-1,2.5+1.5) node[anchor= east] {$z$};
    \draw[] (2,-1+1.5) node[anchor=south] {$w$};
    \draw[] (-0.5,1.5+1.5) node[anchor= north west] {$y$};
    \draw[] (0,0.5+1.5) node[anchor= south west] {$u$};
    \draw[] (1,-0.5+1.5) node[anchor= south west] {$v$};
    \draw[] (-0.5,2.5+1.5) node[anchor= west] {$x$};
    \end{tikzpicture}
    \end{tikzcd}\caption{Partial triangulation of $\mathbb F_2$ and partial embedded tropical curve.}\label{fg:(122)F_2b}
\end{figure}

In Figure \ref{fg:(122)F_0}, comparing the $y$-coordinates of the points $u,v$ and the other point of valence $3$ along one of the two path connecting the two points, that the segment defining a path between the two points has to have length twice the sum of the other segments in the other path.

Instead, in Figure \ref{fg:(122)F_2a}, we see that $l_1= l_2$, where $l_1,l_2$ denote the two parallel edges with $u,v$ as endpoints. Notice that no other edge length equation arises from the partial triangulation. In fact, the bottom part of the partial triangulation can be completes similarly as in Figure \ref{fg:(021)codim}, and gives no relations between the edge lengths of the edges with endpoints $x,y,z.$

In Figure \ref{fg:(122)F_2b}, we see that $l_1\neq l_2.$ Moreover any completion of the triangulation will join $y$ with $z$ with edges of segments $(-1,1)$ and $(0,1)$ or with a single edge defined by $(-1,1)$. The $y$-coordinates of such vectors are the same, hence $yz=xy.$  
Also here notice that no other equation arises from the triangulation.
 
This proves the following Lemma.
\begin{lemma}
    Let $\Gamma$ be a tropical curve of genus $4$ of combinatorial type (122). Let $l_1,l_2$ be the lengths of the two parallel edges with endpoints $u,v$. 
    \begin{itemize}
        \item If $\Gamma$ is realizable in $\mathbb F_0,$ then $l_1=2\cdot l_2$, up to symmetry. 
        \item If $\Gamma$ is realizable in $\mathbb F_2,$ then its edge lengths satisfy, up to symmetry:
        \begin{equation}
        \begin{cases}
            l_1=l_2;
        \end{cases} \text{or} \quad \begin{cases}
            xy=yz.
        \end{cases}
        \end{equation}\end{itemize}
\end{lemma}

\subsubsection*{Type (202)}
Let us now consider the combinatorial type (202). The triangulation describing the embedding in $\mathbb F_0$ is represented in Figure \ref{fg:(202)F_0}.

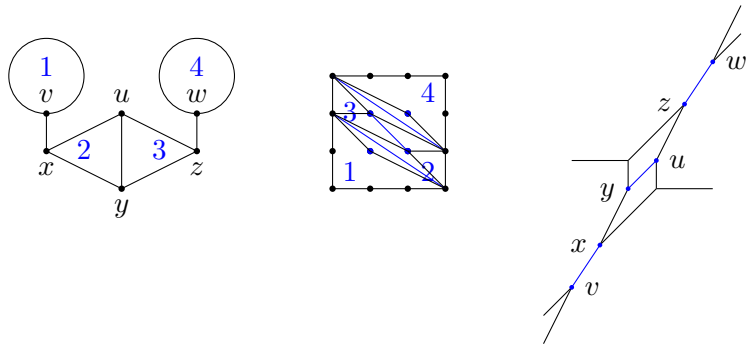
\begin{figure}[h]
\begin{tikzcd}
      \begin{tikzpicture}
        \coordinate (1) at (-0.5,0.5);
        \coordinate (2) at (0.5,0);
        \coordinate (3) at (1.5,0.5);
        \coordinate (4) at (0.5,1);
        \coordinate (5) at (-0.5,1);
        \coordinate (6) at (1.5,1);
        \draw (4)--(2);
        \draw (1)--(2);
        \draw (1)--(4);
        \draw (3)--(4);
        \draw (3)--(2);
        \draw (1)--(5);
        \draw (3)--(6);
        \draw (-0.5,1.5) circle (0.5);
        \draw (1.5,1.5) circle (0.5);
        \foreach \i in {1,2,3,4,5,6}{
            \vertex{\i}
        }
        \draw(1) node[anchor=north] {$x$};
        \draw(2) node[anchor=north] {$y$};
        \draw(3) node[anchor=north] {$z$};
        \draw(4) node[anchor=south] {$u$};
        \draw(5) node[anchor=south] {$v$};
        \draw(6) node[anchor=south] {$w$};
        \draw[blue](-0.5,1.5) node {1};
        \draw[blue] (1.5,1.5) node {4};
        \draw[blue](0,0.4) node {2};
        \draw[blue](1,0.4) node {3};
    \end{tikzpicture}&
    \begin{tikzpicture}
        \foreach \i in {0,0.5,1,1.5}{
            \foreach \j in {0,0.5,1,1.5}{
            \vertex{\i,\j}
            }
    }
    \coordinate (1) at (0.5,0.5);
    \coordinate (3) at (0.5,1);
    \coordinate (2) at (1,0.5);
    \coordinate (4) at (1,1);
        \foreach \j in {1,2,3,4}{
            \vertexblue{\j}
        }
    \draw[blue] (1) node[anchor=north east] {1};\draw[blue] (2) node[anchor=north west] {2};
    \draw[blue] (3) node[anchor=east] {3};\draw[blue] (4) node[anchor=south west] {4};
    \draw(0,0)--(1.5,0);
    \draw(0,0)--(0,1.5);
    \draw(1.5,1.5)--(0,1.5);
    \draw(1.5,1.5)--(1.5,0);
    \draw[blue](1.5,0)--(0,1);
    \draw[blue](1.5,0.5)--(0,1.5);
    \draw[blue] (3)--(2);
    \draw[] (4)--(1.5,0.5);\draw(4)--(0,1.5);
    \draw[](3)--(0,1.5);\draw(3)--(1.5,0.5);
    \draw (3)--(0,1);
    \draw(2)--(0,1);\draw(1)--(0,1);
    \draw[] (2)--(1.5,0.5);\draw (1.5,0)--(2);\draw (1.5,0)--(1);
    \end{tikzpicture}
    &
    \begin{tikzpicture}[scale=0.75]
    \draw[] (0.5,0.5)--(1,1.5);
    \draw[blue] (0.5,0.5)--(0,0);
    \draw[] (0,0)--(0,0.5);
    \draw[] (1,1.5)--(0,0.5);
    \draw[] (-1,0.5)--(0,0.5);
    \draw[] (0.5,0)--(0.5,0.5);
    \draw[] (0.5,0)--(-0.5,-1);\draw[] (0.5,0)--(1.5,0);
    \draw[] (0,0)--(-0.5,-1);
    \draw[blue] (-1,-1.75)--(-0.5,-1);\vertexblue{-0.5,-1}
    \draw[] (-1,-1.75)--(-1.5,-2.25);
    \draw[] (-1,-1.75)--(-1.5,-2.75);
   \draw[blue] (1,1.5)--(1.5,2.25);
    \draw[] (1.5,2.25)--(2,3.25);
    \draw[] (1.5,2.25)--(2,2.75);
    \vertexblue{1,1.5}\vertexblue{1.5,2.25}\vertexblue{0.5,0.5}
    \vertexblue{-1,-1.75}\vertexblue{0,0}
    \draw[] (1,1.5) node[anchor=east] {$z$};
    \draw[] (1.5,2.25) node[anchor=west] {$w$};
    \draw[] (0,0) node[anchor=east] {$y$};
    \draw[] (0.5,0.5) node[anchor=west] {$u$};
    \draw[] (-1,-1.75) node[anchor=west] {$v$};
    \draw[] (-0.5,-1) node[anchor=east] {$x$};
    \end{tikzpicture}
    \end{tikzcd}\caption{Combinatorial type, partial triangulations of $\mathbb F_0$ and partial embedded tropical curve.}\label{fg:(202)F_0}
    \end{figure}
Comparing $x$-coordinates, resp. $y$-coordinates, gives equality $yz=2uz+uy$, resp. $ux=2xy+uy$.

Instead, the partial triangulation of the embedding in $\mathbb F_2,$ is depicted in Figure \ref{fg:(202)F_2}. All possible completions yield the relation $yz=uz.$

\begin{figure}[h]
\begin{tikzcd}
    \begin{tikzpicture}
        \foreach \j in {0,0.5,1,1.5,2,2.5,3}{
            \vertex{0,\j}
        }
        \foreach \j in {0,0.5,1,1.5,2}{
            \vertex{0.5,\j}
        }
        \foreach \j in {0,0.5,1}{
            \vertex{1,\j}
        }
        \vertex{1.5,0}
        \draw (0,0)--(1.5,0);
        \draw (0,3)--(1.5,0);
        \draw (0,3)--(0,0);
        \coordinate (1) at (1,0.5);
        \coordinate (2) at (0.5,0.5);
        \coordinate (3) at (0.5,1);
        \coordinate (4) at (0.5,1.5);
        \foreach \j in {1,2,3,4}{
            \vertexblue{\j}
        }
        \draw[blue] (1) node[anchor=north] {1};
        \draw[blue] (2) node[anchor=east] {2};
        \draw[blue] (3) node[anchor=east] {3};
        \draw[blue] (4) node[anchor=south east] {4};
        \draw[blue](2)--(3);
        \draw[blue] (1,1)--(0.5,0);
        \draw[blue] (0,1.5)--(1,1);
    \end{tikzpicture}&
    \begin{tikzpicture}
        \foreach \j in {0,0.5,1,1.5,2,2.5,3}{
            \vertex{0,\j}
        }
        \foreach \j in {0,0.5,1,1.5,2}{
            \vertex{0.5,\j}
        }
        \foreach \j in {0,0.5,1}{
            \vertex{1,\j}
        }
        \vertex{1.5,0}
        \draw (0,0)--(1.5,0);
        \draw (0,3)--(1.5,0);
        \draw (0,3)--(0,0);
       \coordinate (1) at (1,0.5);
        \coordinate (2) at (0.5,0.5);
        \coordinate (3) at (0.5,1);
        \coordinate (4) at (0.5,1.5);
        \foreach \j in {1,2,3,4}{
            \vertexblue{\j}
        }
        \draw[blue] (1) node[anchor=north] {1};
        \draw[blue] (2) node[anchor=east] {2};
        \draw[blue] (3) node[anchor=east] {3};
        \draw[blue] (4) node[anchor=south east] {4};
        \draw[blue](2)--(3);
        \draw[blue] (1,1)--(0.5,0);
        \draw[blue] (0,1.5)--(1,1);
        \draw(3)--(1,1);
        \draw (4)--(1,1);\draw(4)--(0,1.5);\draw(3)--(0,1.5);\draw(2)--(1,1);\draw(2)--(0.5,0);
        \draw(1)--(0.5,0);\draw(1)--(1,0);
    \draw(1,0.5)--(1.5,0); 
    \draw[] (1,1)--(1,0.5);
    
    \end{tikzpicture}&
    \begin{tikzpicture}[scale=0.75]
    \draw[](1.5,2.5+1.5)--(2,3+1.5);
    \draw[blue](1,1.5+1.5)--(1.5,2.5+1.5); \draw[](1.5,2.7+1.5)--(1.5,2.5+1.5);
    \draw[](0.5,1+1.5)--(1,1.5+1.5);
    \draw[](1,0+1.5)--(1,1.5+1.5);
    \draw[blue](1,0+1.5)--(0,0+1.5);
    \vertexblue{0,0+1.5}
    \draw[](1,0+1.5)--(1.5,-0.5+1.5);
    \vertexblue{1.5,-0.5+1.5}
\draw[](1.5,-0.5+1.5)--(1,-0.5+1.5);
    \draw[blue](1.5,-0.5+1.5)--(2.5,-1+1.5);
    \draw(2.5,-1+1.5)--(4,-1+1.5);\draw(4,-1+1.5)--(5,-0.5+1.5);
    \draw(2.5,-1+1.5)--(3,-1.5+1.5);
    \draw(3.5,-1.5+1.5)--(3,-1.5+1.5);
    \draw(3.5,-1.5+1.5)--(4,-1+1.5);
    \draw(3.5,-1.5+1.5)--(3.5,-2+1.5);
    \draw(3,-2+1.5)--(3,-1.5+1.5);
   
    \vertexblue{1,0+1.5}\vertexblue{1,1.5+1.5}
    \vertexblue{2.5,-1+1.5}
    \vertexblue{1.5,2.5+1.5}
    \draw[] (2.5,-1+1.5) node[anchor= north] {$v$};
    \draw[] (1.5,2.5+1.5) node[anchor=west] {$w$};
    \draw[] (0,0+1.5) node[anchor= east] {$u$};
    \draw[] (1,0+1.5) node[anchor= south west] {$y$};
    \draw[] (1,1.5+1.5) node[anchor= west] {$z$};
    \draw[] (1.5,-0.5+1.5) node[anchor= north] {$x$};
    \end{tikzpicture}
    \end{tikzcd}\caption{Partial triangulation of $\mathbb F_2$ and partial embedded tropical curve.}\label{fg:(202)F_2}
\end{figure}
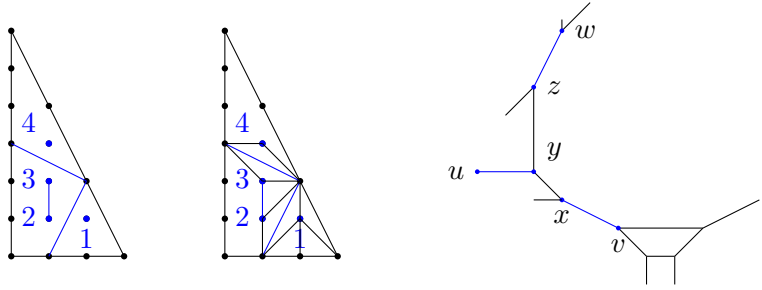

Let us also exhibit in Figure \ref{fg:(202)codim} a triangulation that shows that $xy+uy< ux<2xy+uy,$ but no other equality is satisfied. 

\begin{figure}[h]
\begin{tikzcd}
    \begin{tikzpicture}
        \foreach \j in {0,0.5,1,1.5,2,2.5,3}{
            \vertex{0,\j}
        }
        \foreach \j in {0,0.5,1,1.5,2}{
            \vertex{0.5,\j}
        }
        \foreach \j in {0,0.5,1}{
            \vertex{1,\j}
        }
        \vertex{1.5,0}
        \draw (0,0)--(1.5,0);
        \draw (0,3)--(1.5,0);
        \draw (0,3)--(0,0);
       \coordinate (1) at (1,0.5);
        \coordinate (2) at (0.5,0.5);
        \coordinate (3) at (0.5,1);
        \coordinate (4) at (0.5,1.5);
        \foreach \j in {1,2,3,4}{
            \vertexblue{\j}
        }
        \draw[blue] (1) node[anchor=north] {1};
        \draw[blue] (2) node[anchor=east] {2};
        \draw[blue] (3) node[anchor=east] {3};
        \draw[blue] (4) node[anchor=south east] {4};
        \draw[blue](2)--(3);
        \draw[blue] (1,1)--(0.5,0);
        \draw[blue] (0,1.5)--(1,1);
        \draw(3)--(1,1);
        \draw (4)--(1,1);\draw(4)--(0,1.5);\draw(3)--(0,1.5);\draw(2)--(1,1);\draw(2)--(0.5,0);
        \draw(1)--(0.5,0);\draw(1)--(1,0);
    \draw(1,0.5)--(1.5,0); 
    \draw[] (1,1)--(1,0.5);
    \draw[red](2)--(0,1);
    \draw[red](2)--(0,0.5);\draw[red](0.5,0)--(0,0.5);
    \end{tikzpicture}&
    \begin{tikzpicture}[scale=0.75]
    \draw[](1.5,2.5+1.5)--(2,3+1.5);
    \draw[blue](1,1.5+1.5)--(1.5,2.5+1.5); \draw[](1.5,2.7+1.5)--(1.5,2.5+1.5);
    \draw[](0.5,1+1.5)--(1,1.5+1.5);
    \draw[](1,0+1.5)--(1,1.5+1.5);
    \draw[blue](1,0+1.5)--(0,0+1.5);
    
    \draw[](1,0+1.5)--(1.5,-0.5+1.5);
    \draw[red](0,0+1.5)--(-0.25,-0.25+1.5);\vertexblue{0,0+1.5}
    \draw[red](-0.25,-0.5+1.5)--(-0.25,-0.25+1.5);\draw(-0.5,-0.25+1.5)--(-0.25,-0.25+1.5);
    \draw[](-0.25,-0.5+1.5)--(1.5,-0.5+1.5);
    \draw[red](-0.5,-0.75+1.5)--(-0.25,-0.5+1.5);
    \draw[](-0.5,-0.75+1.5)--(-0.5,-1+1.5);
    \draw[](-0.5,-0.75+1.5)--(-0.75,-0.75+1.5);
    \vertexblue{1.5,-0.5+1.5}
\draw[](1.5,-0.5+1.5)--(1,-0.5+1.5);
    \draw[blue](1.5,-0.5+1.5)--(2.5,-1+1.5);
    \draw(2.5,-1+1.5)--(4,-1+1.5);\draw(4,-1+1.5)--(5,-0.5+1.5);
    \draw(2.5,-1+1.5)--(3,-1.5+1.5);
    \draw(3.5,-1.5+1.5)--(3,-1.5+1.5);
    \draw(3.5,-1.5+1.5)--(4,-1+1.5);
    \draw(3.5,-1.5+1.5)--(3.5,-2+1.5);
    \draw(3,-2+1.5)--(3,-1.5+1.5);
    \vertexblue{1,0+1.5}\vertexblue{1,1.5+1.5}
    \vertexblue{2.5,-1+1.5}
    \vertexblue{1.5,2.5+1.5}
    \draw[] (2.5,-1+1.5) node[anchor= north] {$v$};
    \draw[] (1.5,2.5+1.5) node[anchor=west] {$w$};
    \draw[] (0,0+1.5) node[anchor= east] {$u$};
    \draw[] (1,0+1.5) node[anchor= south west] {$y$};
    \draw[] (1,1.5+1.5) node[anchor= west] {$z$};
    \draw[] (1.5,-0.5+1.5) node[anchor= north] {$x$};
    \end{tikzpicture}
    \end{tikzcd}\caption{A triangulation of $\mathbb F_2$ and the corresponding embedded tropical curve.}\label{fg:(202)codim}
\end{figure}
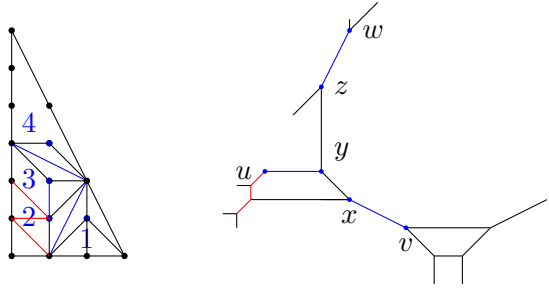

All the above prove the following Lemma.
\begin{lemma}
    Let $\Gamma$ be a tropical curve of genus $4$ of combinatorial type (202).
    \begin{itemize}
        \item If $\Gamma$ is realizable in $\mathbb F_0,$ then its edge lengths satisfy, up to symmetry:
        \begin{equation}
        \begin{cases}
            yz=2uz+uy;\\
            ux=2xy+uy.
        \end{cases}
        \end{equation}
        \item If $\Gamma$ is realizable in $\mathbb F_2$, then its edge lengths satisfy, up to symmetry:
        \begin{equation}
        \begin{cases}
            uz=yz.
        \end{cases}
        \end{equation}
    \end{itemize}
\end{lemma}

\subsubsection*{Type (212)}

Let us now consider the combinatorial type (212). The triangulation describing the embedding in $\mathbb F_0$ is represented in Figure \ref{fg:(212)F_0}.
\begin{figure}[h]
\begin{tikzcd}
      \begin{tikzpicture}
        \coordinate (1) at (0,0);
        \coordinate (2) at (1,0);
        \coordinate (3) at (0,1);
        \coordinate (4) at (1,1);
        \coordinate (5) at (0,1.5);
        \coordinate (6) at (1,1.5);
        \draw (4)--(2);
        \draw (1)--(3);
        \draw (3)--(4);
        \draw (3)--(5);
        \draw (4)--(6);
        \draw (0,1.75) circle (0.25);
        \draw (1,1.75) circle (0.25);
        \foreach \i in {1,2,3,4,5,6}{
            \vertex{\i}
        }
        \draw(1)[] to [out=30, in=150] (2);
        \draw(1)[] to [out=330, in=210] (2);
        \draw(1) node[anchor=north] {$x$};
        \draw(2) node[anchor=north] {$y$};
        \draw(3) node[anchor=east] {$z$};
        \draw(4) node[anchor=west] {$u$};
        \draw(5) node[anchor= east] {$v$};
        \draw(6) node[anchor= west] {$w$};
        \draw[blue](0,1.6) node {1};
        \draw[blue] (1,1.6) node {4};
        \draw[blue](0.5,0.4) node {2};
        \draw[blue](0.5,-0.1) node {3};
    \end{tikzpicture}&
    \begin{tikzpicture}
        \foreach \i in {0,0.5,1,1.5}{
            \foreach \j in {0,0.5,1,1.5}{
            \vertex{\i,\j}
            }
    }
    \coordinate (3) at (0.5,0.5);
    \coordinate (1) at (0.5,1);
    \coordinate (4) at (1,0.5);
    \coordinate (2) at (1,1);
        \foreach \j in {1,2,3,4}{
            \vertexblue{\j}
        }
    \draw[blue] (1) node[anchor=south east] {1};\draw[blue] (2) node[anchor=north west] {2};
    \draw[blue] (3) node[anchor=north east] {3};\draw[blue] (4) node[anchor=north west] {4};
    \draw(0,0)--(1.5,0);
    \draw(0,0)--(0,1.5);
    \draw(1.5,1.5)--(0,1.5);
    \draw(1.5,1.5)--(1.5,0);
    \draw[blue](2)--(3);
    \draw[blue](0.5,0)--(1.5,1.5);
    \draw[blue](0,0.5)--(1.5,1.5);
    \draw[blue](0.5,0)--(2);
    \draw[blue](0,0.5)--(2);
    \draw (2)--(1.5,1.5);   
    \draw(0,0.5)--(3);
    \draw(0.5,0)--(3);
    \draw(0,0.5)--(0.5,0);

    \draw(0.5,0)--(4);\draw(1.5,1.5)--(4);\draw(1.5,1)--(4);
    \draw(0,0.5)--(1);\draw(1.5,1.5)--(1);\draw(1,1.5)--(1);
    \end{tikzpicture}
    &
    \begin{tikzpicture}[scale=0.75]
    \draw(-0.5,0)--(0,0);
    \draw(0,-0.5)--(0,0);
    \draw(0,0)--(0.5,0.5);
    \draw(1,0.5)--(0.5,0.5);
    \draw(0.5,1)--(0.5,0.5);\vertexblue{0.5,1}\vertexblue{1,0.5}
    \draw[blue](0.5,1)--(1,0.5);
    \draw[blue](0.5,1)--(0,2);\vertexblue{0,2}
    \draw[blue](1,0.5)--(2,0);\vertexblue{2,0}
    \draw(2,0)--(0,2);
    \draw(0,2)--(-1,3.5);
    \draw(2,0)--(3.5,-1);\vertexblue{-1,3.5}\vertexblue{3.5,-1}
    \draw(-1,3.5)--(-1.5,4.5);\draw(-1.5,5)--(-1.5,4.5);\draw(-2,5)--(-1.5,4.5);
    \draw(3.5,-1)--(4.5,-1.5);\draw(5,-1.5)--(4.5,-1.5);\draw(5,-2)--(4.5,-1.5);
    \draw(-1,3.5)--(-1.5,4);
    \draw[] (0,2) node[anchor=north] {$z$};
    \draw[] (2,0) node[anchor=north] {$u$};
    \draw[] (1,0.5) node[anchor=north] {$y$};
    \draw[] (-1,3.5) node[anchor=west] {$v$};
    \draw[] (3.5,-1) node[anchor=west] {$w$};
    \draw[] (0.5,1) node[anchor=east] {$x$};
    \end{tikzpicture}
    \end{tikzcd}\caption{Combinatorial type, partial triangulations of $\mathbb F_0$ and partial embedded tropical curve.}\label{fg:(212)F_0}
    \end{figure}
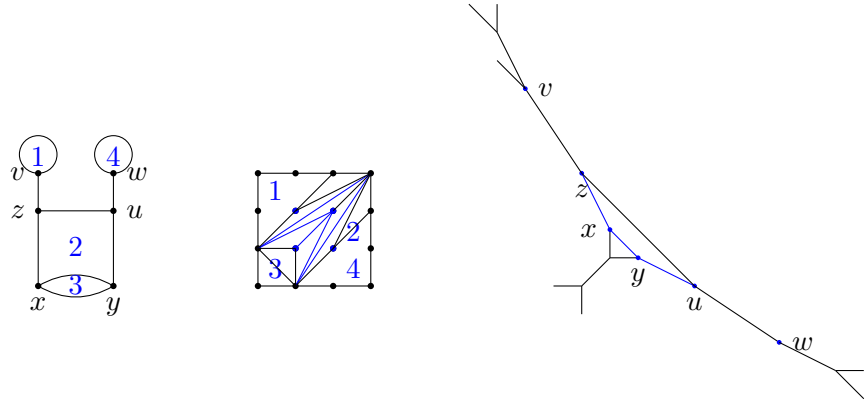

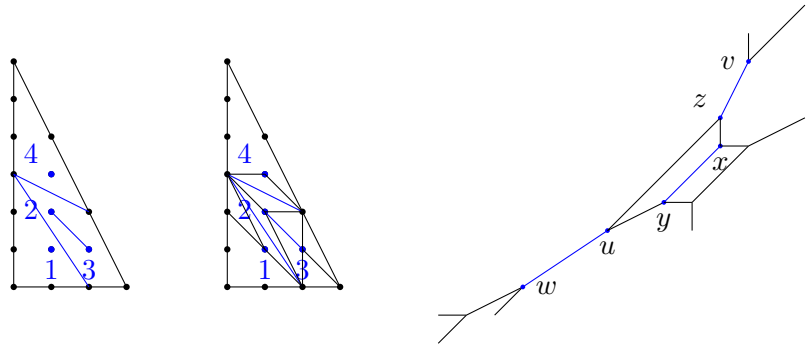
\begin{figure}[h]
\begin{tikzcd}
    \begin{tikzpicture}
        \foreach \j in {0,0.5,1,1.5,2,2.5,3}{
            \vertex{0,\j}
        }
        \foreach \j in {0,0.5,1,1.5,2}{
            \vertex{0.5,\j}
        }
        \foreach \j in {0,0.5,1}{
            \vertex{1,\j}
        }
        \vertex{1.5,0}
        \draw (0,0)--(1.5,0);
        \draw (0,3)--(1.5,0);
        \draw (0,3)--(0,0);
        \coordinate (3) at (1,0.5);
        \coordinate (1) at (0.5,0.5);
        \coordinate (2) at (0.5,1);
        \coordinate (4) at (0.5,1.5);
        \foreach \j in {1,2,3,4}{
            \vertexblue{\j}
        }
        \draw[blue] (1) node[anchor=north] {1};
        \draw[blue] (2) node[anchor=east] {2};
        \draw[blue] (3) node[anchor=north] {3};
        \draw[blue] (4) node[anchor=south east] {4};
        \draw[blue](2)--(3);
        \draw[blue] (0,1.5)--(1,0);
        \draw[blue] (0,1.5)--(1,1);
    \end{tikzpicture}&
    \begin{tikzpicture}
        \foreach \j in {0,0.5,1,1.5,2,2.5,3}{
            \vertex{0,\j}
        }
        \foreach \j in {0,0.5,1,1.5,2}{
            \vertex{0.5,\j}
        }
        \foreach \j in {0,0.5,1}{
            \vertex{1,\j}
        }
        \vertex{1.5,0}
        \draw (0,0)--(1.5,0);
        \draw (0,3)--(1.5,0);
        \draw (0,3)--(0,0);
       \coordinate (3) at (1,0.5);
        \coordinate (1) at (0.5,0.5);
        \coordinate (2) at (0.5,1);
        \coordinate (4) at (0.5,1.5);
        \foreach \j in {1,2,3,4}{
            \vertexblue{\j}
        }
        \draw[blue] (1) node[anchor=north] {1};
        \draw[blue] (2) node[anchor=east] {2};
        \draw[blue] (3) node[anchor=north] {3};
        \draw[blue] (4) node[anchor=south east] {4};
        \draw[blue](2)--(3);
        \draw[blue] (0,1.5)--(1,0);
        \draw[blue] (0,1.5)--(1,1);
        \draw(4)--(1,1);\draw(4)--(0,1.5);
        \draw(2)--(0,1.5);\draw(2)--(1,1); \draw(2)--(1,0);
        \draw(3)--(1,0);
        \draw(1)--(1,0);\draw(1)--(0,1.5);\draw(1)--(0,1);
    \draw(1,0.5)--(1.5,0); 
    \draw[] (1,1)--(1,0.5); 
    \end{tikzpicture}&
    \begin{tikzpicture}[scale=0.75]
    \draw[blue](0,0)--(1.5,1);
    \draw(-0.5,-0.5)--(0,0);
    \draw(-1,-0.5)--(0,0);
    \draw(-1,-0.5)--(-1.5,-1);\draw(-1,-0.5)--(-1.5,-0.5);\vertexblue{0,0}
    \vertexblue{1.5,1}
    \draw(1.5,1)--(2.5,1.5);\vertexblue{2.5,1.5}
    \draw(1.5,1)--(3.5,3);\vertexblue{3.5,3}
    \draw[blue](2.5,1.5)--(3.5,2.5);\vertexblue{3.5,2.5}
    \draw(3.5,2.5)--(3.5,3);
    \draw(3.5,2.5)--(4,2.5);\draw(5,3)--(4,2.5);
    \draw[](2.5,1.5)--(3,1.5);
    \draw[](3,1.5)--(4,2.5);
    \draw[](3,1.5)--(3,1);
    \draw[blue](3.5,3)--(4,4);
    \draw(4,4)--(5,5);\draw(4,4)--(4,4.5);
    \vertexblue{4,4}
    \draw (1.5,1) node[anchor= north] {$u$};
    \draw (0,0) node[anchor=west] {$w$};
    \draw[] (2.5,1.5) node[anchor= north] {$y$};
    \draw[] (3.5,3) node[anchor= south east] {$z$};
    \draw[] (4,4) node[anchor= east] {$v$};
    \draw[] (3.5,2.5) node[anchor= north] {$x$};
    \end{tikzpicture}
    \end{tikzcd}\caption{Partial triangulation of $\mathbb F_2$ and partial embedded tropical curve.}\label{fg:(212)F_2}
\end{figure}

Comparing again $x$ or $y$-coordinates of the points proves the following.
\begin{lemma}
    Let $\Gamma$ be a tropical curve of genus $4$ of combinatorial type (212).
    \begin{itemize}
        \item If $\Gamma$ is realizable in $\mathbb F_0,$ then its edge lengths satisfy, up to symmetry:
        \begin{equation}
        \begin{cases}
        xz+xy+2uy=uz;\\ 
        2xz+xy+uy=uz; 
        \end{cases}\Rightarrow \begin{cases}
        xy+3uy=uz;\\
        uy=xz. 
        \end{cases}
        \end{equation}
        \item If $\Gamma$ is realizable in $\mathbb F_2$, then its edge lengths satisfy, up to symmetry:
        \begin{equation}
        \begin{cases}
        2uy+xy=uz;\\
        uy+xy+xz=uz;\\
        \end{cases}\Rightarrow 
        \begin{cases}
        xy+2uy=uz;\\
        uy=xz. 
        \end{cases}
        \end{equation}
    \end{itemize}
\end{lemma}

\bibliographystyle{plain}
\bibliography{TrigonalAndEmbeddedTropicalCurves.bib}

\end{document}